\documentclass[final]{siamltex}
\usepackage{amsfonts}
\usepackage{amssymb}
\newcommand{\br}{\mathbb R}
\newcommand{\bn}{\mathbb N}

\def\e{{\varepsilon}}
\def\Om{\Omega}
\def\om{\omega}
\def\aal{\alpha}
\def\br{\mathbb R}
\def\bn{\mathbb N}
\def\tow{\rightharpoonup}
\newtheorem{remark}{Remark}[section]
\newcommand{\toH}{\twoheadrightarrow}

\usepackage{pstcol}
\usepackage{pstricks}

\title{Asymptotic analysis of a micropolar fluid flow in thin domain with
a free and  rough  boundary}

\author{{Mahdi Boukrouche  \and Laetitia Paoli}
\thanks{PRES Lyon  University,  UJM F-42023 Saint-Etienne, CNRS UMR 5208
 Institut Camille Jordan, 23 Docteur Paul
Michelon 42023 Saint-Etienne Cedex 2, France.
         Fax: +33 4 77 48 51 53, Phone +33 4 77 48 15 00,
Mahdi.Boukrouche@univ-st-etienne.fr,
 laetitia.paoli@univ-st-etienne.fr}  }

\begin{document}

\maketitle

\begin{abstract}
 Motivated by lubrication problems, we consider a micropolar
fluid flow in a 2D domain with a rough and free boundary.
 We assume that the thickness and the roughness are both
of order $0<\e <<1$. We prove the existence and uniqueness
of a solution of this problem for any value of $\e$ and
 we establish some a priori estimates. Then we use the
 two-scale convergence technique to derive the limit
problem when $\e$ tends to zero.
 Moreover we show that the limit velocity and micro-rotation
fields are uniquely determined via auxiliary well-posed problems and
 the limit pressure is given as the unique solution of a Reynolds equation.
\end{abstract}



\pagestyle{myheadings}
\thispagestyle{plain}
\markboth{Asymptotic analysis of a micropolar fluid flow}{Mahdi Boukrouche  and Laetitia Paoli}

 \renewcommand{\theequation}{1.\arabic{equation}}
 \setcounter{equation}{0}
\section{Introduction}\label{introduction}

The theory of micropolar fluids, was introduced and
formulated by A.C. Eringen
in \cite{erin}. It aims to describe fluids containing suspensions of rigid
particles in a viscous medium. Such fluids exhibit micro-rotational effects and
micro-rotational inertia.
 Therefore they  can support couple stress and
distributed body couples. They form  a class of fluids with nonsymmetric
stress tensor for which the classical  Navier-Stokes theory is  inadequate
since it does not take into account the effects of the micro-rotation.
Experimental studies have showned that the micropolar model better represents
the behavior of numerous fluids
 such as polymeric fluids, liquid crystals,
paints, animal blood, colloidal fluids, ferro-liquids, etc., especially when the
characteristic dimension of the flow becomes small (see for instance
\cite{popel}).
Extensive reviews of the  theory and its  applications can be found in
\cite{ariman1, ariman2} or in the books \cite{erin1999} and \cite{luk1} and also
in more
recent articles (see for example \cite{Bakr2011, m2as05, abdullallh2010, geng2011}).

Motivated by lubrication theory where the domain of flow is
usually very thin and the roughness of the boundary strongly affects the flow
(\cite{bc}),  we consider the
motion of the
micropolar fluid  described by the equilibrium of momentum, mass and moment of
momentum. More precisely,  the velocity field of the fluid  $u^{\e} = (u^{\e}_1,
u^{\e}_2 )$, the pressure $p^{\e}$ and the angular velocity of the
micro-rotations of the particles $\om^{\e}$ satisfy the system

\begin{eqnarray}
    && u^{\e}_{t} - (\nu + \nu_r) \Delta u^{\e}
+ (u^{\e}\cdot\nabla) u^{\e}
    + \nabla p^{\e} = 2 \nu_r \; {\rm rot \,} \, \om^{\e} + f^{\e},
\label{2.1} \\
    && \qquad {\rm div } u^{\e} = 0,
\label{2.2} \\
    && \om^{\e}_{t} - \aal \Delta \om^{\e}   + (u^{\e}\cdot\nabla) \; \om^{\e} +
4 \nu_{r} \om^{\e}
    = 2 \nu_{r} \; {\rm rot \, } u^{\e} + g^{\e},\label{2.3}
\end{eqnarray}
in  the space-time domain $(0 , T)\times\Om^{\e}$ with
$$
\Omega^{\e} = \{z=(z_{1} , z_{2}) \in {\br}^{2}, \quad 0 < z_{1} < L, \quad
            0 < z_{2} <  \e h^{\e}(z_{1})\}, \quad h^{\e}(z_{1}) = h(z_{1} , {z_{1} \over  \e})$$

where
 $h$  is a given smooth function, $f^{\e}$ and  $g^{\e}$ are given external
forces and moments,
 $\nu$ is the usual Newtonian viscosity, $\nu_r$ and $\aal$
are   the micro-rotation viscosities,  which  are assumed to be positive
constants (\cite{erin}).

The choice of the domain $\Omega^{\e}$ comes from
one of the important fields of the theory of lubrication given  by
the study of self-lubricating bearings. These bearings are widely used in
mechanical and electromechanical industry, to lubricate the main axis of
rotation of a device, in order to prevent its endomagement.

Such  bearings consist in an inner cylinder and  a outer cylinder, and
along
a circumferencial  section, one can see two non-concentric discs.
The radii of the two cylinders are much smaller than their lengh and the gap
between the two cylinders,
 which is fullfilled with a lubricant,  is much smaller than their radii
(\cite{Bayada-Chambat}). By assuming that the external fields and the flow do
not depend on the coordinate along the longitudinal axis of the bearing, one can
represent the fluid domain by $\Om^{\e}$ which is a 2D view of a cross section
after a radial cut of the two circumferences. 
The boundary of $\Om^{\e}$ is
$\partial\Om^{\e}=\bar{\Gamma}_{0}\cup\bar{\Gamma}_{L}^{\e}\cup\bar{\Gamma}_{1}^
{\e}$,

where $\Gamma_{0}=\{z\in \partial\Om^{\e} :  z_{2}=0\}$ is the bottom,
$\Gamma_{1}^{\e}=\{z\in \partial\Om^{\e} :  z_{2}= \e h^{\e}(z_{1})\}$ is the
upper strongly oscillating part,
 and $\Gamma_{L}^{\e}$ is the lateral part  of the boundary. The surface of the inner cylinder, which corresponds to $\Gamma_0$,  is in contact with the rotating axis of
the device while the surface of the outer cylinder, which corresponds to $\Gamma_{1}^{\e}$, remains still.

Hence the boundary and initial conditions are given as follows
\begin{eqnarray}
  \om^{\e},\quad   u^{\e} , \quad  p^{\e}\quad \mbox{ are L-periodic with
respect to}~z_{1}
\label{eqn:er2.4a}\\
 u^{\e}= U_0e_{1}=(U_0 , 0) ,  \quad   \om^{\e} = W_0\quad \mbox{on} \quad (0 ,
T)\times\Gamma_{0}
\label{eqn:er2.5}\\
  \om^{\e}=0, \quad  u^{\e}\cdot n=0, \qquad    {\partial u^{\e}\over \partial n} \cdot\tau =0
\qquad\mbox{on} \quad (0 , T)\times\Gamma_{1}^{\e}
 \label{eqn:er2.4}\\
  u^{\e} (0 , z) = u^{\e}_{0}(z) ,  \quad  \om^{\e} (0 , z) =
\om^{\e}_{0}(z)\quad
  {\rm for} \quad z\in\Om^{\e}
 \label{eqn:er2.5a}
 \end{eqnarray}

where $\tau$ and $n$ are
respectively  the tangent and normal  unit vectors
to the boundary of the domain $\Om^{\e}$.
 Let us observe that (\ref{eqn:er2.5}) represents  non-homogeneous Dirichlet
conditions along $\Gamma_0$, which means adherence of the fluid to the boundary of the rotating inner cylinder, so
$U_0$ and $W_0$ are two given functions of the time variable only.
The  second condition    in (\ref{eqn:er2.4}) is the
nonpenetration boundary condition, while the last one  is
non-standard, and it means
that  the tangential component of the flux on $\Gamma^{\e}_{1}$ is equal
zero      (\cite{D.L}).

The choice of the particular scaling, with a roughness in inverse proportion to the thickness of the domain, is quite classical in lubrication theory. In \cite{bcc} and in \cite{bc} a Stokes flow is considered with adhering boundary conditions and Tresca boundary conditions at the fluid solid interface respectively. For other related works see also \cite{bucj1, bucj2} or \cite{bacj} for instance.

We  prove the existence and uniqueness of a weak solution $(u^{\e} ,
\om^{\e}, p^{\e})$ in adequate functional framework.
Then we will  establish some a
priori estimates for the velocity, micro-rotation and pressure fields,
independently of $\e$, and finally we will derive and study the limit problem when $\e$ tends to zero.

The paper is organized as follows.  In Section \ref{secSP} we give the
variational formulation.  Then, using an idea of J.L. Lions
(\cite{JLLions78}), we consider the  divergence free condition
(\ref{2.2}) as a constraint, which can be  penalized, and   we prove in Theorem \ref{th2.1}
the existence and uniqueness of a weak solution $(u^{\e} , \om^{\e}, p^{\e})$  for any value of $\e$.
Let us emphasize that our proof ensures that the pressure (unique up to an additional  function of
time) belong to  $H^{-1}(0 , T , L^{2}_{0}(\Om^{\e}))$.  This result  is more
suitable
for the next parts of our study, than  $W^{-1, \infty}(0 , T ,
L^{2}_{0}(\Om^{\e}))$ obtained by
 J. Simon
\cite{Simon1999} (see also Theorem 2.1 in 
\cite{galdi}).

In Section \ref{uniformEst},
 we establish some a priori estimates for the velocity and micro-rotation fields
 in Proposition
 \ref{pro1} and for the pressure in Proposition  \ref{prop2}.
In Section \ref{twoscaleconv}, since we deal with an evolution problem,
we extend first  the classical   two-scale convergence results
(\cite{allaire, nguetseng}) to a time-dependent setting  and we use this
technique  to
prove  some convergence properties for  the velocity in
Proposition \ref{prop4.1},    the micro-rotation in
Proposition  \ref{prop4.2}, and  the pressure in  Proposition \ref{prop4.3}.

Then, in Section \ref{limitprobl}
we  derive  the limit problem when $\e$ tends to zero in Theorem \ref{limit_pb}.
We notice that the trilinear and rotational terms, as well as
  the time derivative do not contribute  when we pass to the  limit.
However the time variable remains in the limit problem  as
a parameter.  We note also  that the limit problem can be easily
decoupled:  we
obtain a variational equality involving only the limit velocity and the limit
pressure
and another variational equality involving  the
limit  micro-rotation. However, the micropolar nature of the fluid still appears
in the limit problem for the velocity and pressure since we keep the viscosity
$\nu + \nu_r$.
Moreover we show in Proposition \ref{prop5.1} that the limit velocity and
micro-rotation fields are uniquely determined via auxiliary well-posed
problems. In Proposition \ref{prop5.2}, we prove that the limit pressure
is given as the unique
solution of a Reynolds equation. Finally in Section \ref{open-Pbs}
we propose a generalization to the case where both the upper and the lower boundary of the fluid domain are oscillating.
 \renewcommand{\theequation}{2.\arabic{equation}}
 \setcounter{equation}{0}
\section{Existence and uniqueness results} \label{secSP}
We assume that
  \begin{eqnarray} \label{Lper1}
{L\over \e}\in \bn, \quad  h :  (z_{1} , \eta_1) \mapsto h(z_{1} , \eta_1)
 \mbox{ is L-periodic in  } z_{1}
           \mbox{ and 1-periodic  in  }   \eta_1,
  \end{eqnarray}
so $h$ is  L-periodic in $z_{1}$. We  assume also  that
  \begin{eqnarray} \label{Lper2}
 h \in {\cal C}^{\infty}([0 , L]\times \br), \quad {\partial h\over
\partial \eta_1}  \quad \mbox{is 1-periodic in  }  \eta_1,
 \end{eqnarray}
 and  there exist $h_{m}$ and $h_{M}$ such that
  \begin{eqnarray}
0< h_{m}=\min_{[0 , L]\times [0 , 1]}h(z_{1} , \eta_1),
\quad \mbox{and} \quad  h_{M}=\max_{[0 , L]\times [0 , 1]}h(z_{1} ,
\eta_1).
 \end{eqnarray}

\begin{lemma}\label{lUW}
Let  the functions ${\cal U}$, ${\cal W}$ be
in  ${\cal D}(-\infty , h_{m})$, and
$U_{0}$,  $W_{0}$ be  in $H^{1}(0 , T)$, with
${\cal U}(0)=1$, ${\cal W}(0)=1$.
We set
$$
 U^{\e}(t ,z_{2})= {\cal U}^{\e}(z_{2})U_{0}(t)
 = {\cal U}({z_{2}\over \e})U_{0}(t) ,
\qquad
 W^{\e}(t ,z_{2})= {\cal W}^{\e}(z_{2})W_{0}(t)
= {\cal W}({z_{2}\over \e})W_{0}(t).$$
Then we have for all $ (t , z_1 )\in  (0 , T)\times(0,L)$
\begin{equation} \label{eqn:er2.6.1}
U^{\e} (t , 0)=U_0 (t), \, U^{\e} (t , \e h^{\e} (z_1 ))=0,
\, {\partial U^{\e}\over \partial z_{2}}(t , \e h^{\e} (z_1 ))=0,
\end{equation}
\begin{equation} \label{eqn:er2.7.1}
W^{\e} (t , 0)=W_0 (t), \quad W^{\e} (t , \e h^{\e} (z_1 ))=0.
\end{equation}
\end{lemma}
\begin{proof}
Indeed,  $U^{\e}(t , 0) =  {\cal  U}(0)U_{0}(t)=U_{0}(t)$,
$U^{\e}(t , \e h^{\e}(z_{1})) = {\cal  U}(h(z_{1} ,{z_{1}\over \e}
))U_{0}(t) =0$
and
$${\partial U^{\e}\over \partial z_{2}}(t , \e h^{\e}(z_{1}))
 = {1\over \e} {\cal  U}' (h^{\e}(z_{1}))U_{0}(t)
 = {1\over \e} {\cal  U}' (h(z_{1} ,{z_{1}\over \e}))U_{0}(t)
=0, $$
thus (\ref{eqn:er2.6.1}) follows. The proof is valid also for
(\ref{eqn:er2.7.1}).
\end{proof}

We can now set
\begin{equation} \label{eqn:er2.6}
u^{\e} (t , z_{1} , z_{2}) = U^{\e} (t ,  z_{2})e_1 + v^{\e} (t , z_{1} , z_{2})
\end{equation}
\begin{equation} \label{eqn:er2.7}
\om^{\e} (t , z_{1} , z_{2}) = W^{\e} (t , z_{2} ) +  Z^{\e} (t , z_{1} , z_{2})
\end{equation}
with $U^{\e}$ and $W^{\e}$ satisfying (\ref{eqn:er2.6.1})
(\ref{eqn:er2.7.1}). Moreover
$$ {\partial u^{\e}_{i}\over \partial z_{j}}= {\partial v^{\e}_{i}\over \partial
z_{j}} +
  {\partial \over \partial z_{j}}(U^{\e} (\cdot ,  z_{2})e_1 )=
\left\{\begin{array}{ll}
       {\partial v^{\e}_{i}\over \partial z_{j}}    \quad {\rm if} \quad j=1,\\
       {\partial v^{\e}_{i}\over  \partial z_{j}} +  {\partial U^{\e}\over
\partial z_{2}}(\cdot , z_{2})e_1 \quad {\rm if} \quad j=2
       \end{array}\right.
$$
and from (\ref{eqn:er2.6.1})
$  {\partial U^{\e}\over\partial z_{2}}(t , z_{2})={\partial U^{\e}\over
\partial z_{2}}(t , \e h^{\e} (z_1))=0
\quad\mbox{ for } (t , z_{2})\in (0 , T)\times\Gamma_{1}^{\e}$
 so
\begin{eqnarray}\label{e2.15}
 {\partial u^{\e}_{i}\over \partial z_{j}}=
{\partial v^{\e}_{i}\over \partial z_{j}}\quad\mbox{ for} \quad j= 1, 2
\quad  \mbox{ on } \quad (0 , T)\times\Gamma_{1}^{\e}.
\end{eqnarray}
Recall also that
 $${\rm rot \, } u^{\e} = \frac{\partial u^{\e}_2}{\partial z_1}
    - \frac{\partial u^{\e}_1}{\partial z_2}, \;\qquad \;
    {\rm rot \, } \, \om^{\e} = \bigg(\frac{\partial \om^{\e}}{\partial z_2},
    - \frac{\partial \om^{\e}}{\partial z_1}\bigg).$$
Then the problem (\ref{2.1})-(\ref{eqn:er2.5a}) becomes
\begin{eqnarray}\label{eqn:er2.8}
 v^{\e}_{t} -(\nu +\nu_{r})\Delta v^{\e} + (v^{\e} \cdot \nabla) v^{\e}
+ U^{\e}{\partial v^{\e}\over \partial z_{1}}
+ (v^{\e})_2  {\partial U^{\e}\over \partial z_{2}} e_1 +  \nabla p^{\e}
=
2 {\nu}_{r} {\rm  rot \,}Z^{\e}
\\ 
+ (\nu + \nu_r) {\partial^{2} U^{\e}\over \partial z_{2}^{2}}e_1
+ 2\nu_r {\partial W^{\e}\over \partial z_{2}}e_1
-{\partial U^{\e}\over \partial t}e_1  + f^{\e}  \quad \mbox{in } (0 ,
T)\times\Om^{\e},\nonumber
\end{eqnarray}
\begin{eqnarray}
{\rm div \,} v^{\e} = 0    \quad \mbox{in }  \Om^{\e}, \quad t\in  (0 , T),
\label{eqn:er2.9}
\end{eqnarray}
\begin{eqnarray}
 Z^{\e}_{t} - \aal\Delta Z^{\e}+ (v^{\e} \cdot \nabla)Z^{\e} + 4 \nu_{r}Z^{\e}
+  U^{\e}{\partial Z^{\e}\over \partial z_1 } +
(v^{\e})_2  {\partial W^{\e} \over \partial z_{2}}
=
2\nu_{r}{\rm rot \,} v^{\e}
 \nonumber\\ + \aal {\partial^{2}
W^{\e}\over\partial
z_{2}^{2}}
 - 2\nu_r {\partial U^{\e}\over\partial z_{2}} -
4\nu_rW^{\e}  -{\partial W^{\e}\over \partial t}  + g^{\e}\quad \mbox{in } (0 ,
T)\times\Om^{\e}, \label{eqn:er2.10}
\end{eqnarray}
\begin{equation}\label{eqn:er2.8a}
v^{\e}, Z^{\e} \ \mbox{\rm  and }  p^{\e}\  \mbox{\rm  L-periodic in } z_{1},
 \end{equation}
  \begin{equation} \label{eqn:er2.11}
 Z^{\e} = 0,  \quad     v^{\e} =0     \quad \mbox{\rm on} \quad (0 ,
T)\times\Gamma_{0},
 \end{equation}
  \begin{equation} \label{eqn:er2.11b}
 Z^{\e}=0, \quad  v^{\e}\cdot n=0, \qquad {\partial v^{\e}\over \partial n}
\cdot\tau =0
\quad \quad \mbox{\rm on}  \quad (0 , T)\times\Gamma^{\e}_{1},
\end{equation}
 \begin{eqnarray}
v^{\e}(0, z) =  v^{\e}_{0}(z) =  u^{\e}_{0}(z) - U^{\e}(0 , z_2 )e_1  \quad
\mbox{\rm  in } \Om^{\e},\label{eqn:er2.11a} \\
\quad Z^{\e}(0 , z) =
Z^{\e}_{0}(z) = \om^{\e}_{0}(z) - W^{\e}(0 , z_2 )  \quad \mbox{\rm  in }
\Om^{\e}, \label{eqn:er2.11ab}
\end{eqnarray}
where we have denoted by $(v^{\e})_2$  the second component of $v^{\e}$.

To define the  weak formulation of the above problem
(\ref{eqn:er2.8})- (\ref{eqn:er2.11ab}),
we recall that $\Gamma^{\e}_{1}$ is
defined by the equation $z_{2}= \e h^{\e}(z_{1})$,
thus the unit outward normal vector to $\Gamma^{\e}_{1}$
 is given by
$$n=\frac{1}{\sqrt{ 1 + (\e (h^{\e})'(z_{1}))^2}} (-\e
(h^{\e})'(z_{1}) , 1)$$
 and
$v\cdot n= 0$ becomes
$-\e (h^{\e})'(z_{1}) v_{1} + v_{2}  = 0$ on   $\Gamma^{\e}_{1}$.
We consider now the following functional framework
$$\tilde{V^{\e}}= \{ v \in {\cal C}^{\infty}(\overline{\Om^{\e}})^{2}:
 \,\, v \,\,\mbox{\rm is  L-periodic  in }\,\, z_1,\,\,\,  v_{|_{\Gamma_0}} =0,
\,     -\e (h^{\e})'(z_{1})  v_{1} + v_{2}  =0
\mbox{ on } \Gamma^{\e}_{1}\}$$
$$\tilde{H^{1}}^{\e}= \{Z \in {\cal C}^{\infty}(\overline{ \Om^{\e} }): \,\,\,Z
    \,\,\,\mbox{\rm is  L-periodic  in }\,\,\, z_1,\,\,\,Z=0 \,\,\,{\rm
on }\,\,\,
    \Gamma_0\cup\Gamma^{\e}_1\}$$
$$V^{\e} = {\rm closure \,\,  of\,\, } \tilde{V^{\e}} \,\,
 {\rm  in }  \,\, H^{1}(\Om^{\e})\times  H^{1}(\Om^{\e}),\qquad
V^{\e}_{div} = \{v\in V^{\e} \, : \,
   {\rm div \,} v = 0,  \,\, \mbox{\rm in}  \,\, \Om^{\e}\}$$
$$H^{\e} = {\rm  closure \,\,  of   \, } \tilde{V^{\e}}  \,\, {\rm in }\,\,
  L^{2}(\Om^{\e}) \times  L^{2}(\Om^{\e}), \qquad
{H^{1,}}^{\e} = {\rm closure \,\,  of\,\, } \tilde{H^{1}}^{\e} \,\, {\rm  in }
 \,\, H^{1}(\Om^{\e}),
 $$
$${H^{0,}}^{\e} = {\rm  closure \,\,  of   \, } \tilde{H^{1}}^{\e}\,\,
{\rm in }\,\,  L^{2}(\Om^{\e}),
\quad L^{2}_{0}(\Om^{\e})=\{q\in L^{2}(\Om^{\e}) : \, \int_{\Om^{\e}} q(z)
dz=0\}.  $$
We endowed these functional spaces with the inner products and norms defined by
$$[\bar{v} ,\Theta] = (v , \varphi) + (Z , \psi)  \mbox{ in }
 H^{\e}\times {H^{0,}}^{\e}  \mbox{ with the norm }
[\bar{v}] = [\bar{v}, \bar{v}]^{1\over 2} $$
$$ [[\bar{v} , \Theta]] = (\nabla v , \nabla \varphi) + (\nabla Z ,
\nabla \psi)
\mbox{ in }  V^{\e}\times {H^{1,}}^{\e}  \mbox{ with the norm }
[[\bar{v}]] = [[\bar{v}, \bar{v}]]^{1\over 2}
$$
for any pairs of functions   $\bar{v}= (v , Z)$ and $\Theta =
(\varphi , \psi)$.
The  weak formulation of the
problem   (\ref{eqn:er2.8})- (\ref{eqn:er2.11ab}) is given by\newline
\noindent {\bf Problem $(P^{\e})$}
Find
    $$\bar{v}^{\e} = (v^{\e} , Z^{\e})\in
\Bigl( {\cal C}([0 , T]; H^{\e})\cap L^{2}(0 , T; V^{\e}_{div}) \Bigr) \times
\Bigr( {\cal C}([0 ,
      T]; {H^{0,}}^{\e})\cap L^{2}(0 , T;  {H^{1,}}^{\e}) \Bigr) $$
 and $p^{\e}\in  H^{-1}(0 , T ; L^{2}_{0}(\Om^{\e}))$,
such that
\begin{eqnarray} \label{eqn:er2.14}
[{\partial\bar{v}^{\e} \over \partial t}(t) , \Theta^{\e}] +
a(\bar{v}^{\e}(t) , \Theta^{\e})
+B(\bar{v}^{\e}(t) ,\bar{v}^{\e}(t) , \Theta^{\e}) + {\cal
R}(\bar{v}^{\e}(t) ,
\Theta^{\e})=\nonumber\\
= (p^{\e}(t) \, , \, {\rm div }\varphi^{\e})
+ ({\cal F}(\bar{v}^{\e}(t))  , \Theta^{\e}) \qquad
\forall\Theta^{\e} = (\varphi^{\e} , \psi^{\e})  \in V^{\e}\times
{H^{1,}}^{\e},
\end{eqnarray}
 with  the initial condition
\begin{equation} \label{eqn:er2.14a}
  \bar{v}^{\e}(z , 0) = \bar{v^{\e}_{0}} (z) = ( v^{\e}_{0}(z) \, ,
\,Z^{\e}_{0}(z)),
\end{equation}
where
\begin{eqnarray}
 ({\cal F}(\bar{v}^{\e}(t))  , \Theta^{\e})=
-a(\bar{\xi^{\e}}(t),\Theta^{\e})
-B(\bar{\xi^{\e}}(t),\bar{v}^{\e}(t),\Theta^{\e})
- B(\bar{v}^{\e}(t),\bar{\xi^{\e}}(t), \Theta^{\e})
 \nonumber\\
- {\cal
R}(\bar{\xi^{\e}}(t),\Theta^{\e})
- [{\partial \bar{\xi^{\e}}  \over \partial t}(t) ,  \Theta^{\e}] +
[\bar{f}^{\e}(t) ,  \Theta^{\e}],
\quad  \bar{\xi^{\e}} = (U^{\e}e_1 , W^{\e}),
\end{eqnarray}
and for all $\bar{v}= (v , Z)$, $\bar{u}= (u , w)$, and $\Theta = (\varphi
, \psi)$
in $V^{\e}\times {H^{1,}}^{\e}$,
\begin{eqnarray*}
\left[\bar{f}^{\e} , \Theta\right] &=& (f^{\e} ,\varphi) +
(g^{\e} , \psi),\nonumber\\
a(\bar{v} , \Theta) &=& (\nu +\nu_{r})(\nabla v , \nabla \varphi)
                          + \aal(\nabla Z, \nabla \psi),
\nonumber\\
 {\cal R}(\bar{v} , \Theta) &=& - 2 \nu_{r}({\rm rot \,} Z , \varphi)
-
  2 \nu_{r}({\rm rot \,} v , \psi) + 4\nu_{r} (Z , \psi) ,
\nonumber\\
B(\bar{v} ,\bar{u} , \Theta) &=& b(v , u , \varphi)
+ b_{1}(v , w , \psi)= \sum_{i, j=1}^{2}\int_{\Om^{\e}} v_{i}{\partial
u_{j}\over\partial z_{i}} \varphi_{j} dz
+\sum_{i=1}^{2}\int_{\Om^{\e}}v_{i} {\partial w\over\partial z_{i}} \psi dz.
\end{eqnarray*}

\begin{theorem}\label{th2.1}
 Let $T > 0$,  $U^{\e}$ and $ W^{\e}$ be given as  in {\rm Lemma \ref{lUW}},
$f^{\e}$ in $(L^{2}((0 , T)\times\Omega^{\e}))^{2}$,
   $g^{\e}$ in $L^{2}((0 , T)\times\Omega^{\e})$ and
  $(v^{\e}_{0} \, , \,Z^{\e}_{0})$ in $H^{\e}\times  {H^{0,}}^{\e}$.
Then  problem $(P^{\e})$ admits a unique solution $(v^{\e} , Z^{\e}, p^{\e})$.
\end{theorem}
 \begin{proof}
Following the techniques proposed by J.L.Lions in \cite{JLLions78}, we construct
a sequence of approximate solutions by relaxing the divergence free condition
for the velocity field. More precisely we consider the following penalized
problems $(P^{\e}_{\delta})$, with $\delta >0$:

\noindent {\bf Problem $(P^{\e}_{\delta})$}
 Find
$$\bar{v}^{\e}_{\delta} =(v^{\e}_{\delta}, Z^{\e}_{\delta} )  \in
\Bigl( {\cal C}([0 , T]; H^{\e})\cap L^{2}(0 , T; V^{\e}) \Bigr) \times \Bigl(
{\cal C}([0 ,
      T]; {H^{0,}}^{\e})\cap L^{2}(0 , T;  {H^{1,}}^{\e}) \Bigr) $$
 such that
\begin{eqnarray}\label{cal(P)}
&&\qquad [{\partial \bar{v}^{\e}_{\delta}\over \partial t} , \Theta^{\e}] +
 a(\bar{v}^{\e}_{\delta} , \Theta^{\e})
+B(\bar{v}^{\e}_{\delta} , \bar{v}^{\e}_{\delta} , \Theta^{\e})
+{1\over 2}\left\{
( v^{\e}_{\delta}{\rm div \,}  v^{\e}_{\delta} \, ,  \, \varphi^{\e})
+(Z^{\e}_{\delta}{\rm div \,}  v^{\e}_{\delta} \, ,  \, \psi^{\e}) \right\}
\nonumber\\
&&\qquad +{1\over \delta}({\rm div }  v^{\e}_{\delta} \, , \,  {\rm div
}\varphi^{\e})
= ({\cal F}(\bar{v}^{\e}_{\delta})  , \Theta^{\e})
- {\cal R}(\bar{v}^{\e}_{\delta} , \Theta^{\e}) \quad
\forall\Theta^{\e}
= (\varphi^{\e} , \psi^{\e})  \in V^{\e}\times {H^{1,}}^{\e},
\end{eqnarray}
with the initial condition
\begin{eqnarray}\label{cal(Pc)}
\bar{v}^{\e}_{\delta}(0)=  \bar{v^{\e}_{0}}.
\end{eqnarray}
The first term on the right of the second line of
(\ref{cal(P)})
is the penalty term and the term
$${1\over 2}\left\{
( v^{\e}_{\delta}{\rm div \,}  v^{\e}_{\delta} \, ,  \, \varphi^{\e})
+(Z^{\e}_{\delta}{\rm div \,}  v^{\e}_{\delta} \, ,  \, \psi^{\e}) \right\} $$
is added in order to vanish with $B ( \bar{v}^{\e}_{\delta} ,
\bar{v}^{\e}_{\delta} , \Theta^{\e})$ when $\Theta^{\e} = \bar v^{\e}_{\delta}$.

Hence the proof of Theorem \ref{th2.1} is divided in two parts. First we prove the existence of a
solution of $(P^{\e}_{\delta})$, for any $\delta >0$, by using a Galerkin
method. Then we pass to the limit as $\delta$ tends to zero by applying
compactness arguments and we prove that the limit solves problem $(P^{\e})$.

\smallskip

Since   $V^{\e}$ and $ {H^{1,}}^{\e}$ are closed subspaces of
$(H^1(\Om^{\e}))^2$ and $H^1(\Om^{\e})$, they admit  Hilbertian bases, denoted
as $(\Phi_{j})_{j \ge 1}$ and $ (\psi_{j})_{j\geq 1}$ respectively, which are
orthonormal in $(H^1(\Om^{\e}))^2$ and $H^1(\Om^{\e})$ and are  also orthogonal
bases of $(L^2(\Om^{\e}))^2$ and $L^2(\Om^{\e})$.
For all $m \ge 1$ we define $v^{\e}_{0 m}$ and $Z^{\e}_{0m}$ as the
$L^2$-orthogonal projection of $v^{\e}_0$ and $Z^{\e}_0$ on the finite
dimentional subspaces $\langle \Phi_1, \dots, \Phi_m \rangle$ and $\langle
\psi_1, \dots, \psi_m \rangle$ respectively and we let $\bar v^{\e}_{0 m} = (
v^{\e}_{0 m}, Z^{\e}_{0 m})$. Then we consider $\bar{v}^{\e}_{\delta m}=
(v^{\e}_{\delta m} , Z^{\e}_{\delta m})$, with
\begin{eqnarray}\label{jm}
v^{\e}_{\delta m} (t , x)  = \sum_{j=1}^{m} v^{\e}_{\delta mj}(t)
\Phi_{j}(x),
\qquad  Z^{\e}_{\delta m}(t , x)  = \sum_{j=1}^{m} Z^{\e}_{\delta mj}(t)
\psi_{j}(x)
\end{eqnarray}
such that
\begin{eqnarray}\label{cal(Pm)}
&&({\partial \bar{v}^{\e}_{\delta m}\over \partial t} , \Theta_{i})  +
 a(\bar{v}^{\e}_{\delta m} , \Theta_{i})
+B(\bar{v}^{\e}_{\delta m} , \bar{v}^{\e}_{\delta m} , \Theta_{i})
+{1\over 2}
(v^{\e}_{\delta m}{\rm div \,} v^{\e}_{\delta m} \, ,  \, \Phi_{i})
\nonumber\\
&&+{1\over 2}(Z_{\delta m}{\rm div \,} v^{\e}_{\delta m} \, ,  \, \psi_{i})
+{1\over \delta}({\rm div } v^{\e}_{\delta m} \, , \,  {\rm div
}\Phi_{i})
= ({\cal F}(\bar{v}^{\e}_{\delta m})  , \Theta_{i})
- {\cal R}(\bar{v}^{\e}_{\delta m} , \Theta_{i})
\nonumber\\ &&\forall  \Theta_i =
(\Phi_i, \psi_i), \   1\leq i\leq
m,
\\
&& {\bar{v}^{\e}_{\delta m}}(0)=   {\bar{v}}^{\e}_{0m }.\label{cal(Pcm)}
\end{eqnarray}
By taking $\psi_{i}= 0$ in (\ref{cal(Pm)}) we deduce
\begin{eqnarray}\label{Pm1}
({\partial v^{\e}_{\delta m}\over \partial t} , \Phi_{i})  &+&
(\nu+\nu_{r}) (\nabla v^{\e}_{\delta m} , \nabla \Phi_{i})
+b(v^{\e}_{\delta m} , v^{\e}_{\delta m} , \Phi_{i})
+{1\over 2}(v^{\e}_{\delta m}{\rm div } v^{\e}_{\delta m} \,  ,  \,
\Phi_{i})
\nonumber\\
&+&{1\over \delta}({\rm div \, } v^{\e}_{\delta m} \, , \,  {\rm div }\Phi_{i})
= ({\cal F}_{1}(v^{\e}_{\delta m})  , \Phi_{i})
+ 2\nu_{r} ({\rm rot \,} Z^{\e}_{\delta m} , \Phi_{i}) \qquad   1\leq i\leq
m, \qquad
\\
&& v^{\e}_{\delta m} (0)=  v^{\e}_{0m } ,\label{cal(Pcmv)}
\end{eqnarray}
and  by taking $\Phi_{i}= 0$ in  (\ref{cal(Pm)})  we deduce
\begin{eqnarray}\label{Pm2}
({\partial Z^{\e}_{\delta m}  \over \partial t} , \psi_{i}) &+&
 \alpha (\nabla Z^{\e}_{\delta m} , \nabla \psi_{i})
+b_{1}(v^{\e}_{\delta m} , Z^{\e}_{\delta m} , \psi_{i})
+{1\over 2}
(Z^{\e}_{\delta m}{\rm div \,} v^{\e}_{\delta m} \, ,   \psi_{i})
= ({\cal F}_{2}(v^{\e}_{\delta m})  , \psi_{i})
\nonumber\\
&&+2\nu_{r} ({\rm rot \,}  v^{\e}_{\delta m} , \psi_{i})
- 4\nu_{r} (Z^{\e}_{\delta m} \, ,  \,  \psi_{i})  \qquad
1\leq i\leq m,\qquad
\\
&& Z^{\e}_{\delta}(0)=  Z^{\e}_{0 m} , \label{cal(Pcmz)}
\end{eqnarray}
where
\begin{eqnarray}\label{F}
 ({\cal F}_{1}(v^{\e}_{\delta m})  ,   \Phi_{i})&=&
 -(\nu+\nu_{r}) (\nabla U^{\e}e_{1} , \nabla \Phi_{i})
 - b(U^{\e} e_{1}\, ,\,  v^{\e}_{\delta m} \, ,\, \Phi_{i})
- b(v^{\e}_{\delta m} \, ,\, U^{\e}_{\delta}e_{1}\, ,\, \Phi_{i})
\nonumber\\
&&+ 2\nu_{r}({\partial  W^{\e} \over \partial z_{2}}e_{1}     \, ,\,
\Phi_{i})
- ({\partial U^{\e}    \over \partial t}e_{1}  \, ,\, \Phi_{i})
 + (f^{\e} \, ,\, \Phi_{i}),
\end{eqnarray}
and
\begin{eqnarray}\label{FZ}
 ({\cal F}_{2}(v^{\e}_{\delta m})  , \psi_{i})&=&
 -\alpha (\nabla W^{\e} , \nabla \psi_{i})
 - b_{1}(U^{\e}e_{1} \, ,\, Z^{\e} \, ,\,   \psi_{i})
- b_{1}(v^{\e}_{\delta m} \, ,\,  W^{\e}\, ,\, \psi_{i})
\nonumber\\
&& - 2\nu_{r}(  {\partial  U^{\e} \over \partial z_{2}}    \, ,\, \psi_{i})  -4
\nu_{r} (W^{\e} \, ,\, \psi_{i})
- ({\partial W^{\e}    \over \partial t}  \, ,\, \psi_{i})
 + (g^{\e} \, ,\, \psi_{i}).
\end{eqnarray}

Taking  (\ref{jm}) into account, we deduce from (\ref{Pm1})-(\ref{FZ}) a
system of (nonlinear) differential equations for the unknown scalar functions
  $( v^{\e}_{\delta  m i}, Z^{\e}_{\delta m i})_{1 \le i \le m}$, which
possesses an unique maximal solution in $(H^1(0, T_m))^m$ with $T_m \in (0, T]$.

\smallskip
In order to prove that this solution is defined on the whole time interval
$[0,T]$, we will establish some a priori estimates for $v^{\e}_{\delta m}$
and  $Z^{\e}_{\delta m}$, independently of $m$.
More precisely,
 we  multiply the two sides of (\ref{Pm1}) by  $v^{\e}_{\delta mi}(t)$
and  the two sides of (\ref{Pm2}) by $Z^{\e}_{\delta mi}(t)$, then we sum for
$i$ from $1$ to
$m$, to get,  with $\|\cdot\|=\|\cdot\|_{L^{2}(\Omega^{\e})}$, the following
equations
\begin{eqnarray}\label{Pm1ss}
{1\over 2}{\partial\over \partial t} (\|v^{\e}_{\delta m}\|^{2})   &+&
(\nu+\nu_{r}) \|\nabla v^{\e}_{\delta m}\|^{2}
+b(v^{\e}_{\delta m} , v^{\e}_{\delta m} ,  v^{\e}_{\delta m})
+{1\over 2}(v^{\e}_{\delta m}{\rm div } v^{\e}_{\delta m} \,  ,  \,
v^{\e}_{\delta m})
+{1\over \delta}\|{\rm div } v^{\e}_{\delta m}\|^{2}
\nonumber\\
&&= ({\cal F}_{1}(v^{\e}_{\delta m})  ,  v^{\e}_{\delta m})
+ 2\nu_{r} ({\rm rot \,}  Z^{\e}_{\delta m} ,  v^{\e}_{\delta m}),\qquad
\end{eqnarray}
\begin{eqnarray}\label{Pm2ss}
{1\over 2}{\partial \over \partial t} ( \||Z^{\e}_{\delta m}\|^{2}) &+&
 \alpha \|\nabla Z^{\e}_{\delta m}\|^{2}
+b_{1}(v^{\e}_{\delta m} , Z^{\e}_{\delta m} , Z^{\e}_{\delta m})
+{1\over 2}
(Z^{\e}_{\delta m}{\rm div \,} v^{\e}_{\delta m} \, ,  Z^{\e}_{\delta m})
= ({\cal F}_{2}(v^{\e}_{\delta m})  , Z^{\e}_{\delta m})
\nonumber\\
&&+2\nu_{r} ({\rm rot \,}   v^{\e}_{\delta m} , Z^{\e}_{\delta m})
- 4\nu_{r} (Z^{\e}_{\delta m} \, ,  \,  Z^{\e}_{\delta m}) .
\end{eqnarray}
By integration by parts and using the boundary conditions
(\ref{eqn:er2.8a})-(\ref{eqn:er2.11b}), we obtain that
\begin{eqnarray*}
b(v^{\e}_{\delta m} , v^{\e}_{\delta m} ,  v^{\e}_{\delta m})
+b_{1}(v^{\e}_{\delta m} , Z^{\e}_{\delta m} , Z^{\e}_{\delta m})
+{1\over 2}(v^{\e}_{\delta m}{\rm div } v^{\e}_{\delta m} \,  ,  \,
v^{\e}_{\delta m})
+{1\over 2}
(Z^{\e}_{\delta m}{\rm div \,} v^{\e}_{\delta m} \, ,  Z^{\e}_{\delta m})
= 0,
\end{eqnarray*}
and
\begin{eqnarray*}
 b(U^{\e} e_{1}\, ,\,  v^{\e}_{\delta m} \, ,\, v^{\e}_{\delta m})
+ b_{1}(U^{\e}e_{1} \, ,\, Z^{\e}_{\delta m}  \, ,\,   Z^{\e}_{\delta m})
 =  0.
\end{eqnarray*}
Thus by the addition of (\ref{Pm1ss}) and (\ref{Pm2ss}) we obtain
\begin{eqnarray}\label{eq2.28E}
\qquad\qquad {1\over 2}{\partial\over \partial t} (\|v^{\e}_{\delta m}\|^{2}
+\||Z^{\e}_{\delta m}\|^{2})
+(\nu+\nu_{r}) \|\nabla v^{\e}_{\delta m}\|^{2} + \alpha \|\nabla
Z^{\e}_{\delta m}\|^{2}
+{1\over \delta}\|{\rm div } v^{\e}_{\delta m}\|^{2}
= \Xi
\end{eqnarray}
with
 \begin{eqnarray*}\label{Xi}
\Xi&=&
 2\nu_{r} ({\rm rot \,} Z^{\e}_{\delta m} ,  v^{\e}_{\delta m})
+2\nu_{r} ({\rm rot \,}  v^{\e}_{\delta m} , Z^{\e}_{\delta m})
- 4\nu_{r} \|Z^{\e}_{\delta m}\|^{2}
 -(\nu+\nu_{r}) (\nabla U^{\e}e_{1} , \nabla v^{\e}_{\delta m})
 \nonumber\\
&&
-\alpha (\nabla W^{\e} , \nabla Z^{\e}_{\delta m})
- b(v^{\e}_{\delta m} \, ,\, U^{\e}e_{1}\, ,\, v^{\e}_{\delta m})
 - b_{1}(v^{\e}_{\delta m} \, ,\,  W^{\e}\, ,\, Z^{\e}_{\delta m})
+ 2\nu_{r}({\rm rot \,} W^{\e}  \, ,\,  v^{\e}_{\delta m})
\nonumber\\
&&
+ 2\nu_{r}({\rm rot \,} U^{\e}e_{1}\, ,\, Z^{\e}_{\delta m})
 -4 \nu_{r} (W^{\e} \, ,\, Z^{\e}_{\delta m})
- ({\partial U^{\e} e_{1}   \over \partial t}  \, ,\,  v^{\e}_{\delta m})
- ({\partial W^{\e}   \over \partial t}  \, ,\,  Z^{\e}_{\delta m})
\nonumber\\
&&
 + (f^{\e} \, ,\,   v^{\e}_{\delta m})
 + (g^{\e} \, ,\, Z^{\e}_{\delta m}).
\end{eqnarray*}
Using Young's inequality we have
\begin{eqnarray*}\label{esR}
2 \nu_{r}|({\rm rot \,} {Z^{\e}_{\delta m}} , {v^{\e}_{\delta m}})|
\leq
2 \nu_{r}\|{\rm rot \,} {Z^{\e}_{\delta m}}\| \|{v^{\e}_{\delta m}}\|
\leq {\aal\over 4}\|\nabla {Z^{\e}_{\delta m}}\|^{2} +  {4\nu_{r}^{2}\over
\aal}\| {v^{\e}_{\delta m}}\|^{2},
\end{eqnarray*}
\begin{eqnarray*}
 2\nu_{r} | ( {\rm rot \,}  v^{\e}_{\delta m} , Z^{\e}_{\delta m}) |
&& \leq
2 \nu_{r}\|{\rm rot \,} {v^{\e}_{\delta m}}\|\|{Z^{\e}_{\delta m}}\|
\leq {\nu_{r}\over 4} \| {\rm rot \,}  {v^{\e}_{\delta m}}\|^{2}
+ 4\nu_{r}\|{Z^{\e}_{\delta m}}\|^{2} \\
&& \leq {\nu_{r}\over 2} \|\nabla {v^{\e}_{\delta m}}\|^{2}
+ 4\nu_{r}\|{Z^{\e}_{\delta m}}\|^{2},
\end{eqnarray*}
\begin{eqnarray*}\label{aa}
 (\nabla U^{\e}e_{1} , \nabla v^{\e}_{\delta m})
\leq
{1\over 2}\|\nabla ({v^{\e}_{\delta m}})\|^{2}
+ {1\over 2}\|{\partial U^{\e}\over \partial z_{2}}\|^{2},
\end{eqnarray*}
\begin{eqnarray*}\label{al}
(\nabla W^{\e} , \nabla Z^{\e}_{\delta m})
\leq
{1 \over 4}\|\nabla {Z^{\e}_{\delta m}}\|^{2}
+ \|{\partial W^{\e}\over \partial z_{2}}\|^{2},
\end{eqnarray*}
\begin{eqnarray*}\label{bb}
 b(v^{\e}_{\delta m} \, ,\, U^{\e} e_{1}\, ,\, v^{\e}_{\delta m})
\leq
\|({v^{\e}_{\delta m}})_{2}\|
\|{\partial U^{\e}\over \partial z_{2}}\|_{\infty}
\|({v^{\e}_{\delta m}})_{1}\|
\leq \|{\partial U^{\e}\over \partial z_{2}}\|_{\infty}\|{v^{\e}_{\delta
m}}\|^{2},
\end{eqnarray*}
\begin{eqnarray*}\label{bb1}
  b_{1}(v^{\e}_{\delta m} \, ,\,  W^{\e}\, ,\, Z^{\e}_{\delta m})
\leq \|({v^{\e}_{\delta m}})_{2}\|\|{\partial W^{\e} \over \partial
z_{2}}\|_{\infty}
\|Z^{\e}_{\delta m}\|
\leq \frac{1}{2}  \|{\partial W^{\e} \over \partial z_{2}}\|_{\infty}\left(
   \|{v^{\e}_{\delta m}}\|^{2} +  \|Z^{\e}_{\delta m}\|^{2}\right),
\end{eqnarray*}
\begin{eqnarray*}\label{rr}
&& 2\nu_{r}( {\rm rot \,} W^{\e}  \, ,\,  v^{\e}_{\delta m})
+ 2\nu_{r}({\rm rot \,}  U^{\e}e_{1}\, ,\, Z^{\e}_{\delta m})
 -4 \nu_{r} (W^{\e} \, ,\, Z^{\e}_{\delta m})
\nonumber\\
&=& + 2 \nu_{r} \left({\partial W^{\e} \over \partial z_{2}} ,  {(v^{\e}_{\delta
m})}_{1}\right)
    -   2 \nu_{r}   \left({\partial U^{\e} \over \partial z_{2}} ,
Z^{\e}_{\delta
m}\right)
- 4\nu_{r} (W^{\e} , Z^{\e}_{\delta m})
\nonumber\\
&\leq&    \nu_{r}\|{v^{\e}_{\delta m}}\|^{2} + 2 \nu_{r}\|Z^{\e}_{\delta
m}\|^{2} +
 \nu_{r}\| {\partial W^{\e} \over \partial z_{2}}\|^{2} +
 \nu_{r}\|{\partial U^{\e} \over \partial z_{2}}\|^{2}  + 4\nu_{r}\|
W^{\e}\|^{2}.
\end{eqnarray*}
So we have
 \begin{eqnarray}\label{XiMag}
\Xi&\leq&
({\nu\over 2}+\nu_{r})\|\nabla{v^{\e}_{\delta m}}\|^{2}
+ {\alpha\over 2}\|\nabla{Z^{\e}_{\delta m}}\|^{2}
+\left(2+ 4 \nu_{r} + \frac{1}{2} \|{\partial W^{\e}(t) \over \partial
z_{2}}\|_{\infty}
\right)
\|{Z^{\e}_{\delta m}}\|^{2}
\nonumber\\
&&+\left(2+ \nu_{r} + {4\nu_{r}^2\over \alpha} +\|{\partial U^{\e}(t)\over
\partial z_{2}}\|_{\infty}
+ \frac{1}{2} \|{\partial W^{\e}(t) \over \partial z_{2}}\|_{\infty}
\right)\| {v^{\e}_{\delta m}}\|^{2}
\nonumber\\
&&+{(\nu+\nu_{r})\over 2}\|{\partial U^{\e}(t)\over \partial z_{2}}\|^{2}
+ \alpha\|{\partial W^{\e}(t)\over \partial z_{2}}\|^{2}
+
 \nu_{r}\| {\partial W^{\e}(t) \over \partial z_{2}}\|^{2}
+\nu_{r}\|{\partial U^{\e}(t) \over \partial z_{2}}\|^{2}
\nonumber\\ &&
+
  4\nu_{r}\| W^{\e}(t)\|^{2}
  +\left\|{\partial U^{\e} (t) \over \partial t}\right\|^{2}
+\left\|{\partial W^{\e}(t)  \over \partial t}\right\|^{2}
  + \|f^{\e}(t)\|^{2} + \|g^{\e}(t)\|^{2}.
\end{eqnarray}
From  (\ref{eq2.28E})-(\ref{XiMag}),  we get
\begin{eqnarray}\label{eq2.31}
{1\over 2}{\partial \over\partial t}(
[\bar{v}^{\e}_{\delta m}]^{2})
+
{k\over 2}[[\bar{v}^{\e}_{\delta  m}]]^{2}
+{1\over \delta}\|{\rm div\, } v^{\e}_{\delta m}\|^{2}
\leq A(t) [\bar{v}^{\e}_{\delta m}]^{2} + B (t),
\end{eqnarray}
where   $k= \min \{\nu ,
\alpha\}$ and
$A$ and $B$ belong to $L^1(0,T)$ such that $A(t) \ge 2$ and $B(t) \ge 0$ almost
everywhere on $[0,T]$. Moreover $A$  and
$B$   depend neither on $m$ nor on $\delta$.

For any $t \in (0,T_m)$ we can
  integrate  the inequality  (\ref{eq2.31}) over $[0,t]$:
 we obtain
\begin{eqnarray}\label{eqint2.31}
[\bar{v}^{\e}_{\delta m}(t)]^{2} +
k \int_{0}^{t} [[\bar{v}^{\e}_{\delta m}(s)]]^{2} ds
+{2\over \delta}\int_{0}^{t} \|{\rm div\, } v^{\e}_{\delta m}(s)\|^{2}ds
\leq [\bar{v}^{\e}_{0}]^{2} \nonumber\\+ 2\int_{0}^{t} A(t)[\bar{v^{\e}_{\delta
m}}(s)]^{2}ds + 2{\cal B},
 \end{eqnarray}
 with ${\cal B}= \int_{0}^{T}B(t) dt$.
So by Gr\"onwall's inequality,  we deduce first that
$$ [\bar{v}^{\e}_{\delta m}(t)]^{2}
\leq ( [\bar{v}^{\e}_{0}]^{2} + 2 {\cal B})  e^{2 {\cal A}}
 \quad \mbox{with}\quad {\cal A}= \int_{0}^{T}A(t) dt.$$
Thus $\bar{v}^{\e}_{\delta m}$ is defined on the whole interval $[0,T]$ and
\begin{eqnarray}\label{Maj1}
\sup_{t\in [0 , T]}[\bar{v}^{\e}_{\delta m}(t)]^{2} \leq C.
\end{eqnarray}
Then  from (\ref{eqint2.31}) and (\ref{Maj1}), we deduce
\begin{eqnarray}\label{Maj2}
{1\over \delta}\int_{0}^{T} \|{\rm div\, } v^{\e}_{\delta m}(t)\|^{2}dt \le C,
\quad \int_{0}^{T} [[\bar{v}^{\e}_{\delta m}(t)]]^{2} dt\leq C,
 \end{eqnarray}
where here and in what follows  $C's$ denotes various constants
 which   depend neither on $m$ nor on $\delta$.

 We need now  to look at the time derivative of  $v^{\e}_{\delta
m}$ and  $Z^{\e}_{\delta m}$.
Let $\Theta^{\e}=(\varphi^{\e} , \psi^{\e}) \in (H^1_0(\Om^{\e}))^2 \times
H^1_0(\Om^{\e}) \subset  V^{\e}\times H^{1,
\e}$.   There exists  a sequence $(q^{\e}_{i} , k^{\e}_{i})_{i \ge 1}$ in
$\br^{2}$
such that
$$ \Theta^{\e}_{p} = (\varphi^{\e}_{p}, \psi^{\e}_{p})
 \to  (\varphi^{\e} , \psi^{\e}) \quad \mbox{ strongly in } V^{\e}\times H^{1,
\e} $$
with
$$
\varphi^{\e}_{p} = \sum_{i=1}^{p}q^{\e}_{i}\Phi_{i},
\quad \psi^{\e}_{p} = \sum_{i=1}^{p}k^{\e}_{i}\psi_{i} \qquad  \forall p \ge 1.
 $$
Let $p \ge m$. Reminding that $(\Phi_i)_{i \ge 1}$ and $(\psi_i)_{i \ge 1}$ are
orthogonal bases of $(L^2(\Om^{\e}))^2$ and $L^2(\Om^{\e})$ respectively, we get
\begin{eqnarray*}
\left( \frac{\partial v^{\e}_{\delta m}}{\partial t}, \varphi^{\e}_p \right) =
\sum_{j=1}^m ( v^{\e}_{\delta m j})' (t) (\Phi_j, \varphi^{\e}_p) =
\sum_{j=1}^m ( v^{\e}_{\delta m j})' (t) (\Phi_j, \varphi^{\e}_m) = \left(
\frac{\partial v^{\e}_{\delta m}}{\partial t}, \varphi^{\e}_m \right),
\end{eqnarray*}
and
\begin{eqnarray*}
\left( \frac{\partial Z^{\e}_{\delta m}}{\partial t}, \psi^{\e}_p \right) =
\sum_{j=1}^m ( Z^{\e}_{\delta m j})' (t) (\Phi_j, \psi^{\e}_p) =  \sum_{j=1}^m (
Z^{\e}_{\delta m j})' (t) (\Phi_j, \psi^{\e}_m) = \left( \frac{\partial
Z^{\e}_{\delta m}}{\partial t}, \psi^{\e}_m \right).
\end{eqnarray*}
Since $ \frac{\partial \bar v^{\e}_{\delta m}}{\partial t} \in L^2(0,T; V^{\e}
\times H^{1, \e})$, we can pass to the limit as $p$ tends to $+ \infty$ i.e
\begin{eqnarray*}
\left( \frac{\partial v^{\e}_{\delta m}}{\partial t}, \varphi^{\e} \right) =
\left( \frac{\partial v^{\e}_{\delta m}}{\partial t}, \varphi^{\e}_m \right),
\quad \left( \frac{\partial Z^{\e}_{\delta m}}{\partial t}, \psi^{\e} \right) =
\left( \frac{\partial Z^{\e}_{\delta m}}{\partial t}, \varphi^{\e}_m \right).
\end{eqnarray*}
Then, by using  Green's formula and (\ref{Pm1})
 \begin{eqnarray} \label{dg1}
\left( {\partial { v^{\e}_{\delta m}} \over \partial t}, \varphi^{\e} \right)
=\Bigl(  (\nu +\nu_{r})\Delta { v^{\e}_{\delta m}}
-  ({ v^{\e}_{\delta m}}  \cdot \nabla) { v^{\e}_{\delta m}}
-{1\over 2} { v^{\e}_{\delta m}}{\rm div\,}{ v^{\e}_{\delta m}}
\nonumber\\ + {\cal F}_{1}({ v^{\e}_{\delta m}})+ 2\nu_{r}{\rm
rot\,}{Z^{\e}_{\delta m}}
 + {1\over \delta} \nabla ({\rm div\,}{ v^{\e}_{\delta m}}), \varphi^{\e}_m
\Bigr),
\end{eqnarray}
and from  (\ref{Pm2})
\begin{eqnarray} \label{dg2}
 \Bigl( {\partial Z^{\e}_{\delta m} \over \partial t}, \psi^{\e} \Bigr)  =
\Bigl( \alpha\Delta {
Z^{\e}_{\delta m}}
-  ({ v^{\e}_{\delta m}}  \cdot \nabla) { Z^{\e}_{\delta m}}
-{1\over 2} { Z^{\e}_{\delta m}}{\rm div\,}{ v^{\e}_{\delta m}}
+ {\cal F}_{2}({ v^{\e}_{\delta m}})+ 2\nu_{r}{\rm rot\,}{v^{\e}_{\delta m}}
\nonumber \\
 -
4\nu_{r} Z^{\e}_{\delta m},\psi^{\e}_m \Bigr)
\end{eqnarray}
and from (\ref{F})
$${\cal F}_{1}({ v^{\e}_{\delta m}})= (\nu+\nu_{r}) {\partial^{2} U^{\e}\over
\partial z^{2}_{2}}e_{1}
   - U^{\e}{\partial { v^{\e}_{\delta m}} \over \partial z_{1}}
   - ({ v^{\e}_{\delta m}})_{2} {\partial U^{\e}\over \partial z_{2}}e_{1}
 + 2\nu_{r}  {\partial W^{\e}\over \partial z_{2}} e_{1}
 - {\partial  U^{\e}\over \partial t}e_{1}+ f^{\e},$$
and from (\ref{FZ})
$${\cal F}_{2}({ v^{\e}_{\delta m}})=  \alpha {\partial^{2} W^{\e}\over
\partial z^{2}_{2}}
   - U^{\e}{\partial { Z^{\e}_{\delta m}} \over \partial z_{1}}
   - ({ v^{\e}_{\delta m}})_{2} {\partial W^{\e}\over \partial z_{2}}
 + 2\nu_{r}  {\partial U^{\e}\over \partial z_{2}}  - 4 \nu_{r} W^{\e}-
{\partial  W^{\e}\over \partial t} + g^{\e}.
$$
As ${v^{\e}_{\delta m}}$ is bounded in $L^{2}(0 , T ; (H^{1}(\Omega^{\e})^{2}))$
independently of $m$ and $\delta$,  then
$\Delta { v^{\e}_{\delta m}}$ and
$\nabla ({\rm div\,}{ v^{\e}_{\delta m}})$ are also bounded in  $L^{2}(0 , T ;
(H^{-1}(\Omega^{\e}))^{2})$  independently of $m$ and $\delta$.
 Similarly, since   ${Z^{\e}_{\delta m}}$ is bounded  in $L^{2}(0 , T ;
H^{1}(\Omega^{\e}))$  independently of $m$ and $\delta$, then
 ${\rm rot\,}{Z^{\e}_{\delta m}}$ is also bounded in $L^{2}(0 , T ;
(L^{2}(\Omega^{\e}))^{2})$  independently of $m$ and $\delta$.
By assumption,  $f^{\e} \in (L^{2}((0 , T)\times\Om^{\e})^{2}$,  $g^{\e} \in
L^{2}((0 ,
T)\times\Om^{\e})$, and from Lemma \ref{lUW},  $U^{\e}$ and  $W^{\e}$ belong to
$H^{1}(0 , T)\times {\cal
D}((-\infty , h_{m}))$.
Thus we infer  that
${\cal F}_{1}({ v^{\e}_{\delta m}})$ and ${\cal F}_{2}({ v^{\e}_{\delta m}})$
are  bounded in  $L^{2}(0 , T;(L^2(\Omega^{\e}))^{2})$ and $L^{2}(0 ,
T;L^2(\Omega^{\e})$, independently of $m$ and $\delta$.
Moreover let $\varphi\in ( H^{1}(\Omega^{\e}))^2$, we have
 \begin{eqnarray*}
  | (({ v^{\e}_{\delta m}} \cdot \nabla){ v^{\e}_{\delta m}} \, ,\,
\varphi)|
 \leq \|{ v^{\e}_{\delta m}}\|_{L^{3}(\Omega^{\e})}\|\nabla { v^{\e}_{\delta
m}}\|_{L^{2}(\Omega^{\e})}\|\varphi\|_{L^{6}(\Omega^{\e})}.
 \end{eqnarray*}
Using now the classical inequality
$$
\|u\|_{L^{3}(\Omega^{\e})}\leq
\|u\|^{1/2}_{L^{2}(\Omega^{\e})}\|u\|^{1/2}_{L^{6}(\Omega^{\e})}
\quad \forall u\in  L^{6}(\Om^{\e}),
$$
and the  continuous injection of $H^{1}(\Omega^{\e})$ in
$L^{6}(\Omega^{\e})$,
 there  exists a constant $C$ such that
 \begin{eqnarray*}
    | (({ v^{\e}_{\delta m}} \cdot \nabla) { v^{\e}_{\delta m}} \, ,\,
\varphi)|
\leq \left( C\|{ v^{\e}_{\delta m}}\|^{1/2}_{L^{2}(\Omega^{\e})}\|\nabla {
v^{\e}_{\delta m}}\|^{3/2}_{L^{2}(\Omega^{\e})}\right)
\|\varphi\|_{H^{1}(\Omega^{\e})}.
 \end{eqnarray*}
So we get
 \begin{eqnarray*}
\|({ v^{\e}_{\delta m}} \cdot \nabla) { v^{\e}_{\delta
m}}\|_{(H^{1}(\Omega^{\e}))'}
\leq C\|{ v^{\e}_{\delta m}}\|^{1/2}_{L^{2}(\Omega^{\e})}\|\nabla {
v^{\e}_{\delta m}}\|^{3/2}_{L^{2}(\Omega^{\e})}
 \end{eqnarray*}
then
\begin{eqnarray*}
\int_{0}^{T}\| ({ v^{\e}_{\delta m}} \cdot \nabla) { v^{\e}_{\delta
m}}\|^{4/3}_{(H^{1}(\Omega^{\e}))'} dt
&\leq& C^{4/3}\int_{0}^{T}\|{ v^{\e}_{\delta
m}}\|^{2/3}_{L^{2}(\Omega^{\e})}\|\nabla { v^{\e}_{\delta
m}}\|^{2}_{L^{2}(\Omega^{\e})}
\nonumber\\
 &\leq&  C^{4/3}\|{ v^{\e}_{\delta m}}\|^{2/3}_{L^{\infty}(0 , T ;
L^{2}(\Omega^{\e}))}\|\nabla { v^{\e}_{\delta m}}\|^{2}_{L^{2}((0 ,
T)\times\Omega^{\e})}.
 \end{eqnarray*}
With the same arguments,  we  deduce  similar result  for
${v^{\e}_{\delta m}}{\rm div\,}{ v^{\e}_{\delta m}}$,
 $({ v^{\e}_{\delta m}}  \cdot \nabla) { Z^{\e}_{\delta m}}$
and  ${ Z^{\e}_{\delta m}}{\rm div\,}{ v^{\e}_{\delta m}}$.
Finally, recalling that $(\Phi_i)_{i \ge 1}$ and $(\psi_i)_{i \ge 1}$ are
$H^1$-orthonormal, we get
\begin{eqnarray*}
\| \varphi^{\e}_m \|_{(H^1(\Om^{\e}))^2} \le \| \varphi^{\e}
\|_{(H^1(\Om^{\e}))^2}, \quad \| \psi^{\e}_m \|_{H^1(\Om^{\e})} \le \|
\psi^{\e} \|_{H^1(\Om^{\e})} \quad \forall m \ge 1.
\end{eqnarray*}
So from (\ref{dg1}) and (\ref{dg2}) we see that there exists a constant $C$ such
that
\begin{eqnarray}\label{boudv}
\|{\partial {v^{\e}_{\delta m}} \over\partial t}
\|_{L^{4/3}(0 , T ;( H^{-1}(\Omega^{\e}))^{2})}\leq C,
\quad
\|{\partial {Z^{\e}_{\delta m}} \over  \partial t}
\|_{L^{4/3}(0 , T ; H^{-1}(\Omega^{\e}))}\leq C.
\end{eqnarray}
From the estimates  (\ref{Maj1})-(\ref{Maj2}) we infer that there exists a
subsequence
(denoted also by)  $\bar{v}^{\e}_{\delta m}$ such that
\begin{eqnarray}\label{lim1m}
 \bar{v}^{\e}_{\delta m} \tow \bar{v}^{\e}_{\delta}  \quad \mbox{  in }
L^{2}(0 , T; V^{\e})\times L^{2}(0 , T; H^{1 , \e}) \quad \mbox{  weakly for }
m \to +\infty,
\end{eqnarray}
\begin{eqnarray}\label{lim2m}
 \bar{v}^{\e}_{\delta m} \tow \bar{v}^{\e}_{\delta}  \quad \mbox{  in }
L^{\infty}(0 , T;  H^{\e})\times L^{\infty}(0 , T; H^{0 , \e}) \quad \mbox{weak
star for }
 m \to +\infty,
\end{eqnarray}
and from (\ref{boudv}),   by  Aubin's compactness theorem A.11 in \cite{aubin},
  there are two subsequences (denoted also by) ${v^{\e}_{\delta m}}$,
${Z^{\e}_{\delta m}}$  satisfying for $m\to +\infty$ the following strong
convergence
\begin{eqnarray}\label{cfVZm}
 v^{\e}_{\delta m} \to v^{\e}_{\delta}  \mbox{ in }   L^{2}(0 , T;
(L^{4}(\Omega^{\e}))^{2}), \quad
 Z^{\e}_{\delta m}\to Z^{\e}_{\delta} \mbox{ in }
 L^{2}(0 , T;
L^{4}(\Omega^{\e})).
\end{eqnarray}


In order to pass to the limit  as $m  \to  +\infty$,
we remind that for any  $\Theta^{\e}=(\varphi^{\e} , \psi^{\e}) \in V^{\e}\times
H^{1,
\e}$,  there exists  a sequence $(q^{\e}_{i} , k^{\e}_{i})_{i \ge 1}$ in
$\br^{2}$
such that
$$ \Theta^{\e}_{m} = (\varphi^{\e}_{m}, \psi^{\e}_{m})
 \to  (\varphi^{\e} , \psi^{\e}) \quad \mbox{ strongly in } V^{\e}\times H^{1,
\e} $$
with
$$
\varphi^{\e}_{m} = \sum_{i=1}^{m}q^{\e}_{i}\Phi_{i},
\quad \psi^{\e}_{m} = \sum_{i=1}^{m}k^{\e}_{i}\psi_{i} \qquad  \forall m \ge 1.
 $$
We multiply first  the two sides of (\ref{Pm1}) by $q^{\e}_{i}$ then
we sum for
$i=1$ to $m$, and we multiply the two sides of (\ref{Pm2}) by $k^{\e}_{i}$
then we sum also  for
$i=1$ to $m$, we obtain
\begin{eqnarray}\label{Pm1n}
({\partial v^{\e}_{\delta m}\over \partial t} , \varphi^{\e}_{m})  &+&
(\nu+\nu_{r}) (\nabla v^{\e}_{\delta m} , \nabla \varphi^{\e}_{m})
+b(v^{\e}_{\delta m} , v^{\e}_{\delta m} , \varphi^{\e}_{m})
+{1\over 2}(v^{\e}_{\delta m}{\rm div } v^{\e}_{\delta m} \,  ,  \,
\varphi^{\e}_{m})
\nonumber\\
&&
+{1\over \delta}({\rm div } v^{\e}_{\delta m} \, , \,  {\rm div
}\varphi^{\e}_{m})
= ({\cal F}_{1}(v^{\e}_{\delta m})  , \varphi^{\e}_{m})
+ 2\nu_{r} ({\rm rot \,} \, Z^{\e}_{\delta m} , \varphi^{\e}_{m}),
\\
&& v^{\e}_{\delta m} (0)=  v^{\e}_{0 m} , \label{cal(Pcmvn)}
\end{eqnarray}
and
\begin{eqnarray}\label{Pm2n}
({\partial Z^{\e}_{\delta m}  \over \partial t} , \psi^{\e}_{m}) &+&
 \alpha (\nabla Z^{\e}_{\delta m} , \nabla \psi^{\e}_{m})
+b_{1}(v^{\e}_{\delta m} , Z^{\e}_{\delta m} , \psi^{\e}_{m})
+{1\over 2}
(Z^{\e}_{\delta m}{\rm div \,} v^{\e}_{\delta m} \, ,   \psi^{\e}_{m})
\nonumber\\
&&
= ({\cal F}_{2}(v^{\e}_{\delta m})  , \psi^{\e}_{m})
+2\nu_{r} ({\rm rot \,}  v^{\e}_{\delta m} , \psi^{\e}_{m})
- 4\nu_{r} (Z^{\e}_{\delta m} \, ,  \,  \psi^{\e}_{m}),
\\
&& Z^{\e}_{\delta}(0)=  Z^{\e}_{0 m}.\label{cal(Pcmzn)}
\end{eqnarray}
Let  $\theta\in {\cal D}(0 , T)$,
we multiply (\ref{Pm1n}) and (\ref{Pm2n}) by $\theta(t)$ and we  integrate over
$[0,T]$.
We get
\begin{eqnarray}\label{cal(PPPm)}
-  \int_{0}^{T}(\bar{v}^{\e}_{\delta m} (t), \Theta^{\e}_{m})
\theta'(t) dt
+
\int_{0}^{T}  \left\{
 a(\bar{v}^{\e}_{\delta m} , \Theta^{\e}_{m})
+B(\bar{v}^{\e}_{\delta m} , \bar{v}^{\e}_{\delta m} ,
\Theta^{\e}_{m})\right\} \theta(t) dt
\nonumber\\
+{1\over\delta}\int_{0}^{T} ({\rm div\, } v^{\e}_{\delta m}  \, , \,  {\rm div
}\varphi^{\e}_{m}) \theta(t) dt
+{1\over 2}\int_{0}^{T}
\left\{v^{\e}_{\delta m}  {\rm div \,}  v^{\e}_{\delta m} \, ,  \,
\varphi^{\e}_{m})
+(Z^{\e}_{\delta}{\rm div \,}  v^{\e}_{\delta m} \, ,  \, \psi^{\e}_{m})
\right\} \theta(t) dt
\nonumber\\
= \int_{0}^{T}  \left\{  ({\cal F}(\bar{v}^{\e}_{\delta m})  ,
\Theta^{\e}_{m})
- {\cal R}(\bar{v}^{\e}_{\delta m} , \Theta^{\e}_{m}) \right\}
\theta(t) dt
.\quad
\end{eqnarray}
Using the convergences (\ref{lim1m})-(\ref{lim2m}),
we can now pass easily to the limit  in
all terms of (\ref{cal(PPPm)}) except for the nonlinear terms
$$ \int_{0}^{T}  B(\bar{v}^{\e}_{\delta m} , \bar{v}^{\e}_{\delta m} ,
\Theta^{\e}_{m}) \theta(t) dt
= \int_{0}^{T}  b({ v^{\e}_{\delta m}} , { v^{\e}_{\delta m}} ,
\varphi^{\e}_{m}) \theta(t) dt
+   \int_{0}^{T}  b_{1}({ v^{\e}_{\delta m}} , {Z^{\e}_{\delta}} ,
\psi^{\e}_{m}) \theta(t) dt$$
and
$$ \int_{0}^{T}
\left\{  v^{\e}_{\delta m}  {\rm div \,}  v^{\e}_{\delta m} \, ,  \,
\varphi^{\e}_{m})
+(Z^{\e}_{\delta}{\rm div \,}  v^{\e}_{\delta m} \, ,  \, \psi^{\e}_{m})
\right\} \theta(t) dt.$$
We have first
\begin{eqnarray} \label{nonltv}
\int_{0}^{T} b(v^{\e}_{\delta m} , v^{\e}_{\delta m} , \varphi^{\e}_{m}) \theta
(t) dt
&=& -   \int_{0}^{T} b(v^{\e}_{\delta m} , \varphi^{\e}_{m}, v^{\e}_{\delta
m} ) \theta(t)  dt
-\int_{0}^{T} ( {\rm div\,} {v^{\e}_{\delta m}} , \varphi^{\e}_{m}\cdot
{v^{\e}_{\delta m}} ) \theta(t)  dt
\nonumber\\
&&+ \int_{0}^{T} \int_{\partial\Om^{\e}}
 (\varphi^{\e}_{m}\cdot {v^{\e}_{\delta m}})({v^{\e}_{\delta m}}\cdot n)
\theta(t)
d\sigma dt.
\end{eqnarray}
Using the boundary conditions (\ref{eqn:er2.8a})-(\ref{eqn:er2.11b}),
we obtain that the last integral is equal to zero,
then for the first and the second integrals we use the strong convergence
(\ref{cfVZm}).
 So we get
\begin{eqnarray*}
\int_{0}^{T} b(v^{\e}_{\delta m} , v^{\e}_{\delta m} , \varphi^{\e}_{m}) \theta
(t) dt
\to
\int_{0}^{T} b(v^{\e}_{\delta} , v^{\e}_{\delta} , \varphi^{\e}) \theta (t) dt
\mbox{ for } m\to  +\infty.
\end{eqnarray*}
Similarly
\begin{eqnarray} \label{nonltZ}
\int_{0}^{T} b(v^{\e}_{\delta m} , Z^{\e}_{\delta m} , \psi^{\e}_{m}) \theta (t)
dt
&=& -   \int_{0}^{T} b(v^{\e}_{\delta m} ,  \psi^{\e}_{m} , Z^{\e}_{\delta m} )
\theta (t)
dt
-\int_{0}^{T} ( {\rm div\,} v^{\e}_{\delta m} , \psi^{\e}\cdot Z^{\e}_{\delta m}
) \theta (t) dt
\nonumber\\
&&+ \int_{0}^{T} \int_{\partial \Om^{\e}}
 (\psi^{\e}_{m}\cdot Z^{\e}_{\delta m})(v^{\e}_{\delta m}\cdot n) \theta (t)
d\sigma dt.
\end{eqnarray}
Using the boundary conditions (\ref{eqn:er2.8a})-(\ref{eqn:er2.11b}),
we obtain that the last integral is equal to zero,
then for the first and the second integrals we use the strong convergence
(\ref{cfVZm}). So we get
\begin{eqnarray*}
\int_{0}^{T} b(v^{\e}_{\delta m} , Z^{\e}_{\delta m} , \psi^{\e}_{m}) \theta (t)
 dt
\to
\int_{0}^{T} b(v^{\e}_{\delta} , \psi^{\e}, Z^{\e}_{\delta}) \theta (t) dt
\mbox{ for } m\to  +\infty.
\end{eqnarray*}
We can now pass to the limit ($m\to  +\infty$) in  all terms of
(\ref{cal(PPPm)})  to get
\begin{eqnarray}\label{cal(PPP1)}
\int_{0}^{T}  (v^{\e}_{\delta} , \Theta^{\e}) \theta'(t) dt =
 \int_{0}^{T}  \left\{
a( \bar{v}^{\e}_{\delta}  , \Theta^{\e})
+ B(\bar{v}^{\e}_{\delta}  , \bar{v}^{\e}_{\delta}  , \Theta^{\e})
+{1\over\delta} ({\rm div\,} v^{\e}_{\delta}   \, , \,  {\rm div
}\varphi^{\e})
\right\} \theta(t) dt
\nonumber\\
+\int_{0}^{T}\left\{
 {1\over 2}( v^{\e}_{\delta}{\rm div\,} v^{\e}_{\delta}   \, , \,
\varphi^{\e})
+ {1\over 2} ( Z^{\e}_{\delta} {\rm div\,} v^{\e}_{\delta} \, , \, \psi^{\e})
-
 ({\cal F}({\bar{v}^{\e}}_{\delta} ) ,\Theta^{\e})- {\cal
R}(\bar{v}^{\e}_{\delta}  ,\Theta^{\e})
              \right\} \theta(t) dt
\nonumber\\
\quad \forall\Theta^{\e}= (\varphi^{\e} , \psi^{\e}) \in V^{\e}\times
{H^{1,}}^{\e}, \qquad
\end{eqnarray}
 that is $\bar{v}^{\e}_{\delta}$ satisfy (\ref{cal(P)})
in ${\cal D}'(0 , T)$ (distribution sense).
Moreover as the two items between the brackets $\{ \}$, in the
right hand side
of (\ref{cal(PPP1)}), are in $L^{4/3}(0 , T)$,
we deduce that (\ref{cal(P)})
holds for almost every  $t \in (0 , T)$.


\smallskip

In the following we set
\begin{eqnarray*}
p^{\e}_{\delta}= -{1\over \delta} {\rm div \, }  v^{\e}_{\delta},
\end{eqnarray*}
then, rewrite (\ref{cal(P)})  as follows
\begin{eqnarray}\label{cal(PP)}
[{\partial \bar{v}^{\e}_{\delta}\over \partial t} , \Theta^{\e}] +
 a(\bar{v}^{\e}_{\delta} , \Theta^{\e})
+B(\bar{v}^{\e}_{\delta} , \bar{v}^{\e}_{\delta} , \Theta^{\e})
+{1\over 2}\left\{
( v^{\e}_{\delta}{\rm div \,}  v^{\e}_{\delta} \, ,  \, \varphi^{\e})
+(Z^{\e}_{\delta}{\rm div \,} v^{\e}_{\delta} \, ,  \, \psi^{\e}) \right\}
\nonumber\\
-(p^{\e}_{\delta} \, , \,  {\rm div }\varphi^{\e})
= ({\cal F}(\bar{v}^{\e}_{\delta})  , \Theta^{\e})
- {\cal R}(\bar{v}^{\e}_{\delta} , \Theta^{\e}) \qquad
\forall\Theta^{\e}= (\varphi^{\e} , \psi^{\e})
\in V^{\e}\times {H^{1,}}^{\e}.
\end{eqnarray}
The aim now is to pass to the limit for $\delta\to 0$  in
(\ref{cal(PP)}).
Reminding that the different constants $C$ in (\ref{Maj1})-(\ref{Maj2}) and
(\ref{boudv})
 are independent of $\delta$, the same estimates hold for $\bar{v}^{\e}_{\delta}$
  i.e.
 \begin{eqnarray}\label{Maj1d}
 \sup_{t\in [0 , T]}[\bar{v}^{\e}_{\delta}(t)]^{2} \leq C,
\end{eqnarray}
\begin{eqnarray}\label{Maj2d}
 \int_{0}^{T}\|{\rm div\,} v^{\e}_{\delta}\|^{2} dt \leq C \delta, \quad
\int_{0}^{T} [[\bar{v}^{\e}_{\delta}(t)]]^{2}
dt\leq C,
 \end{eqnarray}
and
\begin{eqnarray}\label{boudvd}
\|{\partial {v^{\e}_{\delta }} \over\partial t}
\|_{L^{4/3}(0 , T ; H^{-1}(\Omega^{\e})^{2})}\leq C,
\quad
\|{\partial {Z^{\e}_{\delta }} \over  \partial t}
\|_{L^{4/3}(0 , T ; H^{-1}(\Omega^{\e}))}\leq C.
\end{eqnarray}
Hence,  there exists $\bar{v}^{\e}$  such that, possibly extracting a
subsequence still denoted by $\bar{v}^{\e}_{\delta}$:
\begin{eqnarray}\label{lim1}
 \bar{v}^{\e}_{\delta} \tow \bar{v}^{\e}  \quad \mbox{  in }
L^{2}(0 , T; V^{\e})\times L^{2}(0 , T; H^{1 , \e}) \quad \mbox{  weakly for }
\delta \to 0,
\end{eqnarray}
\begin{eqnarray}\label{lim2}
 \bar{v}^{\e}_{\delta} \tow \bar{v}^{\e} \quad \mbox{  in }
L^{\infty}(0 , T;  H^{\e})\times L^{\infty}(0 , T; H^{0 , \e}) \quad \mbox{weak
star for }  \delta \to 0,
\end{eqnarray}
\begin{eqnarray}\label{lim3}
{\rm div\, } v^{\e}_{\delta} \to 0
\quad \mbox{  in } L^{2}(0 , T; L^{2}(\Om^{\e})) \quad \mbox{strongly for }
\delta \to 0,
\end{eqnarray}
and
\begin{eqnarray}\label{cfVZ}
 \bar v^{\e}_{\delta} \to \bar v^{\e} \mbox{ strongly in }   L^{2}(0 , T;
(L^{4}(\Omega^{\e})^{2}))
\times
 L^{2}(0 , T ; L^4(\Omega^{\e})).
\end{eqnarray}
So from (\ref{lim1}) and (\ref{lim3}) we deduce
\begin{eqnarray}\label{div}
 {\rm div \, } v^{\e} = 0 \quad \mbox{\rm  in } \Om^{\e}, \quad \mbox{\rm a.e.
in } (0 , T).
\end{eqnarray}
We  check now  that $p^{\e}_{\delta}$  remains in a bounded subset of
$H^{-1}(0 , T ; L^{2}_0(\Om^{\e}))$. Reminding that $p^{\e}_{\delta} = -
\frac{1}{\delta} {\rm div \,} v^{\e}_{\delta}$, we have $p^{\e}_{\delta} \in L^2
(0,T; L^2_0 (\Om^{\e}))$.
Now let us consider $\omega\in H^{1}_{0}(0 , T ; L^{2}_{0}(\Om^{\e}))$,
 then (see \cite{JLLions78} page 13-15)
there exists  $\varphi\in H^1_0 (0 , T ;  H^{1}_{0}(\Om^{\e})^{2})$  such that
\begin{eqnarray*}\label{Phi}
&&{\rm div \, }\varphi(t)= \omega(t), \mbox{ and } \varphi(t)=
P\omega(t),  \nonumber\\
&& P \mbox{ is a linear continuous operator from } L^{2}_{0}(\Om^{\e})
\mbox{ to } H^{1}_{0}(\Om^{\e})^{2}.
\end{eqnarray*}
The  choice of $\Theta= (\varphi(t) , 0)$  in
(\ref{cal(PP)}), gives
 \begin{eqnarray}\label{equ2.36}
\int_{0}^{T}(p^{\e}_{\delta} \, , \,  \omega)dt  =
\int_{0}^{T}\left(- (v^{\e}_{\delta} , {\partial\varphi\over \partial t})
+ (\nu+\nu_{r})(\nabla v^{\e}_{\delta} , \nabla\varphi)\right)dt
+\int_{0}^{T}b(v^{\e}_{\delta} , v^{\e}_{\delta} ,  \varphi)dt
\nonumber\\
+{1\over 2}\int_{0}^{T}
(v^{\e}_{\delta}{\rm div \,} v^{\e}_{\delta} \, ,  \, \varphi) dt
-2\nu_{r}\int_{0}^{T} ({\rm rot \,}  Z^{\e}_{\delta} , \varphi) dt
- \int_{0}^{T}({\cal F}_{1}(v^{\e}_{\delta})  ,  \varphi) dt,
\end{eqnarray}
with
 \begin{eqnarray}
({\cal F}_{1}(v^{\e}_{\delta}) , \varphi ) &=&
-(\nu +\nu_{r}) ({\partial U^{\e}\over \partial z_{2}},
{\partial\varphi_{1}\over \partial z_{2}})
- b(U^{\e}e_{1} , v^{\e}_{\delta} , \varphi) - b(v^{\e}_{\delta} ,
U^{\e}e_{1} ,\varphi )
\nonumber\\&&
+2\nu_{r} ( {\partial W^{\e}\over \partial z_{2}}  , \varphi_{1})
-({\partial U^{\e} \over \partial t}, \varphi_{1}) + (f^{\e} ,
\varphi).
\end{eqnarray}
Since  $\omega\in H^{1}_{0}(0 , T ; L^{2}_{0}(\Om^{\e}))\subset L^{\infty}(0 , T
; L^{2}_{0}(\Om^{\e}))$,
with continuous injection,
it follows that  $\varphi$ in $L^{\infty}(0 , T ;
H^{1}_{0}(\Om^{\e})^{2})$,
and ${\partial \varphi\over\partial t} \in L^{2}(0 , T ;
H^{1}_{0}(\Om^{\e})^{2})$,
then also by the continuous injection of $H^{1}(\Om^{\e})$ in
$L^{4}(\Om^{\e})$ we have
\begin{eqnarray*}
 |\int_{0}^{T}b(v^{\e}_{\delta}, v^{\e}_{\delta}, \varphi) dt|&\leq&
\|v^{\e}_{\delta}\|_{L^{2}(0 , T; (L^{4}(\Om^{\e}))^{2})}
\|v^{\e}_{\delta}\|_{L^{2}(0 , T; H^{1} (\Om^{\e})^{2})}
\|\varphi\|_{L^{\infty}(0 , T; (L^{4}(\Om^{\e}))^{2})}
\nonumber\\
&\leq&
C^{2}\|v^{\e}_{\delta}\|^{2}_{L^{2}(0 , T
;H^{1}(\Om^{\e})^{2})}\|\varphi\|_{H^{1}(0 , T ;
H^{1}(\Om^{\e})^{2})}.
\end{eqnarray*}
Similarly for the first term in the second line of (\ref{equ2.36}).
Therefore using (\ref{Maj1d})-(\ref{Maj2d}) we  get
 \begin{eqnarray*}\label{equ2.37}
|\int_{0}^{T}(p^{\e}_{\delta} \, , \,  \omega)dt | \leq
C\|\varphi\|_{H^{1}(0 , T ; H^{1}(\Om^{\e})^{2})}
\quad \forall \varphi\in H^{1}_{0}(0 , T ; H^{1}_{0}(\Om^{\e})^{2}).
\end{eqnarray*}
As $P: \omega(t) \mapsto \varphi(t)$ is a linear continuous operator from
$L^{2}_{0}(\Om^{\e})$ to $H^{1}_{0}(\Om^{\e})^{2}$,   there exists another
constant $C$, independent of $\delta$,  such that
 \begin{eqnarray}\label{equ2.39}
|\int_{0}^{T}(p^{\e}_{\delta} \, , \,  \omega)dt | \leq
C\|\omega\|_{H^{1}_{0}(0 , T ; L^{2}_{0}(\Om^{\e}))}
\quad \forall \omega\in H^{1}(0 , T ; L^{2}(\Om^{\e})).
\end{eqnarray}
 Let us  take now $\omega\in H^{1}_{0}(0 , T ; L^{2}(\Om^{\e}))$ arbitrary, we
can
apply
(\ref{equ2.39}) to
$$\tilde{\omega}=\omega- {1\over meas(\Om^{\e})}\int_{\Om^{\e}}\omega dz$$
which is in $H^{1}_{0}(0 , T ; L^{2}_{0}(\Om^{\e}))$. But $p^{\e}_{\delta} \in
L^2(0,T; L^2_0 (\Om^{\e}))$, so
\begin{eqnarray*}
\int_{0}^{T}(p^{\e}_{\delta} \, , \,  \tilde \omega)dt =
\int_{0}^{T}(p^{\e}_{\delta} \, , \,  \omega)dt
\end{eqnarray*}
 and
(\ref{equ2.39})
remains valid for all $\omega\in H^{1}_{0}(0 , T ; L^{2}(\Om^{\e})$. Thus
$p^{\e}_{\delta}$ remains in a bounded subset of $H^{-1}(0 , T ;
L^{2}_0(\Om^{\e}))$.
It follows that there exists $p^{\e}\in H^{-1}(0 , T ; L^{2}_0(\Om^{\e}))$ such
that
 \begin{eqnarray}\label{lim5}
  p^{\e}_{\delta} \tow p^{\e} \quad \mbox{ in } H^{-1}(0 , T ; L^{2}(\Om^{\e}))
\mbox{ weak}.
 \end{eqnarray}


 In order to pass to the limit  as $\delta\to 0$,
let  $\theta\in {\cal D}(0 , T)$,
 multiply (\ref{cal(PP)}) by $\theta(t)$ and  integrate over $[0,T]$. We get
\begin{eqnarray}\label{cal(PPPs)}
-  \int_{0}^{T}(\bar{v}^{\e}_{\delta} (t), \Theta^{\e}) \theta'(t) dt
+
\int_{0}^{T}  \left(
 a(\bar{v}^{\e}_{\delta} , \Theta^{\e})
+B(\bar{v}^{\e}_{\delta} , \bar{v}^{\e}_{\delta} , \Theta^{\e})\right)
\theta(t) dt
-\int_{0}^{T} (p^{\e}_{\delta} \, , \,  {\rm div }\varphi^{\e})
\theta(t) dt
\nonumber\\
+{1\over 2}\int_{0}^{T}
\left\{
(v^{\e}_{\delta}{\rm div \,} v^{\e}_{\delta} \, ,  \, \varphi^{\e})
+(Z^{\e}_{\delta}{\rm div \,} v^{\e}_{\delta} \, ,  \, \psi^{\e}) \right\}
\theta(t) dt
= \int_{0}^{T}  \left\{  ({\cal F}(\bar{v}^{\e}_{\delta})  , \Theta^{\e})
- {\cal R}(\bar{v}^{\e}_{\delta} , \Theta^{\e}) \right\}
\theta(t) dt
\nonumber\\
  \qquad \forall\Theta^{\e}= (\varphi^{\e} , \psi^{\e})
\in V^{\e}\times {H^{1,}}^{\e}. \
\end{eqnarray}
Using (\ref{lim1}), (\ref{lim3}), (\ref{cfVZ}) and (\ref{lim5}), then taking
into
account
(\ref{nonltv})-(\ref{nonltZ}) for $v^{\e}_{\delta}$ and  $Z^{\e}_{\delta}$
instead of
 $v^{\e}_{\delta m}$ and  $Z^{\e}_{\delta m}$ for the nonlinear terms, we can
now pass to the limit  in
all the terms of (\ref{cal(PPPs)})  to get
\begin{eqnarray*}
\int_{0}^{T}(v^{\e}, \Theta^{\e}) \theta'(t) dt =
\int_{0}^{T}  \left\{
 a(\bar{v}^{\e} , \Theta^{\e}) +B(\bar{v}^{\e} , \bar{v}^{\e} ,
\Theta^{\e})
-(p^{\e} \, , \,  {\rm div }\varphi^{\e})
\right\} \theta(t) dt
\nonumber\\
-\int_{0}^{T}  \left\{ ({\cal F}(\bar{v}^{\e}),\Theta^{\e})- {\cal
R}(\bar{v}^{\e} ,\Theta^{\e})\right\} \theta(t) dt
  \qquad \forall\Theta^{\e}= (\varphi^{\e} , \psi^{\e}) \in
V^{\e}\times {H^{1,}}^{\e} ,
\end{eqnarray*}
 that is ($\bar{v}^{\e}, p^{\e}$) satisfy (\ref{eqn:er2.14})
 in ${\cal D}'(0 , T)$ (distribution sense). Moreover we can see also
that
(\ref{eqn:er2.14}) is satisfied for almost every  $t \in (0 , T)$.

\smallskip
Finally, by considering test-functions  $\Theta^{\e}\in V_{div}\times H^{1,
\e}$,  we can prove  the uniqueness of $(v^{\e}, Z^{\e})$
and its continuity in time
 as in Theorem 2.2 \cite{gl}. Thus the proof of the existence and uniqueness
of a solution of Problem $(P^{\e})$  is complete.
\end{proof}

\renewcommand{\theequation}{3.\arabic{equation}}
\setcounter{equation}{0}
\section{A priori uniform estimates of $\bar{v}^{\e}$ and
$p^{\e}$}\label{uniformEst}
The aim in this section is to establish  uniform
estimates with respect
to $\e$ for   $\bar{v}^{\e}$  and $p^{\e}$,
which will  allow
us  to derive in the next sections   the limit problem as $\e$ tends to zero by
using the
two-scale convergence technique.
 More precisely we consider first   the following  scaling
\begin{eqnarray}\label{2.2a}
 x_{1} = z_{1}, \quad\mbox{ and }\quad x_{2} ={z_{2}\over \e},
 \end{eqnarray}
 which  transforms the domain $\Om^{\e}$ into the domain
$$
\Om_{\e}=\bigl\{x=(x_{1} , x_{2})\in \br^{2} :\quad  0< x_{1}< L
 \qquad 0< x_{2}<  h^{\e}(x_{1})= h(x_{1}, {x_{1}\over \e}) \bigr\},$$
then we introduce a second  scaling
 \begin{eqnarray}\label{2.2b}
 y_{1} = x_{1}, \quad\mbox{ and }\quad  y_{2} ={x_{2}\over  h^{\e}(x_{1})}
=  {z_{2}\over \e h^{\e}(x_{1})}
 \end{eqnarray}
which transforms  the domain $\Om_{\e}$  into  $\Om=\{y=(y_{1} , y_{2}) \in
\Gamma_{0} \times (0 , 1)\}$.
With the chain rule,  we get easily the following relations
\begin{eqnarray}\label{not}
{\partial \over \partial z_{2}}= {1\over \e h^{\e}(y_{1})}{\partial \over
\partial y_{2}},
\quad
 {\partial \over \partial z_{1}}
&=&{\partial \over \partial y_{1}}{\partial y_{1}\over \partial z_{1}} +
 {\partial  \over \partial y_{2}}{\partial y_{2}\over \partial z_{1}}
= {\partial \over \partial y_{1}} + \left(-{y_{2}\over h^{\e}(y_{1})}
{\partial h^{\e}\over\partial y_{1}}\right)
{\partial  \over \partial y_{2}}
\nonumber\\
&=&
\left(1  , -{y_{2}\over h^{\e}(y_{1})}
{\partial h^{\e}\over\partial y_{1}}\right)
\left(\begin{array}[c]{c}
 {\partial \over\partial y_{1}} \\ \\
{\partial \over\partial y_{2}}
\end{array}
\right)
=b_{\e}\cdot\nabla.
\end{eqnarray}

Now we define the  functional setting in $\Om$: let
$\Gamma_{1}=
\{(y_{1} ,y_{2})\in \overline{\Om} :  \quad y_{2}=1\}$ and
$$\tilde{V}= \{ v \in {\cal C}^{\infty}({\overline \Om})^{2}:
 \,\, v \,\,\mbox{\rm is  L-periodic in }\,\, y_1,\,\,\,  v_{|_{\Gamma_0}} =0,
\,
    -\e (h^{\e})'(y_{1}) v_{1} + v_{2}  =0
\mbox{ on } \Gamma_{1}\}$$
$$V = {\rm closure \,\,  of\,\, } \tilde{V} \,\,
 {\rm  in }  \,\, H^{1}(\Om)\times  H^{1}(\Om)$$
$$\tilde{H^{1}}= \{Z \in {\cal C}^{\infty}(\overline{\Om}): \,\,\,Z
    \,\,\,\mbox{\rm is  L-periodic in }\,\,\, y_1,\,\,\,Z=0 \,\,\,{\rm
on}\,\,\,
    \Gamma_0\cup\Gamma_1\}$$
$$H= {\rm  closure \,\,  of   \, } \tilde{V}  \,\, {\rm in }\,\,
  L^{2}(\Om) \times  L^{2}(\Om), \qquad
H^{1} = {\rm closure \,\,  of\,\, } \tilde{H^{1}} \,\, {\rm  in }
 \,\, H^{1}(\Om),
 $$
$${H^{0}} = {\rm  closure \,\,  of   \, } \tilde{H^{1}}\,\,
{\rm in }\,\,  L^{2}(\Om).  $$

In order to avoid new notations,  we have still denoted by $ v^{\e}$, $Z^{\e}$
and $p^{\e}$ the unknown velocity, micro-rotation and pressure fields as
functions of the rescaled variables $(y_1, y_2)$ instead of $(z_1, z_2)$.
Similarly, we still denote the data by $\bar f^{\e}$ and $\bar \xi^{\e}$
considered now as functions of $(y_1, y_2)$.

Let $\Theta = (\varphi, \psi) \in V \times H^1$ and let $\Theta^{\e} =
(\varphi^{\e}, \psi^{\e}) \in V^{\e} \times H^{1,\e} $ be given by
\begin{eqnarray*}
\varphi^{\e} (z_1, z_2) = \varphi \left( z_1, \frac{z_2}{\e h^{\e} (z_1)}
\right), \quad \psi^{\e} (z_1, z_2) = \psi \left( z_1, \frac{z_2}{\e h^{\e}
(z_1)} \right) \quad \forall (z_1, z_2) \in \Om^{\e}.
\end{eqnarray*}
Using (\ref{not}) we obtain that
\begin{eqnarray} \label{fla}
a(\bar{v}^{\e}(t),\Theta^{\e}) &=&
(\nu + \nu_{r})
\sum_{i, j=1}^{2}\int_{\Om^{\e}}
             {\partial v_{i}^{\e}(t)\over \partial z_{j}}
             {\partial \varphi_{i}^{\e} (t)\over \partial z_{j}}dz
  +\aal\sum_{i=1}^{2} \int_{\Om^{\e}}
{\partial Z^{\e}(t)\over \partial z_{i}}
             {\partial \psi^{\e} \over \partial z_{i}}dz
\nonumber\\
&=&  (\nu + \nu_{r})
\int_{\Om}
\sum_{i=1}^{2} \left((b_{\e}\cdot \nabla v_{i}^{\e}(t))(b_{\e}\cdot \nabla
\varphi_{i})
+ {1\over (\e h^{\e})^{2}}{\partial   v^{\e}_{i}(t)\over \partial y_{2}}
{\partial \varphi_{i}\over \partial y_{2}}\right) \e h^{\e} dy
 \nonumber\\
&&+\aal\int_{\Om}
\left((b_{\e}\cdot \nabla Z^{\e}(t))(b_{\e}\cdot \nabla \psi) +
 {1\over (\e h^{\e})^{2}}{\partial   Z^{\e}(t)\over \partial y_{2}}
{\partial \psi\over \partial y_{2}}\right) \e h^{\e} dy
\nonumber\\
&=&  {(\nu + \nu_{r})\over \e}
\int_{\Om}
\sum_{i=1}^{2} \left((\e b_{\e}\cdot \nabla v_{i}^{\e}(t))(\e b_{\e}\cdot \nabla
\varphi_{i})
 + {1\over (h^{\e})^{2}}{\partial v^{\e}_{i}(t)\over \partial y_{2}}
{\partial \varphi_{i}\over
\partial y_{2}}\right) h^{\e} dy
 \nonumber\\
&&+{\aal\over \e}\int_{\Om}
\left((\e b_{\e}\cdot \nabla Z^{\e}(t))(\e b_{\e}\cdot \nabla \psi) +
 {1\over (h^{\e})^{2}}{\partial   Z^{\e}(t)\over \partial y_{2}}
{\partial\psi\over \partial y_{2}}\right)  h^{\e} dy
={1\over \e}\hat{a}(\bar{v}^{\e}(t),\Theta),
\end{eqnarray}
\begin{eqnarray} \label{trifb}
B(\bar{v}^{\e}(t),\bar{v}^{\e}(t),\Theta^{\e}) &=& b(v^{\e}(t), v^{\e}(t),
\varphi^{\e}) +
b_{1}(v^{\e}(t), Z^{\e}(t),  \psi^{\e}) \nonumber\\
&=&\int_{\Om^{\e}} \sum_{i , j =1}^{2} v^{\e}_{i}(t){\partial v^{\e}_{j}(t)\over
\partial z_{i}} \varphi_{j}^{\e} dz
 + \sum_{i =1}^{2}\int_{\Om^{\e}} v^{\e}_{i}(t){\partial Z^{\e}(t)\over \partial
z_{i}} \psi^{\e} dz
 \nonumber\\
&=& \int_{\Om}\left(\sum_{j =1}^{2} v^{\e}_{1}(t)
\left(\e b_{\e}\cdot \nabla v^{\e}_{j}(t)   \right)\varphi_{j} +
  {v^{\e}_{2}(t)\over h^{\e}}{\partial v^{\e}_{j}(t) \over \partial y_{2}}
\varphi_{j}\right)  h^{\e} dy
 \nonumber\\
&+&
\int_{\Om}  \left(\sum_{j =1}^{2}v^{\e}_{1}(t)(\e b_{\e}\cdot \nabla Z^{\e}(t))
\psi   +
 {v^{\e}_{2}(t)\over h^{\e}} {\partial Z^{\e}(t)\over \partial y_{2}} \psi
\right) h^{\e}  dy
 \nonumber\\
&=&
\hat{B}(\bar{v}^{\e}(t),\bar{v}^{\e}(t),\Theta),
\end{eqnarray}
\begin{eqnarray}\label{Rot}
{\cal R}(\bar{v}^{\e}(t) , \Theta^{\e})
&=& -2\nu_{r}
\int_{\Om^{\e}}\left( {\partial Z^{\e}(t)\over\partial z_{2}}\varphi_{1}^{\e}
                    - {\partial Z^{\e}(t)\over\partial z_{1}}\varphi_{2}^{\e}\right)
+\left({\partial v_{2}^{\e}(t)\over\partial z_{1}}- {\partial
v_{1}^{\e}(t)\over\partial z_{2}}\right)\psi^{\e}  dz
\nonumber\\
&& +4\nu_{r}  \int_{\Om^{\e}}  Z^{\e}(t)\psi^{\e}  dz
= -2\nu_{r} \int_{\Om} \left({1\over  h^{\e} }
{\partial Z^{\e}(t)\over \partial y_{2}}\varphi_{1} -
(\e b_{\e}\cdot \nabla Z^{\e}(t)) \varphi_{2}\right) h^{\e} dy
\nonumber\\
&&-2\nu_{r} \int_{\Om} \left(
(\e b_{\e}\cdot \nabla v^{\e}_{2}(t))
- {1\over  h^{\e}} {\partial v_{1}^{\e}(t)\over \partial y_{2}}\right)\psi h^{\e} dy
+ 4\nu_{r}\e \int_{\Om} Z^{\e}(t)\psi \,  h^{\e}  dy
\nonumber\\
&=& \hat{{\cal R}}(\bar{v}^{\e}(t) , \Theta).
\end{eqnarray}
Using Lemma \ref{lUW} we have
$U^{\e}(t , z_{2})= {\cal U}({z_{2}\over \e}) U_{0}(t)
= {\cal U}(y_{2}h^{\e}(y_{1})) U_{0}(t)$,
so
 \begin{eqnarray}\label{For3.8}
b_{\e}\cdot \nabla  {\cal U}(y_{2}h^{\e}(y_{1}))&=&
\left( {\partial \over \partial  y_{1}}
-{y_{2}\over h^{\e}}  {\partial h^{\e}\over \partial  y_{1}} {\partial
\over \partial y_{2}}  \right)
{\cal U}(y_{2}h^{\e}(y_{1}))
=  {\cal U}'(y_{2}h^{\e}(y_{1})) \left(
 y_{2} {\partial h^{\e}\over \partial  y_{1}}
 - {y_{2}\over h^{\e} } {\partial h^{\e}\over \partial y_{1}}
h^{\e}
 \right)
 \nonumber\\
&=& 0,
 \end{eqnarray}
and similarly  for
$W^{\e}(t , z_{2})= {\cal W}({z_{2}\over \e}) W_{0}(t)
= {\cal W}(y_{2}h^{\e}(y_{1}))W_{0}(t)$, so
\begin{eqnarray}\label{For3.9}
 b_{\e}\cdot \nabla  {\cal W}(y_{2}h^{\e}(y_{1}))= 0.
 \end{eqnarray}
 Then
\begin{eqnarray}\label{For3.10}
a(\bar{\xi^{\e}},\Theta^{\e}) &=& (\nu + \nu_{r}) \int_{\Om^{\e}} \nabla
U^{\e}(z_{2},t)e_{1}\nabla\varphi^{\e} dz + \aal \int_{\Om^{\e}} \nabla
W^{\e}(z_{2},t) \nabla \psi^{\e} dz \nonumber\\
&=&  (\nu + \nu_{r}) \int_{\Om^{\e}} \sum_{i=1}^{2}{\partial U^{\e}\over
\partial z_{i}} {\partial \varphi_{1}^{\e}\over \partial z_{i}} dz
 + \aal \int_{\Om^{\e}}  \sum_{i=1}^{2}{\partial W^{\e}\over \partial z_{i}}
{\partial \psi^{\e}\over \partial z_{i}} dz
 \nonumber\\
&=&(\nu + \nu_{r})U_{0}(t) \int_{\Om} \left( (b_{\e}\cdot \nabla {\cal U})
(b_{\e}\cdot \nabla \varphi_{1}) +
               {1\over (\e h^{\e})^{2}} {\cal U}'(
y_{2}h^{\e}) h^{\e} {\partial \varphi_{1}\over \partial y_{2}}
\right) \e  h^{\e} dy
\nonumber\\
&&+ \aal W_{0}(t)\int_{\Om} \left( (b_{\e}\cdot \nabla {\cal W}) (b_{\e}\cdot
\nabla \psi) +
               {1\over (\e h^{\e})^{2}} {\cal W}'( y_{2}h^{\e}) h^{\e} {\partial
\psi\over \partial y_{2}}
\right)  \e h^{\e} dy
 \nonumber\\
&=&
{(\nu + \nu_{r})\over \e} U_{0}(t) \int_{\Om}
              {\cal U}'( y_{2}h^{\e}) {\partial \varphi_{1}\over
\partial y_{2}}  dy
+ {\aal \over \e} W_{0}(t)\int_{\Om}
                 {\cal W}'( y_{2}h^{\e}) {\partial \psi\over \partial
y_{2}} dy
= {1\over  \e}\hat{a}(\bar{\xi^{\e}},\Theta).
\end{eqnarray}
We have also
\begin{eqnarray} \label{Rot1}
B(\bar{\xi^{\e}},\bar{v}^{\e},\Theta^{\e})
&=& b(U^{\e}e_{1}, v^{\e}, \varphi^{\e}) + b_{1}(U^{\e}e_{1} ,  Z^{\e},
\psi^{\e})
\nonumber\\
&=&\int_{\Om^{\e}} \sum_{j =1}^{2} U^{\e}{\partial v^{\e}_{j}\over \partial
z_{1}} \varphi_{j}^{\e} dz
+ \int_{\Om^{\e}}  U^{\e}{\partial Z^{\e}\over \partial z_{1}} \psi^{\e} dz
 \nonumber\\
&=& U_{0}(t) \int_{\Om} {\cal U}( y_{2}h^{\e}) \left(
\sum_{j =1}^{2}
(\e b_{\e}\cdot \nabla v^{\e}_{j})\varphi_{j}
+
(\e b_{\e}\cdot \nabla  Z^{\e}) \psi\right)  h^{\e}  dy
=\hat{B}(\bar{\xi^{\e}},\bar{v}^{\e},\Theta),
\end{eqnarray}
\begin{eqnarray}\label{Rott}
 B(\bar{v}^{\e},\bar{\xi^{\e}}, \Theta^{\e})&=& -B(\bar{v}^{\e},
\Theta^{\e},\bar{\xi^{\e}})
= -b(v^{\e},  \varphi^{\e}, U^{\e}e_{1}) - b_{1}(v^{\e} ,  \psi^{\e}, W^{\e})
\nonumber\\
&=&-\sum_{i=1}^{2} \int_{\Om^{\e}}
v^{\e}_{i}{\partial\varphi_{1}^{\e}\over \partial z_{i}} U^{\e} dz -
\sum_{i=1}^{2} \int_{\Om^{\e}}
v^{\e}_{i}{\partial\psi^{\e}\over \partial z_{i}} W^{\e} dz
 \nonumber\\
&=& - U_{0}(t) \int_{\Om}  {\cal U}( y_{2}h^{\e})\left(
v^{\e}_{1} (\e b_{\e}\cdot \nabla \varphi_{1})  h^{\e}   +
v^{\e}_{2} {\partial \varphi_{1}\over \partial y_{2}} \right)  dy
 \nonumber\\
&&-  W_{0}(t) \int_{\Om} {\cal W}( y_{2}h^{\e})\left(
v^{\e}_{1} (\e b_{\e}\cdot \nabla \psi)  h^{\e}   +
v^{\e}_{2}  {\partial  \psi\over \partial y_{2}} \right)
 dy
=\hat{B}(\bar{v}^{\e},\bar{\xi^{\e}}, \Theta),
\end{eqnarray}
\begin{eqnarray} \label{Rotxi}
 {\cal R}(\bar{\xi^{\e}} , \Theta^{\e})
&=& -2\nu_{r} W_{0}(t) \int_{\Om} \left(\varphi_{1} {1\over \e h^{\e} }
{\partial \over \partial y_{2}} {\cal W}(y_{2}h^{\e})  -  (b_{\e}\cdot \nabla
{\cal W}) \varphi_{2}\right)\e h^{\e}  dy
\nonumber\\
&&-2\nu_{r} U_{0}(t)\int_{\Om} \left(- {1\over \e h^{\e} }{\partial\over
\partial y_{2}}{\cal U}(y_{2}h^{\e})\right)\psi \e h^{\e} dy
+ 4\nu_{r}  W_{0}(t) \int_{\Om}   {\cal W}(y_{2}h^{\e})  \psi \, \e
h^{\e} dy\nonumber\\
&=& -2\nu_{r}  W_{0}(t)
 \int_{\Om} {\cal W}'(y_{2}h^{\e})\varphi_{1} h^{\e} dy+
2 \nu_r U_{0}(t)\int_{\Om} {\cal U}'(y_{2}h^{\e}) \psi  h^{\e} dy
\nonumber\\
&&+
4\nu_{r}  W_{0}(t) \int_{\Om}   {\cal W}(y_{2}h^{\e})  \psi \, \e
h^{\e}  dy
=\hat{{\cal R}}(\bar{\xi^{\e}} , \Theta),
\end{eqnarray}
 and
\begin{eqnarray}\label{pression}
\qquad (p^{\e}(t) , {\rm div\,} \varphi^{\e} )= \int_{\Om^{\e}} p^{\e}(t){\rm div\,}
\varphi^{\e} dz
= \int_{\Om} p^{\e}(t) \left(
(\e b_{\e}\cdot \nabla \varphi_{1}) + {1\over h^{\e}}{\partial
\varphi_{2} \over \partial y_{2}}
                          \right)  h^{\e}  dy.
\end{eqnarray}

\smallskip

\begin{lemma} \label{lemma3.1}
Using {\rm(\ref{2.2a})-(\ref{2.2b})}, the variational identity
{\rm(\ref{eqn:er2.14})}
written in   $\Om^{\e}$ leads to  the following one  in $\Om$:
 \begin{eqnarray} \label{eqvar}
\e
\int_{\Om} {d\bar{v}^{\e}\over dt}(t)\Theta^{\e} h^{\e} dy + {1\over
\e}\hat{a}(\bar{v}^{\e}(t), \Theta^{\e})
      + \hat{B}(\bar{v}^{\e}(t), \bar{v}^{\e}(t), \Theta^{\e}) +
\hat{{\cal R}}(\bar{v}^{\e}(t), \Theta^{\e})
= -\e
\int_{\Om} {d\bar{\xi}^{\e}\over dt}(t)\Theta^{\e} h^{\e} dy
\nonumber\\
 -{1\over
\e}\hat{a}(\bar{\xi}^{\e}(t) , \Theta^{\e}) -\hat{B}(\bar{\xi}^{\e}(t)
, \bar{v}^{\e}(t), \Theta^{\e})
-\hat{B}(\bar{v}^{\e}(t),  \bar{\xi}^{\e}(t) ,  \Theta^{\e}) -
\hat{{\cal R}}(\bar{\xi}^{\e}(t) ,\Theta^{\e})
+ \e \int_{\Om} \bar{f}^{\e}(t) \Theta^{\e} h^{\e} dy
\nonumber\\
+
\int_{\Om} p^{\e}(t) \left( (\e b_{\e}\cdot \nabla \varphi^{\e}_{1})
                + {1\over h^{\e}}{\partial \varphi^{\e}_{2} \over \partial
y_{2}} \right) h^{\e}  dy,
\quad \forall \Theta^{\e}=(\varphi^{\e} , \psi^{\e})\in V^{\e}\times H^{1 , \e},
\end{eqnarray}
where $\hat{a}$, $\hat{B}$, and  $\hat{{\cal R}}$, are
 defined   by
 \rm(\ref{fla}), \rm(\ref{trifb}), and \rm(\ref{Rot}) respectively.
\end{lemma}
 \begin{proof}
Indeed, from (\ref{fla})-(\ref{pression}), the variational identity
(\ref{eqvar}) follows.
 \end{proof}

\bigskip

We prove now  the following uniform estimates, with respect to $\e$:

 \begin{proposition}\label{pro1}
Assume that $\e^{2}\bar{f}^{\e}$ and $\e \bar v^{\e}_0$ are  bounded
independently of $\e$ in $(L^2( (0,T) \times \Om))^3$ and in $(L^2(\Om))^3$
respectively and
 $U_{0}  \in H^{1}(0, T)$,  $W_{0} \in H^{1}(0, T)$.
There exists a constant $C>0$ which does not depends on  $\e$,
such that, for $i=1,2$, we have  the following
estimates:
\begin{eqnarray}\label{E3.13}
 \|(\e b_{\e}\cdot \nabla v^{\e}_{i})\|_{L^{2}((0 , T)\times \Om)}\leq C,\qquad
 \|(\e b_{\e}\cdot \nabla Z^{\e})\|_{L^{2}((0 , T)\times \Om)}\leq C,
\end{eqnarray}
 \begin{eqnarray}\label{E3.14}
  \|{\partial v_{i}^{\e}\over \partial y_{2}}\|_{L^{2}((0 , T)\times \Om)}\leq
C, \qquad
 \|{\partial Z^{\e}\over \partial y_{2}}\|_{L^{2}((0 , T)\times \Om)}\leq C,
 \end{eqnarray}
 \begin{eqnarray}\label{E3.15}
  \|{\partial v_{i}^{\e}\over \partial y_{1}}\|_{L^{2}((0 , T)\times \Om)}\leq
{C\over \e}, \qquad
 \|{\partial Z^{\e}\over \partial y_{1}}\|_{L^{2}((0 , T)\times \Om)}\leq
{C\over \e},
 \end{eqnarray}
\begin{eqnarray}\label{E3.16}
  \|v_{i}^{\e}\|_{L^{2}((0 , T)\times \Om)}\leq C,
\qquad \|Z^{\e}\|_{L^{2}((0 , T)\times \Om)}\leq C.
\end{eqnarray}
 \end{proposition}
 \begin{proof}
Taking $\Theta^{\e}= \bar{v}^{\e}(t)$  in (\ref{eqvar}), and observing that
$B (\bar{v}^{\e}(t), \bar{v}^{\e}(t), \bar{v}^{\e}(t))=
 B (\bar{\xi^{\e}}(t), \bar{v}^{\e}(t), \bar{v}^{\e}(t))=0$,
  we  obtain
\begin{eqnarray}\label{eq117}
\e{d\over 2dt} \int_{\Om}(\bar{v}^{\e}(t))^{2} h^{\e} dy
+  {(\nu + \nu_{r})\over \e} \int_{\Om}(\e\, b_{\e}\cdot \nabla
v_{1}^{\e}(t))^{2} h^{\e} dy
+ {(\nu + \nu_{r})\over \e}\int_{\Om}(\e\, b_{\e}\cdot \nabla v_{2}^{\e}(t))^{2}
h^{\e} dy
+
\nonumber\\
+ {(\nu + \nu_{r})\over \e}\left(
\int_{\Om}\left({1\over h^{\e}}{\partial v^{\e}_{1}(t)\over \partial
y_{2}}\right)^{2} h^{\e} dy
+
\int_{\Om}\left({1\over h^{\e}}{\partial v^{\e}_{2}(t)\over \partial
y_{2}}\right)^{2} h^{\e} dy\right)
+{\aal \over \e} \int_{\Om}(\e\, b_{\e}\cdot \nabla Z^{\e}(t))^{2} h^{\e} dy+
\nonumber\\
+ {\aal \over \e} \int_{\Om}
\left({1\over h^{\e}}{\partial Z^{\e}(t)\over \partial y_{2}}\right)^{2}
h^{\e}  dy
+ 4\nu_{r} \e \int_{\Om}(Z^{\e}(t))^{2} h^{\e} dy
= 2\nu_{r} \int_{\Om} \left( {\partial Z^{\e}(t)\over \partial
y_{2}}v^{\e}_{1}(t)
 -(\e b_{\e}\cdot \nabla Z^{\e}(t)) v^{\e}_{2}(t) h^{\e} \right) dy
\nonumber\\
+2\nu_{r} \int_{\Om} \left(  (\e b_{\e}\cdot \nabla v^{\e}_{2}(t))Z^{\e}(t)
h^{\e}
-  {\partial v_{1}^{\e}(t)\over \partial y_{2}} Z^{\e}(t) \right)   dy
-{(\nu + \nu_{r})\over \e}  U_{0} (t) \int_{\Om} {\cal U}'(y_{2}h^{\e})
    {\partial v_{1}^{\e}(t)\over \partial y_{2}}  dy
\nonumber\\
-{\aal\over \e} W_{0}(t)\int_{\Om} {\cal W}'(y_{2}h^{\e}) {\partial
Z^{\e}(t)\over\partial y_{2}}  dy
+2\nu_{r}W_{0}(t)  \int_{\Om} {\cal W}'(y_{2} h^{\e}) v_{1}^{\e}(t)h^{\e}  dy
- 2\nu_{r} U_{0}(t) \int_{\Om}{\cal U}'(y_{2} h^{\e}) Z^{\e}(t) h^{\e}  dy
 \nonumber\\
  - 4\nu_{r}\e W_{0}(t) \int_{\Om} {\cal W}(y_{2} h^{\e}) Z^{\e}(t)  h^{\e} dy
+  U_{0}(t)\int_{\Om}  {\cal U}(y_{2} h^{\e}) \left(
      v^{\e}_{1}(t) (\e b_{\e}\cdot \nabla v^{\e}_{1}(t))   h^{\e}
  +  v^{\e}_{2}(t)  {\partial v^{\e}_{1}(t)\over \partial y_{2}}\right) dy
   \nonumber\\
 + W_{0}(t) \int_{\Om} \left(  v^{\e}_{1}(t) (\e b_{\e}\cdot \nabla Z^{\e}(t))
 {\cal
W}(y_{2} h^{\e})  h^{\e}
   +  v^{\e}_{2}(t) {\partial  Z^{\e}(t)\over \partial y_{2}} {\cal W}(y_{2}
h^{\e}) \right)  dy
  \nonumber\\
 -\e  U'_{0}(t)\int_{\Om}{\cal U}(y_{2} h^{\e}) v^{\e}_{1}(t)  h^{\e} dy
-\e  W'_{0}(t)\int_{\Om}{\cal W}(y_{2} h^{\e}) Z^{\e}(t)  h^{\e} dy
  \nonumber\\
+ \e \int_{\Om}g^{\e}(t)Z^{\e}(t) h^{\e}  dy
+ \e \int_{\Om} (f^{\e}_{1}(t)v^{\e}_{1}(t)+ f^{\e}_{2}(t)v^{\e}_{2}(t))  h^{\e}dy.
\end{eqnarray}
Now we estimate the right hand side  of the above inequality
(\ref{eq117}). \newline
Let  $\lambda_{j}$ for $1\leq j\leq 16$, which must be some strictly
positive constants, such that

\begin{eqnarray}\label{es1}
2\nu_{r} \left|\int_{\Om} {\partial Z^{\e}(t)\over \partial y_{2}} v^{\e}_{1}(t)
dy\right|
\leq {\nu_{r} \over \e \lambda_{1}}\|{1\over  h^{\e} }
                  {\partial Z^{\e}(t)\over \partial y_{2}}\|^{2}_{L^{2}(\Om)}
+ \e \nu_{r}\lambda_{1}h^{2}_{M}\|v^{\e}_{1}(t)\|^{2}_{L^{2}(\Om)},
\end{eqnarray}
\begin{eqnarray}\label{es2}
 2\nu_{r} \left|\int_{\Om} (\e b_{\e}\cdot \nabla Z^{\e}(t)) \, v^{\e}_{2}(t)
h^{\e}  dy\right|
\leq {\nu_{r}\over \e \lambda_{2}}\|(\e b_{\e}\cdot \nabla
Z^{\e}(t)\|^{2}_{L^{2}(\Om)}
\nonumber\\
+ \nu_{r}\e \lambda_{2}h^{2}_{M}\|v^{\e}_{2}(t)\|^{2}_{L^{2}(\Om)},
\end{eqnarray}
\begin{eqnarray}\label{es3}
2\nu_{r}\left| \int_{\Om}   (\e b_{\e}\cdot \nabla v^{\e}_{2}(t))\,
Z^{\e}(t)h^{\e}  dy\right|
\leq {\nu_{r}\over \e \lambda_{3}}
\|(\e b_{\e}\cdot \nabla v^{\e}_{2}(t))\|^{2}_{L^{2}(\Om)}
\nonumber\\
 +\e \lambda_{3}\nu_{r}h^{2}_{M} \|Z^{\e}(t)\|^{2}_{L^{2}(\Om)},
\end{eqnarray}
\begin{eqnarray}\label{es4}
2\nu_{r} \left|\int_{\Om} {\partial v_{1}^{\e}(t)\over \partial y_{2}}
\, Z^{\e}(t) dy \right|
\leq   {\nu_{r}\over \e \lambda_{4}}
\|{1\over h^{\e}}{\partial v_{1}^{\e}(t)\over \partial y_{2}}\|^{2}_{L^{2}(\Om)}
+ \nu_{r}\e \lambda_{4}h^{2}_{M}\|Z^{\e}(t)\|^{2}_{L^{2}(\Om)}
\end{eqnarray}
\begin{eqnarray}\label{es5}
 {1\over \e}  \left|\int_{\Om}  U_{0}(t){\cal U}'(y_{2}h^{\e})
{\partial v_{1}^{\e}(t)\over \partial y_{2}} dy
\right|
\leq  {U_{0}(t)^{2}\|{\cal U}'\|^{2}_{L^{2}(0, h_m )}h_{M} L \over 2\e
\lambda_{5}(\nu +\nu_{r})}
\nonumber\\
+ {\lambda_{5}(\nu +\nu_{r})\over 2\e}
\|{1\over h^{\e}} {\partial v_{1}^{\e}(t)\over \partial
y_{2}}|^{2}_{L^{2}(\Om)},
\end{eqnarray}
\begin{eqnarray}\label{es6}
{\aal\over \e} \left|
\int_{\Om}W_{0}(t){\cal W}'(t , y_{2}h^{\e}) {\partial Z^{\e}(t)\over\partial
y_{2}}  h^{\e} dy\right|
\leq
{\aal W_{0}(t)^{2}\|{\cal W}'\|^{2}_{L^{2}(0, h_m )}h_{M} L\over
2\lambda_{6}\e}
\nonumber\\+ {\lambda_{6}\aal\over 2\e}
\|{1\over h^{\e}}{\partial Z^{\e}(t)\over \partial y_{2}}\|^{2}_{L^{2}(\Om)}
 \end{eqnarray}
\begin{eqnarray}\label{es11}
2\nu_{r}\left| W_{0}(t)  \int_{\Om}{\cal W}'(y_{2} h^{\e}) v_{1}^{\e}(t)  h^{\e}
dy\right|
\leq {\nu_{r}W_{0}(t)^{2}\|{\cal W}'\|^{2}_{L^{2}(0, h_m )}L\over \e \lambda_{7}
h_M}
\nonumber\\+\nu_{r} \lambda_{7} \e  h^{2}_{M}\|v_{1}^{\e}(t)\|^{2}_{L^{2}(\Om)},
 \end{eqnarray}
\begin{eqnarray}\label{es12}
2\nu_{r}\left|U_{0}(t) \int_{\Om} {\cal U}'(y_{2} h^{\e}) Z^{\e}(t) h^{\e}
dy\right|
\leq {\nu_{r}U_{0}(t)^{2} \|{\cal U}'\|^{2}_{L^{2}(0, h_m)} h_{M} L
 \over \e \lambda_{8}}
\nonumber\\
+\nu_{r} \lambda_{8} \e \|Z^{\e}(t))^{2}\|^{2}_{L^{2}(\Om)},
 \end{eqnarray}
\begin{eqnarray}\label{es14}
\left|U_{0}(t)\int_{\Om}  v^{\e}_{1}(t) (\e b_{\e}\cdot \nabla v^{\e}_{1}(t))
{\cal U}(y_{2}h^{\e})  h^{\e} dy\right|
&\leq& {\e U^{2}_{0}(t) \|{\cal U}\|^{2}_{\infty} h^{2}_{M}
\lambda_{9}\over 2 }
 \|v^{\e}_{1}(t)\|^{2}_{L^{2}(\Om)}
 \nonumber\\
&&
+ {1\over 2\e\lambda_{9}} \|(\e b_{\e}\cdot \nabla
v^{\e}_{1})\|^{2}_{L^{2}(\Om)},
\end{eqnarray}

\begin{eqnarray}\label{es15}
\left|U_{0}(t)\int_{\Om} v^{\e}_{2}(t) {\partial v^{\e}_{1}\over \partial y_{2}}
{\cal U}(y_{2}h^{\e})  dy\right|
 &\leq& {1\over 2\e \lambda_{10}}
\|{1 \over h^{\e}}{\partial v^{\e}_{1}\over \partial y_{2}}\|^{2}_{L^{2}(\Om)}
\nonumber\\
&&
+
{\e\lambda_{10}U^{2}_{0}(t)\|{\cal  U}\|^{2}_{\infty}h^{2}_{M} \over 2}
\|v^{\e}_{2}(t)\|^{2}_{L^{2}(\Om)},
\end{eqnarray}
\begin{eqnarray}\label{es16}
\left|W_{0}(t)\int_{\Om} v^{\e}_{1}(t) (\e b_{\e}\cdot \nabla Z^{\e}) {\cal
W}(y_{2}h^{\e})
h^{\e} dy\right|
&\leq& {1\over 2\e  \lambda_{11}} \|(\e b_{\e}\cdot \nabla
Z^{\e})\|^{2}_{L^{2}(\Om)}
\nonumber\\
&&
+{\e  \lambda_{11}W^{2}_{0}(t)\|{\cal  W}\|^{2}_{\infty}h^{2}_{M} \over 2}
\|v^{\e}_{1}(t)\|^{2}_{L^{2}(\Om)},
\end{eqnarray}
\begin{eqnarray}\label{es17}
\left|W_{0}(t)\int_{\Om}  v^{\e}_{2}(t) {\partial Z^{\e}\over \partial y_{2}}
 {\cal  W}(y_{2}h^{\e}) dy
\right|
&\leq&  {1\over 2\e  \lambda_{12}}
\|{1\over h^{\e}} {\partial  Z^{\e}\over \partial y_{2}} \|^{2}_{L^{2}(\Om)}
\nonumber\\
&&+{\e  \lambda_{12}W^{2}_{0}(t)\|{\cal  W}\|^{2}_{\infty}h^{2}_{M}
\over 2}\|v^{\e}_{2}(t)\|^{2}_{L^{2}(\Om)},
\end{eqnarray}
\begin{eqnarray}\label{es131}
4\nu_{r}\e\left|W_{0}(t) \int_{\Om}  {\cal  W}(y_{2}h^{\e}) Z^{\e}(t)  h^{\e}
dy\right|
\leq  {2\nu_{r}\e W_{0}(t)^{2}  \|{\cal  W}\|^{2}_{L^{2}(0,h_m)}h_{M} L
\over \lambda_{13}}
\nonumber \\
+ 2 \e \lambda_{13}\nu_{r} \|Z^{\e}(t)\|^{2}_{L^{2}(\Om)},
 \end{eqnarray}
\begin{eqnarray}\label{es13a}
\e \left| U'_{0}(t)\int_{\Om} {\cal U}(y_{2}h^{\e})v^{\e}_{1}(t)h^{\e}
dy\right|
\leq \frac{\e}{2}  h^{2}_{M}\|v^{\e}_{1}(t)\|^{2}_{L^{2}(\Om)}
\nonumber \\
+ \frac{\e L}{2
h_M} |U'_{0}(t)|^{2}
\|{\cal  U}\|^{2}_{L^{2}(0, h_m)},
 \end{eqnarray}
\begin{eqnarray}\label{es13aa}
\e |W'_{0}(t)\int_{\Om} {\cal W}(y_{2}h^{\e}) Z^{\e}(t) h^{\e} dy|
 \leq \frac{\e}{2}  h^{2}_{M}\|Z^{\e}(t)\|^{2}_{L^{2}(\Om)} \nonumber \\+ \frac{\e L}{2 h_M}
|W'_{0}(t)|^{2}\|
  \|{\cal  W}\|^{2} _{L^{2}(0, h_m)}.
 \end{eqnarray}
Finally
\begin{eqnarray*}
\e \left|\int_{\Om}g^{\e}(t)Z^{\e}(t) h^{\e}  dy\right|
\le \frac{h_M}{\e} \|\e^2 g^{\e} \|_{L^2(\Om)} \|Z^{\e}\|_{L^2(\Om)}.
\end{eqnarray*}
By using Poincar\'e's inequality and the boundary conditions
(\ref{eqn:er2.8a})-(\ref{eqn:er2.11b}), we get
\begin{eqnarray*}
\|Z^{\e}\|_{L^2(\Om)} \le \| \frac{\partial Z^{\e}}{\partial y_2} \|_{L^2(\Om)}
\quad \mbox{\rm a.e. in} \ (0,T),
\end{eqnarray*}
and
\begin{eqnarray}\label{es13}
&& \e \left|\int_{\Om}g^{\e}(t)Z^{\e}(t) h^{\e}  dy\right|
 \le  \frac{h_M}{\e} \|\e^2 g^{\e} \|_{L^2(\Om)} \|\frac{\partial
Z^{\e}}{\partial y_2}\|_{L^2(\Om)}
\nonumber \\
&&
 \le \frac{h_M^2}{\e }  \|\e^2 g^{\e} \|_{L^2(\Om)} \|\frac{1}{h^{\e}}
\frac{\partial Z^{\e}}{\partial y_2}\|_{L^2(\Om)}
\nonumber \\
&& \le \frac{h_M^4}{2 \lambda_{14}  \e } \|\e^2 g^{\e} \|_{L^2(\Om)}^2 +
\frac{\lambda_{14}}{2 \e}   \|\frac{1}{h^{\e}} \frac{\partial Z^{\e}}{\partial
y_2}\|_{L^2(\Om)}^2.
\end{eqnarray}
Similarly
\begin{eqnarray} \label{ess14}
  \e \left|\int_{\Om}f_1^{\e}(t)v_1^{\e}(t) h^{\e}  dy\right| \le
\frac{h_M^4}{2 \lambda_{15}  \e}  \|\e^2 f_1^{\e} \|_{L^2(\Om)}^2 +
\frac{\lambda_{15}}{2 \e}   \|\frac{1}{h^{\e}} \frac{\partial v_1^{\e}}{\partial
y_2}\|_{L^2(\Om)}^2,
\end{eqnarray}
and
\begin{eqnarray} \label{ess15}
  \e \left|\int_{\Om}f_2^{\e}(t)v_2^{\e}(t) h^{\e}  dy\right| \le
\frac{h_M^4}{2 \lambda_{16}  \e } \|\e^2 f_2^{\e} \|_{L^2(\Om)}^2 +
\frac{\lambda_{16}}{2 \e}   \|\frac{1}{h^{\e}} \frac{\partial v_2^{\e}}{\partial
y_2}\|_{L^2(\Om)}^2.
\end{eqnarray}
So from (\ref{eq117}) and  (\ref{es1})-(\ref{es13})  we get
\begin{eqnarray}\label{eq3.33}
\frac{\e}{2} {d\over dt}([{\bar{v}}^{\e}(t)]^{2})
+ {c_{1}\over \e}\|(\e\, b_{\e}\cdot \nabla v_{1}^{\e}(t)\|^{2}_{L^{2}(\Om)}
+ {c_{2}\over \e}\|(\e\, b_{\e}\cdot \nabla v_{2}^{\e}(t)\|^{2}_{L^{2}(\Om)}
\nonumber\\
+ {c_{3}\over \e} \|{1\over h^{\e}}{\partial Z^{\e}(t)\over \partial
y_{2}}\|^{2}_{L^{2}(\Om)}
+ {c_{4}\over \e}\|{1\over h^{\e}}{\partial v^{\e}_{1}(t)\over \partial
y_{2}}\|^{2}_{L^{2}(\Om)}
+ {c_{5}\over \e} h_m \|{1\over h^{\e}}{\partial v^{\e}_{2}(t)\over
\partial y_{2}}\|^{2}_{L^{2}(\Om)}
\nonumber\\
+ {c_{6}\over \e}\|(\e\, b_{\e}\cdot \nabla Z^{\e}(t)\|^{2}_{L^{2}(\Om)}
+ \nu_{r}\e c_{7} \|Z^{\e}(t)\|^{2}_{L^{2}(\Om)} \leq
\e c_{8}(t)[{\bar{v}^{\e}}(t)]^{2} +   {c_{9}(t)\over\e}
\end{eqnarray}
where
\begin{eqnarray*}
 &&c_{1}=(\nu+ \nu_{r}) h_m  - {1\over 2\lambda_{9}}, \qquad
 c_{2}= (\nu + \nu_{r}) h_m   - {\nu_{r}\over\lambda_{3}},
\nonumber\\
&&c_{3}=\aal h_m  - {\nu_{r}\over \lambda_{1}}- {\lambda_{6}\aal \over 2}
-  {1\over 2  \lambda_{12}} - \frac{\lambda_{14}}{2}, \quad
c_{4}=(\nu + \nu_{r}) h_m   - {\lambda_{5}(\nu + \nu_{r})^2 \over 2}
 -{\nu_{r}\over \lambda_{4}} -{1 \over 2\lambda_{10}} -
\frac{\lambda_{15}}{2},
\nonumber\\
&& c_{5} = (\nu + \nu_r) h_m - \frac{\lambda_{16}}{2} , \quad
    c_{6}=\aal h_m  -  {1\over 2  \lambda_{11}} -{\nu_{r}\over \lambda_{2}},\quad
 c_{7}=4 h_m   - \lambda_{8} - 2 \lambda_{13}, \nonumber \\
&&
c_{8}(t)=\max\{A(t) , B(t), h_M^2 (1 + \lambda_3 + \lambda_4) \}
\end{eqnarray*}
 with
$$A(t)=h_{M}^{2}
\left(1+\nu_{r}\lambda_{1}+ \nu_{r}\lambda_{7} +
{\lambda_{9} U_{0}^{2}(t) \|{\cal U}\|^{2}_{\infty}
+ \lambda_{11}W^{2}_{0}(t)\|{\cal W}\|^{2}_{\infty}  \over 2}\right)$$
$$B(t)=h_{M}^{2}\left( \frac{1}{2} + \nu_{r}\lambda_{2} +
{\lambda_{10}U^{2}_{0}(t)\|{\cal U}\|^{2}_{\infty}
 \over 2}
+ {\lambda_{12}W^{2}_{0}(t) \|{\cal W}\|^{2}_{\infty}\over 2}  \right)$$
and
\begin{eqnarray*}
{c_{9}(t)\over\e}&=&   {U^{2}_{0}(t) \|{\cal U}'\|^{2}_{L^{2}(0, h_m)}
h_{M} L
 \over 2\e\lambda_{5}  }
 + {\nu_{r}   U^{2}_{0}(t) \|{\cal U}\|^{2}_{L^{2}(0, h_m)} h_{M} L
\over \e\lambda_{8}  }
+{\aal   W^{2}_{0}(t)  \|{\cal W}'\|^{2}_{L^{2}(0, h_m)} h_{M} L    \over 2
\e \lambda_{6}  }
 \nonumber\\
&&
+ {\nu_{r} W^{2}_{0}(t)  \|{\cal W}'\|^{2}_{L^{2}(0, h_m )} L \over \e
\lambda_{7} h_M}
+ {2\nu_{r} \e W_{0}(t)^{2} \|{\cal W}\|^{2}_{L^{2}(0, h_m )}    h_{M} L \over
\lambda_{13}  }
+ \frac{\e L}{2 h_M} |U'_{0}(t)|^{2}\|{\cal U}\|^{2}_{L^{2}(0, h_m)}
 \nonumber\\
&&
+ \frac{\e L}{2 h_M} |W'_{0}(t)|^{2} \|{\cal W}\|^{2}_{L^{2}(0, h_m)}
+ {h_M^4\over 2  \e}\left(\frac{1}{\lambda_{15}} \|\e^2 f_{1}^{\e}
(t)\|^{2}_{L^{2}(\Omega)}
+ \frac{1}{\lambda_{16}} \|\e^2 f_{2}^{\e} (t)\|^{2}_{L^{2}(\Omega)}
+ \frac{1}{\lambda_{14}} \|\e^2 g^{\e} (t)\|^{2}_{L^{2}(\Omega)}\right).
\end{eqnarray*}
Each  $c_{i}$ for $i=1,\cdots, 6$  must be strictly positive, which  is
possible for example with
\begin{eqnarray*}
\lambda_{1}= \lambda_{2}= {4\nu_{r}\over \aal h_m },\qquad \qquad \lambda_{3}=
\lambda_{4}= \frac{1}{h_m} ,\qquad
 \qquad \lambda_{5}={\nu h_m \over 2(\nu + \nu_{r})^2},
\nonumber\\
 \qquad \lambda_{6}= \lambda_{8}=\lambda_{13}={h_m \over 2},\qquad
\lambda_{9}={4\over (\nu + \nu_{r}) h_m },  \qquad \lambda_{10}={2\over \nu h_m
},  \qquad
\lambda_{11}=\lambda_{12}={2\over \aal h_m }, \nonumber \\
\qquad \lambda_{14} = \frac{ \alpha h_m}{8} , \qquad \qquad \lambda_{15} =
\frac{ \nu h_m}{4}, \qquad \qquad \lambda_{16} = \frac{ (\nu + \nu_r) h_m}{2} .
\end{eqnarray*}
Note that $\lambda_{7}$  remains  arbitrary and can be taken as $\lambda_7=1$.
So
from (\ref{eq3.33})  we get
\begin{eqnarray*}
\frac{\e^{2}}{2} {d\over dt} ([\bar{v}^{\e}(t)]^{2})
\leq
\e^{2}  c_{7}(t) [\bar{v}^{\e}(t)]^{2}
+ c_{8}(t).
\end{eqnarray*}
As $U_{0}$ and $W_{0}$ belong to  $H^{1}(0 , T)$, then $c_8$ is bounded in
$L^1(0,T)$ independently of $\e$ and  by Gr\"onwall's
lemma we deduce that
 there exists a constant $C$ independent of $\e$ such that
\begin{eqnarray}\label{eq3.34}
\e^2 [\bar{v}^{\e}(t)]^{2} \leq C
\quad \forall t\in [0 , T].
\end{eqnarray}
Now we integrate the inequality (\ref{eq3.33}) over the time interval  $(0 , s)$
for $0< s\leq T$,
we deduce
\begin{eqnarray}\label{eq3.35}
\frac{\e^{2}}{2} [{\bar{v}}^{\e}(s)]^{2}
+
C_{1}\int_{0}^{s}\|(\e\, b_{\e}\cdot \nabla v_{1}^{\e}(t)\|^{2} _{L^{2}(\Om)}
 +\|(\e\, b_{\e}\cdot \nabla v_{2}^{\e}(t)\|^{2}_{L^{2}(\Om)}+
+ \|(\e\, b_{\e}\cdot \nabla Z^{\e}(t)\|^{2}_{L^{2}(\Om)} dt
\nonumber\\
+ C_{2}\int_{0}^{s}
\|{1\over h^{\e}}{\partial Z^{\e}(t)\over \partial y_{2}}\|^{2}_{L^{2}(\Om)}
+
\|{1\over h^{\e}}{\partial v^{\e}_{1}(t)\over \partial y_{2}}\|^{2}_{L^{2}(\Om)}
+ \|{1\over h^{\e}}{\partial v^{\e}_{2}(t)\over \partial
y_{2}}\|^{2}_{L^{2}(\Om)} dt
\nonumber\\
+ \e^{2} \nu_r c_6 \int_{0}^{s} \|(Z^{\e}(t)\|^{2}_{L^{2}(\Om)} dt
\leq
\e^{2}  \int_{0}^{s} c_{7}(t)[{\bar{v}}^{\e}(t)]^{2}dt +   \int_{0}^{s}
c_{8}(t) dt
+ \frac{\e^{2}}{2} [{\bar{v}}^{\e}(0)]^{2},
\end{eqnarray}

where $C_{1}= \min\{c_{1}, c_{2},  c_{6}\}$,
$C_{2}= \min\{c_{3},c_{4}, c_{5} \}$, are two constants independent of $\e$.
Observing that $c_7 \in L^{\infty}(0,T)$ and
\begin{eqnarray*}
 \int_{0}^{T}\int_{\Om} \left({1\over h^{\e}}{\partial v^{\e}_{i}(t)\over
\partial y_{2}}\right)^{2}  \, dy dt
\geq {1\over h_{M}^{2}}\|{\partial v^{\e}_{i}\over \partial
y_{2}}\|_{L^{2}((0,T) \times\Om)}^{2}
\end{eqnarray*}
 we deduce (\ref{E3.13}) and (\ref{E3.14}) from (\ref{eq3.34}). Moreover, from
(\ref{not})
\begin{eqnarray}\label{not2}
b_{\e}\cdot\nabla= {\partial \over \partial y_{1}} - {y_{2}\over
h^{\e}}{\partial h^{\e}\over \partial y_{1}}  {\partial \over \partial
y_{2}}
\quad \mbox{ with }\quad
|{\partial h^{\e}\over \partial y_{1}}| =|{\partial h\over \partial
y_{1}} + {1\over \e} {\partial h\over \partial \eta_{1}}|\leq {C\over \e}.
\end{eqnarray}
Thus we have
\begin{eqnarray*}
 \int_{0}^{T}\int_{\Om} \left(\e{\partial v^{\e}_{i}(t)\over \partial
y_{1}}\right)^{2}  \, dy dt
=\int_{0}^{T}\int_{\Om}  \left((\e b_{\e}\cdot\nabla v^{\e}_{i}(t))
 + {y_{2}\e \over h^{\e}}{\partial h^{\e}\over \partial y_{1}}  {\partial
v^{\e}_{i}(t) \over \partial y_{2}}     \right)^{2}  \,  dy dt
\nonumber\\
\leq 2 \|(\e b_{\e}\cdot\nabla v^{\e}_{i})\|^{2}_{L^{2}((0 ,
T)\times\Om)}
 + 2 {C\over h_{m}}\| {\partial v^{\e}_{i} \over \partial
y_{2}}\|^{2}_{L^{2}((0
, T)\times\Om)}
\end{eqnarray*}
and a similar estimate holds for $Z^{\e}$.
Finally  with (\ref{E3.13}) and   (\ref{E3.14}) we deduce  (\ref{E3.15}).
Next, using again the boundary conditions (\ref{eqn:er2.8a})-(\ref{eqn:er2.11b})
and    Poincar\'e's inequality, we get
\begin{eqnarray*}
 \|v^{\e}_{i}\|^{2}_{L^{2}((0 , T)\times\Om)}\leq  \int_{0}^{T}\int_{\Om}
\| {\partial v^{\e}_{i} \over \partial y_{2}}\|^{2} dy dt=
\| {\partial v^{\e}_{i} \over \partial y_{2}}\|^{2}_{L^{2}((0 , T)\times\Om)}
\end{eqnarray*}
and  we deduce (\ref{E3.16}) from  (\ref{E3.14}).
\end{proof}


\begin{proposition}\label{prop2}
Assume that the  proposition {\rm\ref{pro1}} holds. Then
there exists a constant $C>0$ which does not depends on  $\e$,
such that  we have
\begin{eqnarray}\label{E3.18}
 \e^{2} \|p^{\e}\|_{H^{-1}(0 , T; L^{2}(\Om))} \leq C.
\end{eqnarray}
  \end{proposition}
\begin{proof}
Let  $\varphi\in {\mathcal D}(0,T) \times {\mathcal D}(\Om)$,  then choose
$\Theta= ( 0, \varphi(t), 0)$ as a test-function
 in (\ref{eqvar}) and  multiply the two sides by $\e$:  we obtain
 \begin{eqnarray}\label{eq3.43}
\e \int_{0}^{T}\int_{\Om} p^{\e}
{\partial \varphi\over \partial  y_{2}} dy dt
=
-\e^{2}\int_{0}^{T} \int_{\Om} v^{\e}_{2} {\partial\varphi\over \partial t}
h^{\e}
dy dt
\nonumber\\
+ (\nu+\nu_{r})\sum_{i=1}^{2}
\int_{0}^{T} \int_{\Om}\left(
h^{\e}(\e b_{\e}\cdot  \nabla v^{\e}_{2}) (\e b_{\e}\cdot  \nabla\varphi)
+{1\over h^{\e}}{\partial v^{\e}_{2}\over \partial y_{2}}
           {\partial\varphi\over \partial y_{2}}\right) dy dt
\nonumber\\
+\int_{0}^{T}\int_{\Om} \left(\sum_{i=1}^{2}
\e v^{\e}_{1}(\e b_{\e}\cdot \nabla v^{\e}_{2})\varphi h^{\e}
 + \e v^{\e}_{2}{\partial v^{\e}_{2}\over \partial
y_{2}}\varphi\right)  dy
dt
 + 2\nu_{r}
\int_{0}^{T} \int_{\Om}
(\e b_{\e}\cdot  \nabla Z^{\e})\varphi \e h^{\e} dy dt
\nonumber\\
+ \int_{0}^{T}\int_{\Om}
U_{\e}(\e b_{\e}\cdot  \nabla v^{\e}_{2})\varphi\e h^{\e} dy dt
- \int_{0}^{T} \int_{\Om}
\e^{2}  f^{\e}_{2}\varphi h^{\e} dy dt,
\end{eqnarray}
with $U_{\e} (t, y) = U_0(t) {\mathcal U}(y_2 h^{\e} (y_1))$ for all
$(t, y_1, y_2) \in [0,T] \times \Om$. Using (\ref{E3.13})-(\ref{E3.16}),  we  get
 \begin{eqnarray}\label{equPy1}
 |\int_{0}^{T} \int_{\Om}
p^{\e} {\partial \varphi\over \partial  y_{2}} dy dt
|\leq  {C\over \e} \|\varphi\|_{H^{1}(0 , T ,
H^{1}_{0}(\Om))} \quad \forall \varphi\in {\mathcal D}(0,T) \times {\mathcal
D}(\Om) .
 \end{eqnarray}
Now let $\phi \in {\mathcal D}(0,T) \times {\mathcal D}(\Om)$ and choose $\Theta
= ( \frac{\phi}{h^{\e}}, 0,0)$ as a test-function in (\ref{eqvar}), then
multiply the two sides by $\e$: we obtain
 \begin{eqnarray}\label{eq3.45}  
\e^{2} \int_{0}^{T}\int_{\Om} p^{\e}
\left( {\partial\phi \over \partial y_{1}}
- \frac{\partial }{\partial y_2} \left( y_2 \frac{1}{h^{\e}} \frac{\partial
h^{\e}}{\partial y_1} \phi \right) \right)
 dy dt
=
-\e^{2}\int_{0}^{T} \int_{\Om} v^{\e}_{1} {\partial\phi\over \partial t}
dy dt
\nonumber\\
+ (\nu+\nu_{r})
\int_{0}^{T} \int_{\Om} \e
(\e b_{\e}\cdot  \nabla v^{\e}_{1})
 \left(
\frac{\partial \phi}{\partial y_1} - \frac{\partial h^{\e}}{\partial y_1}
\frac{1}{h^{\e}} \phi - y_2 \frac{1}{h^{\e}} \frac{\partial h^{\e}}{\partial
y_1} \frac{\partial \phi}{\partial y_2}
\right) dy dt \nonumber\\
 +
(\nu+\nu_{r})
\int_{0}^{T} \int_{\Om}
{1\over ( h^{\e} )^2}{\partial v^{\e}_{1}\over \partial y_{2}}
           {\partial\phi\over \partial y_{2}}  dy dt
+\int_{0}^{T}\int_{\Om} \left(
\e v^{\e}_{1}(\e b_{\e}\cdot  \nabla v^{\e}_{1})\phi
 + \frac{\e}{h^{\e}} v^{\e}_{2}{\partial v^{\e}_{1}\over \partial
y_{2}}\phi\right)  dy
dt
\nonumber\\
+ (\nu+\nu_{r})\int_{0}^{T} \int_{\Om} \frac{1}{(h^{\e})^2}
{\partial U_{\e}\over \partial y_{2}}{\partial\phi\over \partial y_{2}}
dy
dt
-2\nu_{r}
\int_{0}^{T} \int_{\Om} \frac{\e}{h^{\e}}
{\partial Z^{\e}\over \partial y_{2}}\phi  dy dt
+ \int_{0}^{T} \int_{\Om} \e U_{\e}
(\e b_{\e}\cdot  \nabla v^{\e}_{1})\phi dy dt
\nonumber\\
- \int_{0}^{T} \int_{\Om}
\e^2  v^{\e}_{1} \left( \frac{\partial \phi}{\partial y_1} - \frac{\partial
h^{\e}}{\partial y_1} \frac{1}{h^{\e}} \phi - y_2 \frac{\partial
h^{\e}}{\partial y_1}\frac{1}{h^{\e}} \frac{\partial \phi}{\partial y_2} \right)
U_{\e}  dy dt
-  \int_{0}^{T} \int_{\Om}
\frac{\e}{h^{\e}} v^{\e}_{2}{\partial \phi \over \partial y_{2}}U_{\e} dy dt
\nonumber\\
-2\e \nu_{r}\int_{0}^{T} \int_{\Om}
 \frac{1}{h^{\e}}    {\partial W_{\e}\over \partial y_{2}}\phi dy dt
- \int_{0}^{T} \int_{\Om}
\left( f^{\e}_{1}\phi -{\partial U^{\e} \over \partial
t}\phi\right) \e^{2}  dy dt,
\end{eqnarray}
where $W_{\e} (t, y) = W_0(t) {\mathcal W}(y_2 h^{\e} (y_1))$ for all $(t, y_1,
y_2) \in [0,T] \times \Om$.

Using the estimates (\ref{E3.13})-(\ref{E3.16}) and (\ref{not2}), we infer that
\begin{eqnarray*}
|\int_{0}^{T}\int_{\Om} p^{\e}
\left( {\partial\phi \over \partial y_{1}}
- \frac{\partial }{\partial y_2} \left( y_2 \frac{1}{h^{\e}} \frac{\partial
h^{\e}}{\partial y_1} \phi \right) \right)
 dy dt | \leq {C \over \e^{2}}\|\phi\|_{H^{1}(0 , T ;  H^{1} (\Om))}.
 \end{eqnarray*}
By choosing now
$\varphi= y_{2} \frac{1}{h^{\e}} {\partial h^{\e}\over\partial y_{1}}\phi$  in
(\ref{eq3.43}),  we get
\begin{eqnarray*}
\frac{\partial \varphi}{\partial t} =  y_{2}  \frac{1}{h^{\e}} {\partial
h^{\e}\over\partial y_{1}} \frac{\partial \phi}{\partial t} , \quad
\frac{\partial \varphi}{\partial y_2} =   \frac{1}{h^{\e}} {\partial
h^{\e}\over\partial y_{1}} \left( \phi + y_2  \frac{\partial \phi}{\partial y_2}
\right)
\end{eqnarray*}
and
\begin{eqnarray*}
b_{\e} \cdot \nabla \varphi =
 - y_2 \frac{1}{(h^{\e})^2}
 \left( \frac{\partial h^{\e}}{\partial y_1} \right)^2 \left( 2 \phi + y_2
\frac{\partial \phi}{\partial y_2} \right)
 + y_2 \frac{1}{h^{\e}} \left( \frac{\partial^2 h^{\e}}{\partial y_1^2} \phi +
\frac{\partial h^{\e}}{\partial y_1} \frac{\partial \phi}{\partial y_1} \right)
.
\end{eqnarray*}
Hence
\begin{eqnarray*}
&& \| \varphi \|_{L^{\infty}(0,T; L^4(\Om))} \le \frac{C}{\e} \| \phi\|_{H^1
(0,T; H^1(\Om))}, \quad
\| \frac{\partial \varphi}{\partial t} \|_{L^2 ((0,T) \times \Om)} \le
\frac{C}{\e} \| \phi\|_{H^1 (0,T; H^1(\Om))}, \\
&&  \| \frac{\partial \varphi}{\partial y_2} \|_{L^2 ((0,T) \times \Om)} \le
\frac{C}{\e} \| \phi\|_{H^1 (0,T; H^1(\Om))}, \quad
\| \e b_{\e} \cdot \nabla \varphi  \|_{L^2 ((0,T) \times \Om)} \le \frac{C}{\e}
\| \phi\|_{H^1 (0,T; H^1(\Om))}
\end{eqnarray*}
and with (\ref{eq3.43})
\begin{eqnarray*}
|\int_{0}^{T}\int_{\Om} p^{\e}
 \frac{\partial }{\partial y_2} \left( y_2 \frac{1}{h^{\e}} \frac{\partial
h^{\e}}{\partial y_1} \phi \right)
 dy dt | \leq {C \over \e^{2}}\|\phi\|_{H^{1}(0 , T ;  H^{1}(\Om))}.
 \end{eqnarray*}
It follows that
\begin{eqnarray} \label{equPy2}
|\int_{0}^{T}\int_{\Om} p^{\e}
 {\partial\phi \over \partial y_{1}}
 dy dt | \leq {C \over \e^{2}}\|\phi\|_{H^{1}(0 , T ;  H^{1} (\Om))} \quad
\forall \phi \in {\mathcal D}(0,T) \times {\mathcal D}(\Om).
 \end{eqnarray}
By density of ${\mathcal D}(0,T) \times {\mathcal D}(\Om)$ into $H^1_0(0,T;
H^1_0(\Om))$ we get from (\ref{equPy1})-(\ref{equPy2})
\begin{eqnarray} \label{equPy3}
\| \frac{\partial p^{\e}}{\partial y_2} \|_{H^{-1} (0,T; H^{-1} (\Om))} \le
\frac{C}{\e}, \quad \| \frac{\partial p^{\e}}{\partial y_1} \|_{H^{-1} (0,T;
H^{-1} (\Om))} \le \frac{C}{\e^2}.
\end{eqnarray}
Finally  we can deduce \cite{Temam1979}
that
$\e^{2}p^{\e}$ remains in a bounded subset of $H^{-1}(0 , T ; L^{2}(\Om))$.
 \end{proof}

\renewcommand{\theequation}{4.\arabic{equation}}
\setcounter{equation}{0}
\section{Two-scale convergence properties}\label{twoscaleconv}

Since our unknown functions depend on the time variable,
  we are not in the
classical framework
of two-scale convergence as it has been introduced by G. Allaire in
\cite{allaire} or G. Nguetseng in \cite{nguetseng}. Nevertheless this technique can be easily adpated to a time-dependent framework (see for instance \cite{miller, holm97,  gilbert-Mik2000, wright00}). For the convenience of the reader we will provide a complete proof a the generalization of \cite{allaire} that will be used later for the study of the sequences $(v^{\e})_{\e >0}$, $(Z^{\e})_{\e >0}$ and $(p^{\e})_{\e >0}$.

Let us recall the following usual notations: $Y=[0,1]^2$,
$ {\mathcal C}^{\infty}_{\sharp} (Y) $ is the space of infinitely differentiable
functions in
$\br^2$ that are $Y$-periodic
and
\begin{eqnarray*}
L^2_{\sharp} (Y) = \overline{{\mathcal C}^{\infty}_{\sharp} (Y)}^{L^2(Y)}, \quad
H^1_{\sharp} (Y) = \overline{{\mathcal C}^{\infty}_{\sharp} (Y)}^{H^1(Y)}.
\end{eqnarray*}

\begin{remark} The space $L^2_{\sharp} (Y)$ coincides with the space of
functions of $L^2(Y) $ extended by $Y$-periodicity to $\br^2$.
\end{remark}

We extend the definition of the two-scale convergence as follows

\begin{definition} A sequence $(w^{\e})_{\e >0}$ of $L^2 \bigl( (0,T) \times
\Omega \bigr)$ (resp. in $H^{-1}  \bigl( 0,T; L^2( \Omega) \bigr)$) two-scale
converges to $w^0 \in L^2 \bigl( 0,T; L^2 (\Omega \times Y) \bigr)$ (resp. $w^0
\in H^{-1} \bigl( 0,T; L^2 (\Omega \times Y) \bigr)$) if and only if
\begin{eqnarray*}
\lim_{\e \to 0} \int_0^T \int_{\Omega} w^{\e} (t,y) \varphi \left( y,
\frac{y}{\e} \right) \theta (t) \, dy dt = \int_0^T \int_{\Omega \times Y} w^0 (
t, y, \eta) \varphi (y, \eta) \theta (t) \, d\eta dy dt
\end{eqnarray*}
for all $\theta \in {\mathcal D}(0,T)$, for all $\varphi \in {\mathcal D}
\bigl(\Omega; {\mathcal C}^{\infty}_{\sharp} (Y) \bigr)$.
In such a case we will denote $w^{\e} \toH w^0$.
\end{definition}

Then we obtain

\begin{theorem}
Let $(w^{\e})_{\e >0}$ be a bounded sequence of $L^2 \bigl( (0,T) \times \Omega
\bigr)$ (resp. in $H^{-1} \bigl( 0,T ;  L^2( \Omega ) \bigr)$). There exists
$w^0 \in  L^2 \bigl( 0,T; L^2 (\Omega \times Y) \bigr)$ (resp.  $w^0 \in H^{-1}
\bigl( 0,T; L^2 (\Omega \times Y) \bigr)$) such that, possibly extracting a
subsequence still denoted $(w^{\e})_{\e >0}$, we have
\begin{eqnarray*}
w^{\e} \toH w^0.
\end{eqnarray*}
\end{theorem}

\begin{proof}
The proof is similar to the proof of Theorem 1.2 in \cite{allaire}.  Let us
assume first that  $(w^{\e})_{\e >0}$ is a bounded sequence of
$L^2 \bigl( (0,T)\times \Omega \bigr)$. In our time-dependent framework we
consider test-functions $\psi \in {\mathcal C} \bigl( [0,T]; {\mathcal C}
(\overline{\Omega}; {\mathcal C}_{\sharp} (Y) ) \bigr)$.
Furthermore,
For any $\psi \in {\mathcal C} \bigl( [0,T]; {\mathcal C} (\overline{\Omega};
{\mathcal C}_{\sharp} (Y) ) \bigr)$ and for any fixed $\e >0$, the mapping $
\displaystyle (t,y) \mapsto \psi^{\e} (t,y)= \psi \left( t, y, \frac{y}{\e}
\right)$ is mesurable on $(0,T) \times \Omega$ and satisfies
\begin{eqnarray*}
\| \psi^{\e} \|_{L^2 ( (0,T) \times \Omega)}  =  \left( \int_0^T \int_{\Omega}
\left(  \psi \left( t, y, \frac{y}{\e} \right) \right)^2 \, dy dt \right)^{1/2}
\le \sqrt{ T | \Omega|} \| \psi\|_{{\mathcal C} ( [0,T]; {\mathcal C}
(\overline{\Omega}; {\mathcal C}_{\sharp} (Y) ) )}.
\end{eqnarray*}
Hence we can define $\Lambda_{\e} \in \Bigl(  {\mathcal C} \bigl( [0,T];
{\mathcal C} (\overline{\Omega}; {\mathcal C}_{\sharp} (Y) ) \bigr) \Bigr)'$ by
\begin{eqnarray*}
\Lambda_{\e} (\psi) = \int_0^T \int_{\Omega} w^{\e} (t,y) \psi \left( t, y,
\frac{y}{\e} \right) \, dy dt
   \quad \forall \psi \in  {\mathcal C} \bigl( [0,T]; {\mathcal C}
(\overline{\Omega}; {\mathcal C}_{\sharp} (Y) ) \bigr).
\end{eqnarray*}
Since $(w^{\e})_{\e >0}$ is a bounded sequence of $L^2 \bigl( (0,T) \times
\Omega \bigr)$, we infer with Cauchy-Schwarz's inequality that there exists a
real number $C>0$, independent of $\e$, such that
\begin{eqnarray} \label{sec4.1}
\bigl| \Lambda_{\e} (\psi) \bigr| &\le& \|w^{\e} \|_{L^2 ( (0,T) \times \Omega)}
\| \psi^{\e} \|_{L^2 ( (0,T) \times \Omega)}  \le C  \| \psi^{\e} \|_{L^2 (
(0,T) \times \Omega)}   \nonumber\\
&\le& C  \sqrt{ T | \Omega|} \| \psi\|_{{\mathcal C} (
[0,T]; {\mathcal C} (\overline{\Omega}; {\mathcal C}_{\sharp} (Y) ) )}
\end{eqnarray}
for all $ \psi \in  {\mathcal C} \bigl( [0,T]; {\mathcal C} (\overline{\Omega};
{\mathcal C}_{\sharp} (Y) ) \bigr)$ and  the sequence $(\Lambda_{\e})_{{\e} >0}$
is bounded in $\Bigl(  {\mathcal C} \bigl( [0,T]; {\mathcal C}
(\overline{\Omega}; {\mathcal C}_{\sharp} (Y) ) \bigr) \Bigr)'$. Reminding that
$ {\mathcal C} \bigl( [0,T]; {\mathcal C} (\overline{\Omega}; {\mathcal
C}_{\sharp} (Y) ) \bigr) $ is a separable Banach space, we infer that there
exists $\Lambda_0 \in \Bigl(  {\mathcal C} \bigl( [0,T]; {\mathcal C}
(\overline{\Omega}; {\mathcal C}_{\sharp} (Y) ) \bigr) \Bigr)'$ such that,
possibly extracting a subsequence still denoted $(\Lambda_{\e})_{{\e} >0}$,
\begin{eqnarray*}
(\Lambda_{\e}) \rightharpoonup \Lambda_{0} \quad \hbox{\rm weak * in }
\Bigl( {\mathcal C} \bigl( [0,T]; {\mathcal C} (\overline{\Omega}; {\mathcal
C}_{\sharp} (Y) ) \bigr) \Bigr)'
\end{eqnarray*}
i.e.
\begin{eqnarray*}
\lim_{{\e} \to 0} \int_0^T \int_{\Omega} w^{\e} (t,y) \psi \left( t, y,
\frac{y}{\e} \right) \, dy dt = \Lambda_0 (\psi)
   \quad \forall \psi \in  {\mathcal C} \bigl( [0,T]; {\mathcal C}
(\overline{\Omega}; {\mathcal C}_{\sharp} (Y) ) \bigr).
\end{eqnarray*}
Observing that, for all $t \in [0,T]$, $\psi^2 ( t, \cdot, \cdot) \in L^1
\bigl(\Omega; {\mathcal C}_{\sharp} (Y) ) \bigr)$, we can also use Lemma 5.2 of
\cite{allaire}, which yields
\begin{eqnarray*}
\lim_{{\e} \to 0} \int_{\Omega} \left(  \psi \left( t, y, \frac{y}{\e} \right)
\right)^2 \, dy = \int_{\Omega \times Y} \bigl( \psi( t, y, \eta) \bigr)^2 \, d
\eta  dy  \quad \forall t \in [0,T].
\end{eqnarray*}
Then, using Lebesgue's convergence theorem, we obtain
\begin{eqnarray} \label{sec4.2}
\lim_{{\e} \to 0} \int_0^T \int_{\Omega} \left(  \psi \left( t, y, \frac{y}{\e}
\right) \right)^2 \, dy dt = \int_0^T \int_{\Omega \times Y} \bigl( \psi( t, y,
\eta) \bigr)^2 \, d \eta  dy dt .
\end{eqnarray}
With (\ref{sec4.1}) and (\ref{sec4.2}) we get
\begin{eqnarray*}
\bigl| \Lambda_0 ( \psi) \bigr| \le C \| \psi\|_{L^2 (0,T; L^2 (\Omega \times
Y))} \quad \forall \psi \in  {\mathcal C} \bigl( [0,T]; {\mathcal C}
(\overline{\Omega}; {\mathcal C}_{\sharp} (Y) ) \bigr).
\end{eqnarray*}
It follows that $\Lambda_0$ can be extended to $ \Bigl( L^2 \bigl(0,T; L^2
(\Omega \times Y) \bigr) \Bigr)'$ and with Riesz's representation theorem we
infer that there exists $w^0 \in L^2 \bigl(0,T; L^2 (\Omega \times Y) \bigr)$
such that
\begin{eqnarray*}
\Lambda_0 ( \psi) = \int_0^T \int_{\Omega \times Y} w^0 (t, y, \eta) \psi (t, y,
\eta) \, d\eta dy dt \quad \forall \psi \in L^2 \bigl(0,T; L^2 (\Omega \times Y)
\bigr)
\end{eqnarray*}
which allows us to conclude for the first part of the theorem.

Let us assume now that $(w^{\e})_{\e >0}$ is a bounded sequence of $H^{-1}
\bigl( 0,T; L^2( \Omega) \bigr)$ and let
\begin{eqnarray*}
{\mathcal C}^1_0 \bigl( [0,T]; {\mathcal C} (\overline{\Omega}; {\mathcal
C}_{\sharp} (Y) ) \bigr) = \bigl\{ \psi \in {\mathcal C}^1 \bigl( [0,T];
{\mathcal C} (\overline{\Omega}; {\mathcal C}_{\sharp} (Y) ) \bigr); \psi(0, y,
\eta) = \psi(T, y, \eta) =0 \ \forall (y, \eta) \in \overline{\Omega} \times Y
\bigr\}.
\end{eqnarray*}
It is a closed subspace of ${\mathcal C}^1 \bigl( [0,T]; {\mathcal C}
(\overline{\Omega}; {\mathcal C}_{\sharp} (Y) ) \bigr) $ for the usual norm of
${\mathcal C}^1 \bigl( [0,T]; {\mathcal C} (\overline{\Omega}; {\mathcal
C}_{\sharp} (Y) ) \bigr) $
 and for any  $\psi \in {\mathcal C}^1_0 \bigl( [0,T]; {\mathcal C}
(\overline{\Omega}; {\mathcal C}_{\sharp} (Y) ) \bigr)$, we have
\begin{eqnarray*}
&& \displaystyle \| \psi^{\e} \|_{H^1 ( 0,T ; L^2( \Omega) )}  =  \left(
\int_0^T \int_{\Omega} \left(  \psi \left( t, y, \frac{y}{\e} \right) \right)^2
\, dy dt
 + \int_0^T \int_{\Omega} \left(  \frac{\partial \psi}{\partial t}  \left( t, y,
\frac{y}{\e} \right) \right)^2 \, dy dt  \right)^{1/2} \\
&& \displaystyle \le \sqrt{ T | \Omega|} \| \psi\|_{{\mathcal C}^1 ( [0,T];
{\mathcal C} (\overline{\Omega}; {\mathcal C}_{\sharp} (Y) ) )}.
\end{eqnarray*}
Furthermore, we may now define $\Lambda_{\e} \in \Bigl(  {\mathcal C}^1_0 \bigl(
[0,T]; {\mathcal C} (\overline{\Omega}; {\mathcal C}_{\sharp} (Y) ) \bigr)
\Bigr)'$ by
\begin{eqnarray*}
\Lambda_{\e} (\psi) = \int_0^T \int_{\Omega} w^{\e} (t,y) \psi \left( t, y,
\frac{y}{\e} \right) \, dy dt
   \quad \forall \psi \in  {\mathcal C}^1_0 \bigl( [0,T]; {\mathcal C}
(\overline{\Omega}; {\mathcal C}_{\sharp} (Y) ) \bigr).
\end{eqnarray*}
Since $(w^{\e})_{\e >0}$ is a bounded sequence of $H^{-1} \bigl( 0,T;
L^2(\Omega) \bigr)$, we infer  that there exists a real number $C'>0$,
independent of $\e$, such that
\begin{eqnarray*}
\bigl| \Lambda_{\e} (\psi) \bigr| \le \|w^{\e} \|_{H^{-1} ( 0,T ; L^2( \Omega))}
 \| \psi^{\e} \|_{H^1 ( 0,T; L^2( \Omega))}  \le C ' \| \psi^{\e} \|_{H^1 ( 0,T;
L^2( \Omega))}   \le C ' \sqrt{ T | \Omega|} \| \psi\|_{{\mathcal C}^1 ( [0,T];
{\mathcal C} (\overline{\Omega}; {\mathcal C}_{\sharp} (Y) ) )}
\end{eqnarray*}
for all $ \psi \in  {\mathcal C}^1_0 \bigl( [0,T]; {\mathcal C}
(\overline{\Omega}; {\mathcal C}_{\sharp} (Y) ) \bigr)$ and  the sequence
$(\Lambda_{\e})_{{\e} >0}$ is bounded in $\Bigl(  {\mathcal C}^1_0 \bigl( [0,T];
{\mathcal C} (\overline{\Omega}; {\mathcal C}_{\sharp} (Y) ) \bigr) \Bigr)'$.
Since   $ {\mathcal C}^1 \bigl( [0,T]; {\mathcal C} (\overline{\Omega};
{\mathcal C}_{\sharp} (Y) ) \bigr) $ is a separable Banach space, we can
conclude in the same way as previously.
\end{proof}

\begin{remark} We may observe that this proof shows that we can choose
test-functions in
 $ {\mathcal C} \bigl( [0,T]; {\mathcal C} (\overline{\Omega}; {\mathcal
C}_{\sharp} (Y) ) \bigr)$ (resp. in  $ {\mathcal C}^1_0 \bigl( [0,T]; {\mathcal
C} (\overline{\Omega}; {\mathcal C}_{\sharp} (Y) ) \bigr)$) instead of
${\mathcal D} (0,T) \times {\mathcal D} \bigl( \Omega ; {\mathcal
C}^{\infty}_{\sharp} (Y) \bigr)$.
\end{remark}

Then the convergence results for the velocity, the micro-rotation and the
pressure are given in the following three propositions.
\begin{proposition} \label{prop4.1}
{\bf (Two-scale limit of the velocity)}
Under the assumptions of Proposition \ref{pro1}, there exist $v^0 \in \Bigl( L^2
\bigl( 0,T; L^2 (\Omega; H^1_{\sharp}(Y) ) \bigr) \Bigr)^2$ such that
$\displaystyle \frac{\partial v^0}{\partial y_2} \in  \Bigl( L^2 \bigl(0,T; L^2
(\Omega \times Y) \bigr) \Bigr)^2$ and $v^1 \in \Bigl( L^2 \bigl( 0,T; L^2
\bigl(\Omega \times (0,1) ; H^1_{\sharp}(0,1)_{/ \br} ) \bigr) \Bigr)^2$ such
that, possibly extracting a subsequence still denoted $(v^{\e})_{\e>0}$, we have
for $i=1,2$:
\begin{eqnarray}\label{eq:sec4-1}
v_i^{\varepsilon} \toH v^0_i, \quad \frac{\partial v_i^{\varepsilon}}{\partial
y_2} \toH \frac{\partial v^0_i}{\partial y_2} + \frac{\partial v^1_i}{\partial
\eta_2},
\end{eqnarray}
and
\begin{eqnarray} \label{eq:sec4-2}
\varepsilon \frac{\partial v^{\e}_i}{\partial y_1} \toH  \frac{\partial
v^0_i}{\partial \eta_1}.
\end{eqnarray}
Furthermore $v^0$ does not depend on $\eta_2$, $v^0$ is divergence free in the
following sense
\begin{eqnarray} \label{eq:sec4-3}
h(y_1, \eta_1) \frac{\partial v^0_1}{\partial \eta_1} - y_2 \frac{\partial
h}{\partial \eta_1} (y_1, \eta_1) \frac{\partial v^0_1}{\partial y_2} +
\frac{\partial v^0_2}{\partial y_2} = 0 \quad \hbox{\rm in } (0,T) \times \Omega
\times (0,1),
\end{eqnarray}
and
\begin{eqnarray} \label{eq:sec4-4}
&& \displaystyle v^0 = 0 \quad \hbox{\rm on } (0,T) \times \Gamma_0 \times (0,1)
, \Gamma_0 = (0,L) \times \{ 0 \},  \\
&& \displaystyle - v_1^0  \frac{\partial h}{\partial \eta_1} ( y_1, \eta_1) +
v_2^0  =0 \quad \hbox{\rm on } (0,T) \times \Gamma_1 \times (0,1), \Gamma_1 =
(0,L) \times \{1\}.
\end{eqnarray}
\end{proposition}

\begin{proof}
The first part of the result is a direct consequence of the previous theorem and
is obtained by using the same techniques as in Proposition 1.14 in
\cite{allaire}.

Indeed, from Proposition \ref{pro1} we know that $(v^{\e}_i)_{{\e} >0}$, $
\displaystyle \left( \frac{\partial v^{\e}_i}{\partial y_2} \right)_{{\e} >0}$
and $ \displaystyle \left( \e \frac{\partial v^{\e}_i}{\partial y_1}
\right)_{{\e} >0}$ are bounded in $L^2 \bigl( (0,T) \times \Omega \bigr)$. It
follows that, possibly extracting a subsequence,   they two-scale converge to
$v^0_i$, $\xi^0_i$ and $\xi^1_i$ respectively, i.e.
\begin{eqnarray} \label{sec4.3}
\lim_{\e \to 0} \int_0^T \int_{\Omega} v^{\e}_i (t,y) \varphi \left( y,
\frac{y}{\e} \right) \theta (t) \, dy dt = \int_0^T \int_{\Omega \times Y} v^0_i
( t, y, \eta) \varphi (y, \eta) \theta (t) \, d\eta dy dt
\end{eqnarray}
\begin{eqnarray}  \label{sec4.4}
\begin{array}{l}
 \displaystyle \lim_{\e \to 0} \int_0^T \int_{\Omega} \frac{\partial
v^{\e}_i}{\partial y_2} (t,y) \varphi \left( y, \frac{y}{\e} \right) \theta (t)
\, dy dt \\
\displaystyle =  - \lim_{\e \to 0} \int_0^T \int_{\Omega} v^{\e}_i (t,y) \left(
\frac{\partial \varphi}{\partial y_2}  \left( y, \frac{y}{\e} \right)  +
\frac{1}{\e}  \frac{\partial \varphi}{\partial \eta_2}  \left( y, \frac{y}{\e}
\right)  \right) \theta (t) \, dy dt  \\
 \displaystyle = \int_0^T \int_{\Omega \times Y} \xi^0_i ( t, y, \eta) \varphi
(y, \eta) \theta (t) \, d\eta dy dt
\end{array}
\end{eqnarray}
and
\begin{eqnarray}  \label{sec4.5}
\begin{array}{l}
 \displaystyle \lim_{\e \to 0} \int_0^T \int_{\Omega} \e \frac{\partial
v^{\e}_i}{\partial y_1}  (t,y) \varphi \left( y, \frac{y}{\e} \right) \theta (t)
\, dy dt \\
\displaystyle =  - \lim_{\e \to 0} \int_0^T \int_{\Omega} v^{\e}_i (t,y) \left(
\e  \frac{\partial \varphi}{\partial y_1}  \left( y, \frac{y}{\e} \right)  +
\frac{\partial \varphi}{\partial \eta_1}  \left( y, \frac{y}{\e} \right)
\right) \theta (t) \, dy dt  \\
 \displaystyle = \int_0^T \int_{\Omega \times Y} \xi^1_i ( t, y, \eta) \varphi
(y, \eta) \theta (t) \, d\eta dy dt
 \end{array}
\end{eqnarray}
for all $\theta \in {\mathcal D}(0,T)$, $\varphi  \in {\mathcal D} \bigl(\Omega;
{\mathcal C}^{\infty}_{\sharp} (Y) \bigr)$. From (\ref{sec4.5}) and
(\ref{sec4.3}) we obtain
\begin{eqnarray} \label{sec4.5bis}
\begin{array}{l}
 \displaystyle - \lim_{\e \to 0} \int_0^T \int_{\Omega} v^{\e}_i (t,y) \left( \e
 \frac{\partial \varphi}{\partial y_1}  \left( y, \frac{y}{\e} \right)  +
\frac{\partial \varphi}{\partial \eta_1}  \left( y, \frac{y}{\e} \right)
\right) \theta (t) \, dy dt  \\
\displaystyle = - \int_0^T \int_{\Omega \times Y} v^0_i ( t, y, \eta)
\frac{\partial \varphi}{\partial \eta_1} (y, \eta) \theta (t) \, d\eta dy dt
\\
 \displaystyle
= \int_0^T \int_{\Omega \times Y} \xi^1_i ( t, y, \eta) \varphi (y, \eta) \theta
(t) \, d\eta dy dt
\end{array}
\end{eqnarray}
for all $\theta \in {\mathcal D}(0,T)$, $\varphi  \in {\mathcal D} \bigl(\Omega;
{\mathcal C}^{\infty}_{\sharp} (Y) \bigr)$, which implies that $\displaystyle
\xi^1_i = \frac{\partial v^0_i}{\partial \eta_1} \in L^2 \bigl( 0,T; L^2 (\Omega
\times Y) \bigr)$. Thus (\ref{eq:sec4-2}) holds.

Similarly, by multiplying (\ref{sec4.4}) by $\e$ and taking into account
(\ref{sec4.3}) we get
\begin{eqnarray*}
\lim_{\e \to 0} \int_0^T \int_{\Omega} v^{\e}_i (t,y)  \frac{\partial
\varphi}{\partial \eta_2}  \left( y, \frac{y}{\e} \right)   \theta (t) \, dy dt
= 0 = \int_0^T \int_{\Omega \times Y} v^0_i ( t, y, \eta) \frac{\partial
\varphi}{\partial \eta_2}  (y, \eta) \theta (t) \, d\eta dy dt
\end{eqnarray*}
and thus $v_i^0$ does not depend on $\eta_2$. Moreover, by choosing $\varphi$
independent of $\eta$ (i.e $\varphi \in {\mathcal D} (\Omega)$) in
(\ref{sec4.5bis}), we get
\begin{eqnarray*}
\int_0^T \int_{\Omega \times (0,1)} \frac{\partial v^0_i}{\partial \eta_1}  ( t,
y, \eta_1) \varphi (y) \theta (t) \, d\eta_1 dy dt =0 =  \int_0^T \int_{\Omega}
\bigl( v^0_i (t, y, 1) - v^0_i (t,y, 0) \bigr)  \varphi (y) \theta (t) \, dy dt
\end{eqnarray*}
and $v^0_i \in  L^2 \bigl( 0,T; L^2 (\Omega ; H^1_{\sharp} (Y) ) \bigr)$.

Next, by choosing $\varphi  \in {\mathcal D} \bigl(\Omega \times (0,1) \bigr)$
(i.e. $\varphi$ does not depend on $\eta_2$), we obtain now
\begin{eqnarray*}
&& \displaystyle  \lim_{\e \to 0} \int_0^T \int_{\Omega} v^{\e}_i (t,y)
\frac{\partial \varphi}{\partial y_2}  \left( y, \frac{y_1}{\e} \right)   \theta
(t) \, dy dt = \int_0^T \int_{\Omega \times Y} v^0_i ( t, y, \eta_1)
\frac{\partial \varphi}{\partial y_2}  (y, \eta_1) \theta (t) \, d\eta dy dt \\
&& \displaystyle = - \int_0^T \int_{\Omega \times Y} \xi^0_i ( t, y, \eta)
\varphi (y, \eta_1) \theta (t) \, d\eta dy dt.
\end{eqnarray*}
Hence
\begin{eqnarray*}
\int_0^T \int_{\Omega \times Y}  \left( - \frac{\partial v^0_i}{\partial y_2}
(t,y, \eta_1) + \xi^0_i ( t, y, \eta) \right) \varphi (y_1, y_2, \eta_1) \theta
(t) \, d\eta dy dt =0.
\end{eqnarray*}
It follows that there exists $v^1_i \in L^2 \bigl( 0,T; L^2 \bigl(\Omega \times
(0,1); H^1_{\sharp} (0,1)_{| {\br}} \bigr) \bigr)$ such that
\begin{eqnarray*}
 \frac{\partial v^{\e}_i}{\partial y_2} \toH \xi^0_i =  \frac{\partial
v^0_i}{\partial y_2} +  \frac{\partial v^1_i}{\partial \eta_2},
 \end{eqnarray*}
and the second part of (\ref{eq:sec4-1}) holds.


Now, let $\displaystyle \varphi^{\e} (z)= \varphi \left( z_1, \frac{z_2}{\e
h^{\e} ( z_1)}, \frac{z_1}{\e} \right)$ for all $(z_1, z_2 ) \in \Omega^{\e}$.
Recalling that ${\rm div}_z v^{\e} =0$ in $\Omega^{\e}$, we get
\begin{eqnarray*}
&& \displaystyle 0 = \int_0^T \int_{\Omega^{\e}} \left( \frac{\partial
v_1^{\e}}{\partial z_1} (t,z) + \frac{\partial v_2^{\e}}{\partial z_2} (t,z)
\right)
\varphi^{\e} (z) \theta (t) \, dz dt  \\
&& \displaystyle = - \int_0^T \int_{\Omega^{\e}} \left( v_1^{\e} (t,z)
\frac{\partial \varphi^{\e}}{\partial z_1}
(z) + v_2^{\e} (t,z) \frac{\partial \varphi^{\e}}{\partial z_2}
(z) \right) \theta (t) \, dz dt \\
&& = \displaystyle - \int_0^T \int_{\Omega} \left(
v_1^{\e} (t, y) \bigl(b_{\e} \cdot \nabla \varphi^{\e} \bigr) (y) + v_2^{\e} (t,
y) \frac{1}{\e h^{\e} (y_1)}
\frac{\partial \varphi^{\e}}{\partial y_2} (y) \right)\e h^{\e}(y_{1})
\theta(t) \, dy dt \\
&& = \displaystyle - \int_0^T \int_{\Omega}
v_1^{\e} (t,y)  \left(
\e h \left( y_1, \frac{y_1}{\e} \right)  \frac{\partial \varphi}{\partial y_1}
\left( y, \frac{y}{\e} \right)
+ h \left( y_1, \frac{y_1}{\e} \right)  \frac{\partial \varphi}{\partial \eta_1}
\left( y, \frac{y}{\e} \right)  \right.\\
&& \displaystyle  \left. - y_2 \left(
\e \frac{\partial h}{\partial y_1} \left( y_1, \frac{y_1}{\e} \right) +
\frac{\partial h}{\partial \eta_1} \left( y_1, \frac{y_1}{\e} \right)
\right)
\frac{\partial \varphi}{\partial y_2} \left( y, \frac{y}{\e} \right)
\right) \theta (t) \, dy dt \\
&& \displaystyle  - \int_0^T \int_{\Omega} v_2^{\e} (t,y)
 \frac{\partial \varphi}{\partial y_2} \left( y, \frac{y}{\e} \right)
  \theta (t) \, dy dt.
\end{eqnarray*}
We pass to the limit as $\e $ tends to zero:
\begin{eqnarray*}
&& \displaystyle 0 =  - \int_0^T \int_{\Omega \times (0,1) } v_1^{0} (t, y,
\eta_1)  \left(
h ( y_1, \eta_1)   \frac{\partial \varphi}{\partial \eta_1} (y, \eta_1)
- y_2
\frac{\partial h}{\partial \eta_1} (y_1, \eta_1)    \frac{\partial
\varphi}{\partial y_2} ( y, \eta_1) \right) \theta (t) \, d \eta_1 dy dt \\
&& \displaystyle  - \int_0^T \int_{\Omega \times (0,1) } v_2^{0} (t, y,\eta_1)
\frac{\partial \varphi}{\partial y_2} (y, \eta_1)  \theta (t) \, d \eta_1 dy dt
\\
&& \displaystyle = \int_0^T \int_{\Omega \times (0,1) } \left(
h(y_1, \eta_1) \frac{\partial v^0_1}{\partial \eta_1} (t, y, \eta_1)
 - y_2 \frac{\partial h}{\partial \eta_1} (y_1, \eta_1) \frac{\partial
v^0_1}{\partial y_2} (t, y, \eta_1)
+ \frac{\partial v^0_2}{\partial y_2} (t, y, \eta_1) \right) \varphi (y, \eta_1)
\theta (t) \, d\eta_1 dy dt
\end{eqnarray*}
which gives (\ref{eq:sec4-3}).
But, taking into account the boundary conditions on $v^{\e}$, we may reproduce
the same computation with $\varphi \in {\mathcal C}^{\infty} \bigl(\overline{
\Omega}; {\mathcal C}^{\infty}_{\sharp} (0,1) \bigr)$ such that $\varphi$ is
$L$-periodic in $y_1$,  so with (\ref{eq:sec4-3}) it remains
\begin{eqnarray*}
&& \displaystyle \int_0^T \int_{\Gamma_1 \times (0,1)} \left( - v_1^0 (t, y,
\eta_1) \frac{\partial h}{\partial \eta_1} (y_1, \eta_1) + v_2^0 (t, y, \eta_1)
\right) \varphi (y, \eta_1) \theta (t) \, d \eta_1 dy_1 dt \\
&& \displaystyle - \int_0^T \int_{\Gamma_0 \times (0,1)} v_2^0 (t, y, \eta_1)
\varphi (y, \eta_1) \theta(t) \, d \eta_1 dy_1 dt = 0.
\end{eqnarray*}
We choose more precisely $\varphi (y_1, y_2, \eta_1) = \hat \varphi (y_2) \tilde
\varphi (y_1, \eta_1)$ with $\hat \varphi \in {\mathcal C}^{\infty} \bigl(
[0,1])$ and $\tilde \varphi \in {\mathcal C}^{\infty}_{\sharp} \bigl( [0,L];
{\mathcal C}^{\infty}_{\sharp} (0,1) \bigr)$. With $\hat \varphi (1)=0$ and
$\hat \varphi (0) = 1$
we get first
\begin{eqnarray*}
v_2^0 = 0 \quad \hbox{\rm on } (0,T) \times \Gamma_0 \times (0,1).
\end{eqnarray*}
Next with $\hat \varphi(1)=1$ and $\hat \varphi (0)=0$
we get
\begin{eqnarray*}
- v_1^0  \frac{\partial h}{\partial \eta_1} + v_2^0  = 0 \quad \hbox{\rm on }
(0,T) \times \Gamma_1 \times (0,1).
\end{eqnarray*}

Finally, let $\varphi \in {\mathcal D} \bigl( 0,L; {\mathcal
C}^{\infty}_{\sharp} (0,1) \bigr)$ and $\displaystyle \varphi^{\e} (y_1, y_2) =
\varphi \left( y_1, \frac{y_1}{\e} \right) (1-y_2)$ for all $(y_1, y_2) \in
\Omega$. Taking into account the boundary conditions for $v_1^{\e}$ (see
(\ref{eqn:er2.8a})-(\ref{eqn:er2.11b})) we have
\begin{eqnarray*}
\int_0^T \int_{\Omega} \frac{\partial v_1^{\e}}{\partial y_2} (t,y) \varphi^{\e}
(y) \theta (t) \, dy dt = \int_0^T \int_{\Omega} v_1^{\e} (t,y) \varphi \left(
y_1, \frac{y_1}{\e} \right) \theta (t) \, dy dt.
\end{eqnarray*}
By passing to the limit as $\e$ tends to zero we obtain
\begin{eqnarray*}
&& \displaystyle \int_0^T \int_{\Omega \times (0,1)} \frac{\partial
v_1^{0}}{\partial y_2} (t,y, \eta_1) \varphi (y_1, \eta_1) (1-y_2) \theta (t) \,
d \eta_1 dy dt \\
&& \displaystyle = \int_0^T \int_{\Omega \times (0,1)} v_1^{0} (t,y, \eta_1)
\varphi ( y_1, \eta_1) \theta (t) \, d \eta_1 dy dt.
\end{eqnarray*}
It follows that
\begin{eqnarray*}
\int_0^T \int_{\Gamma_0 \times (0,1)} v_1^{0} (t,y, \eta_1) \varphi ( y_1,
\eta_1) \theta (t) \,  d \eta_1 dy dt =0
\end{eqnarray*}
which implies that
$ v_1^0 = 0$   on $(0,T)\times \Gamma_0 \times (0,1)$.
\end{proof}

Similarly we can define the two-scale limit of $Z^{\e}$.

\begin{proposition} \label{prop4.2}
{\bf (Two-scale limit of the micro-rotation field)} Under the assumptions of
Proposition \ref{pro1}, there exist $Z^0 \in L^2 \bigl( 0,T; L^2 (\Omega;
H^1_{\sharp}(Y) ) \bigr)$ such that $\displaystyle \frac{\partial Z^0}{\partial
y_2} \in L^2 \bigl( 0,T; L^2 (\Omega \times Y ) \bigr)$  and $Z^1 \in L^2 \bigl(
0,T; L^2 \bigl(\Omega \times (0,1) ; H^1_{\sharp}(0,1)_{/ \br} \bigr) \bigr)$
such that, possibly extracting a subsequence still denoted $(Z^{\e})_{\e >0}$,
we have
\begin{eqnarray}\label{eq:sec4-5}
Z^{\varepsilon} \toH Z^0,  \quad \frac{\partial Z^{\varepsilon}}{\partial y_2}
\toH \frac{\partial Z^0}{\partial y_2} + \frac{\partial Z^1}{\partial \eta_2},
\end{eqnarray}
and
\begin{eqnarray} \label{eq:sec4-6}
\varepsilon \frac{\partial Z^{\e}}{\partial y_1} \toH \frac{\partial
Z^0}{\partial \eta_1}.
\end{eqnarray}
Furthermore $Z^0$ does not depend on $\eta_2$, and $Z^0 \equiv 0$ on $(\Gamma_0
\cup \Gamma_1) \times (0,1) \times (0,T)$.
\end{proposition}
\begin{proof}
The first part of the proof is identical to the proof of the previous
proposition. Let us establish now the boundary conditions for the limit $Z^0$.
Let $\theta \in {\mathcal D}(0,T)$, $\varphi \in {\mathcal C}^{\infty}
\bigl(\overline{ \Omega}; {\mathcal C}^{\infty}_{\sharp} (0,1) \bigr)$ such that
$\varphi$ is $L$-periodic in $y_1$ and we define
$\displaystyle \varphi^{\e} (z)= \varphi \left( z_1, \frac{z_2}{\e h^{\e} (
z_1)}, \frac{z_1}{\e} \right)$ for all $(z_1, z_2) \in \Omega^{\e}$.
With the boundary conditions (\ref{eqn:er2.8a})-(\ref{eqn:er2.11b}) for $Z^{\e}$
we get
\begin{eqnarray*}
&& \displaystyle \int_0^T \int_{\Omega^{\e}} \frac{\partial Z^{\e}}{\partial
z_2} (t,z) \varphi^{\e} (z) \theta (t) \, dz dt =
-  \int_0^T \int_{\Omega^{\e}} Z^{\e} (t,z) \frac{\partial \varphi^{\e}
}{\partial z_2} (z) \theta (t) \, dz dt \\
&& \displaystyle =  - \int_0^T \int_{\Omega} Z^{\e} (t,y) \frac{\partial
\varphi}{\partial y_2} \left( y, \frac{y_1}{\e} \right) \theta (t) \, dy dt =
\int_0^T \int_{\Omega} \frac{\partial Z^{\e}}{\partial y_2} (t,y) \varphi \left(
y, \frac{y_1}{\e} \right) \theta (t) \, dy dt
\end{eqnarray*}
and taking $\e \to 0^+$ we obtain
\begin{eqnarray*}
&& \displaystyle \int_0^T \int_{\Omega \times Y} \left( \frac{\partial
Z^0}{\partial y_2} (t, y, \eta_1) +  \frac{\partial Z^1}{\partial \eta_2} (t, y,
\eta) \right) \varphi (y, \eta_1) \theta (t) \, d\eta_1 dy dt\\
&& \displaystyle  = - \int_0^T \int_{\Omega \times (0,1)}  Z^0 (t, y, \eta_1)
\frac{\partial \varphi}{\partial y_2} (y, \eta_1) \theta (t) \, d \eta_1 dy dt.
\end{eqnarray*}
But the periodicity properties of $Z^1$ with respect to $\eta_2$ imply that
\begin{eqnarray*}
\int_0^T \int_{\Omega \times Y}  \frac{\partial Z^1}{\partial \eta_2} (t, y,
\eta) \varphi (y, \eta_1) \theta (t) \, d\eta_1 dy dt =0.
\end{eqnarray*}
Hence
\begin{eqnarray*}
\int_0^T \int_{\Omega \times (0,1)} \frac{\partial Z^0}{\partial y_2} (t, y,
\eta_1)  \varphi (y, \eta_1) \theta (t) \, d\eta_1 dy dt = - \int_0^T
\int_{\Omega \times (0,1)}  Z^0 (t, y, \eta_1) \frac{\partial \varphi}{\partial
y_2} (y, \eta_1) \theta (t) \, d \eta_1 dy dt.
\end{eqnarray*}
By Green's formula we infer that
\begin{eqnarray*}
0 = - \int_0^T \int_{\Gamma_0 \times (0,1)} Z^0 (t, y, \eta_1) \varphi (y,
\eta_1) \theta (t) \, d \eta_1 dy dt + \int_0^T \int_{\Gamma_1 \times (0,1)} Z^0
(t, y, \eta_1) \varphi (y, \eta_1) \theta (t) \, d \eta_1 dy dt.
\end{eqnarray*}

Now we choose $\varphi (y_1, y_2, \eta_1) = \hat \varphi (y_2) \tilde \varphi
(y_1, \eta_1)$ with $\hat \varphi \in {\mathcal C}^{\infty} \bigl( [0,1] \bigr)$
and $\tilde \varphi \in  {\mathcal C}^{\infty}_{\sharp}( [0,L]; {\mathcal
C}^{\infty}_{\sharp} (0,1) \bigr)$, with $\hat \varphi (1)=0$, $\hat \varphi (0)
= 1$ then $\hat \varphi (1)=1$, $\hat \varphi (0) = 0$, we get
\begin{eqnarray*}
0= \int_0^T \int_{\Gamma_0 \times (0,1)} Z^0 (t, y, \eta_1) \tilde \varphi (y_1,
\eta_1) \theta (t) \, d \eta_1 dy dt  =  \int_0^T \int_{\Gamma_1 \times (0,1)}
Z^0 (t, y, \eta_1) \tilde \varphi (y_1, \eta_1) \theta (t) \, d \eta_1 dy dt
\end{eqnarray*}
which allows us to conclude the proof of Proposition \ref{prop4.2}.
\end{proof}

Finally we can define the two-scale limit of $p^{\e}$.

\begin{proposition} \label{prop4.3}
{\bf (Two-scale limit of the pressure field)} Under the assumptions of
Proposition \ref{prop2},
there exists $p^0 \in H^{-1} \bigl( 0,T; L^2 (\Omega \times Y) \bigr)$ such
that, possibly extracting a subsequence still denoted $(p^{\e})_{\e >0}$, we
have
\begin{eqnarray*}
\e^2 p^{\e} \toH p^0.
\end{eqnarray*}
Moreover $p^0$ depends only $t$ and $y_1$, $p^0 \in H^{-1} \bigl( 0,T;
H^1_{\sharp} (0,1) \bigr)$ and satisfies \newline
$ \displaystyle \int_0^L p^0 (t,y_1) \left( \int_0^1   h(y_1, \eta_1) \, d
\eta_1 \right) \, dy_1 =0$ almost everywhere in $(0,T)$.
\end{proposition}

\begin{proof} The first part of the result is an immediate consequence of the
estimate (\ref{E3.18}) (see Proposition \ref{prop2}).
From proposition \ref{pro1} and (\ref{eq3.43}) we know that there exists a
constant $C>0$, independent of $\e$, such that for all $\varphi^{\e} \in H^1_0
\bigl(0,T; H^1_0(\Omega) \bigr)$ we have
\begin{eqnarray*}
&& \displaystyle \left| \int_0^T \int_{\Omega} p^{\e} (t,y) \frac{\partial
\varphi^{\e}}{\partial y_2} (t,y) \, dy dt \right|
\le  C  \left( \| \varphi^{\e} \|_{L^{2} (0,T; L^2(\Omega))} + \e \left\|
\frac{\partial \varphi^{\e}}{\partial t} \right\|_{L^2(0,T; L^2(\Omega))}
\right) \\
&& \displaystyle + \frac{C}{\e} \left(  \| \varphi^{\e} \|_{L^{\infty} (0,T;
L^4(\Omega))} + \| \e b_{\e} \cdot \nabla \varphi^{\e} \|_{L^{2} (0,T;
L^2(\Omega))} +  \left\| \frac{\partial \varphi^{\e}}{\partial y_2}
\right\|_{L^2(0,T; L^2(\Omega))} \right) .
\end{eqnarray*}
Now let $\varphi \in {\mathcal D}\bigl( \Omega; {\mathcal C}^{\infty}_{\sharp}
(Y) \bigr)$ and $\theta \in {\mathcal D}(0,T)$. We define $\varphi^{\e}(t,y) =
\theta (t) \varphi \left( y, \frac{y}{\e} \right)$ for all $(t,y) \in (0,T)
\times \Omega$ and we get
\begin{eqnarray}\label{eq:4.3.1}
\left| \int_0^T \int_{\Omega} \e^2 p^{\e} (t, y) \left( \frac{\partial
\varphi}{\partial y_2} \left(t, y, \frac{y}{\e} \right) + \frac{1}{\e}
\frac{\partial \varphi}{\partial \eta_2} \left(t, y, \frac{y}{\e} \right)
\right) \theta (t)  \, dy dt \right|
\nonumber\\
\le {\mathcal O}(\e) + C \|\theta\|_{C^0
([0,T])} \left\| \frac{\partial \varphi}{\partial \eta_2} \right\|_{C^0
(\overline{\Omega}; C_{\sharp}(Y))}.
\end{eqnarray}
We multiply the two members of this inequality by $\e$ and we pass to the limit
as $\e$ tends to zero. We obtain
\begin{eqnarray*}
\int_0^T \int_{\Omega \times Y} p^0 (t, y, \eta) \frac{\partial
\varphi}{\partial \eta_2} (y, \eta) \theta (t) \, d \eta dy dt =0.
\end{eqnarray*}
Hence $p^0$ does not depend on $\eta_2$. Now we consider $\varphi \in {\mathcal
D}\bigl( \Omega; {\mathcal C}^{\infty}_{\sharp} (0,1) \bigr)$ (i.e $\varphi$ is
independent of $\eta_2$) and we pass to the limit in (\ref{eq:4.3.1}) as $\e$
tends to zero. We get
\begin{eqnarray*}
\int_0^T \int_{\Omega \times (0,1)} p^0 (t, y, \eta) \frac{\partial
\varphi}{\partial y_2} (y, \eta_1) \theta (t) \, d \eta dy dt =0
\end{eqnarray*}
which implies that $p^0$ does not depend on $y_2$.

Now we take $\Theta = (\varphi^{\e}, 0, 0)$ in (\ref{eqvar}) and we multiply by
$\e$: we get
 \begin{eqnarray}\label{e3.43}
\e^{2} \int_{0}^{T}\int_{\Om} p^{\e}
\left( {\partial\varphi^{\e} \over \partial y_{1}}
- {y_{2}\over h^{\e}}  {\partial h^{\e}\over \partial y_{1}}
{\partial\varphi^{\e} \over \partial y_{2}}
\right) h^{\e} dy dt
=
-\e^{2}\int_{0}^{T} \int_{\Om} v^{\e}_{1} {\partial\varphi^{\e}\over \partial t}
h^{\e}
dy dt
\nonumber\\
+ (\nu+\nu_{r})
\int_{0}^{T} \int_{\Om}\left(
h^{\e}(\e b_{\e}\cdot  \nabla v^{\e}_{1}) (\e b_{\e}\cdot\nabla\varphi^{\e})
+{1\over h^{\e}}{\partial v^{\e}_{1}\over \partial y_{2}}
           {\partial\varphi^{\e}\over \partial y_{2}}\right) dy dt
\nonumber\\
+\int_{0}^{T}\int_{\Om} \left(
\e v^{\e}_{1}(\e b_{\e}\cdot  \nabla v^{\e}_{1})\varphi^{\e} h^{\e}
 + \e v^{\e}_{2}{\partial v^{\e}_{1}\over \partial
y_{2}}\varphi^{\e} \right)  dy
dt
+ (\nu+\nu_{r})\int_{0}^{T} \int_{\Om} \frac{1}{h^{\e}}
{\partial U_{\e}\over \partial y_{2}}{\partial\varphi^{\e} \over \partial y_{2}}
dy
dt
\nonumber\\
-2\nu_{r}
\int_{0}^{T} \int_{\Om} \e
{\partial Z^{\e}\over \partial y_{2}}\varphi^{\e}   dy dt
+ \int_{0}^{T} \int_{\Om}U_{\e}
(\e b_{\e}\cdot  \nabla v^{\e}_{1})\varphi^{\e} \e h^{\e} dy dt
\nonumber\\
- \int_{0}^{T} \int_{\Om}
\e v^{\e}_{1}(\e b_{\e}\cdot  \nabla\varphi^{\e})U_{\e} h^{\e} dy dt
-  \int_{0}^{T} \int_{\Om}
\e v^{\e}_{2}{\partial \varphi^{\e} \over \partial y_{2}}U_{\e} dy dt
-2\e \nu_{r}\int_{0}^{T} \int_{\Om}
     {\partial W_{\e}\over \partial y_{2}}\varphi^{\e} dy dt
\nonumber\\
- \int_{0}^{T} \int_{\Om}
\left( f^{\e}_{1}\varphi^{\e} -{\partial U^{\e} \over \partial
t}\varphi^{\e} \right) \e^{2} h^{\e} dy dt,
\end{eqnarray}
where we recall that $U_{\e} (t, y) = U_0(t) {\mathcal U}(y_2 h^{\e} (y_1))$ and
$W_{\e} (t, y) = W_0(t) {\mathcal W}(y_2 h^{\e} (y_1))$ for all $(t, y_1, y_2)
\in [0,T] \times \Om$.
With the results of Proposition \ref{pro1},  we infer  that there exists a
constant $C>0$, independent of $\e$, such that
\begin{eqnarray*}
&& \displaystyle  \left| \int_0^T \int_{\Omega} \e^2 p^{\e} (t,y) ( b_{\e} \cdot
\nabla \varphi^{\e})  (t,y) h^{\e} (y) \, dy dt \right|  \\
&& \displaystyle
\le C \left( \| \varphi^{\e} \|_{L^{2}(0,T; L^2(\Omega))} + \| \e b_{\e} \cdot
\nabla \varphi^{\e} \|_{L^{2}(0,T; L^2(\Omega))}  + \left\| \frac{\partial
\varphi^{\e}}{\partial y_2} \right\|_{L^{2}(0,T; L^2(\Omega))}
+  \| \varphi^{\e} \|_{L^{\infty}(0,T; L^4(\Omega))}\right) \\
&& \displaystyle  + C \e^2 \left\| \frac{\partial \varphi^{\e}}{\partial t}
\right\|_{L^{2}(0,T; L^2(\Omega))} .
\end{eqnarray*}
We multiply the two members of this inequality by $\e$ and we obtain
\begin{eqnarray*}
&& \displaystyle \left| \int_0^T \int_{\Omega} \e^2 p^{\e} (t,y) \left( \e
\frac{\partial \varphi}{\partial y_1}
\left(y_1, y_2, \frac{y_1}{\e} \right) +  \frac{\partial \varphi}{\partial
\eta_1}
\left(y_1, y_2, \frac{y_1}{\e} \right)
 \right) h \left(y_1, \frac{y_1}{\e} \right)  \theta (t) \, dy dt  \right. \\
 && \displaystyle \left.
- \int_0^T \int_{\Omega} \e^2 p^{\e} (t,y) y_2  \left( \e \frac{\partial
h}{\partial y_1}
\left(y_1, \frac{y_1}{\e} \right) +  \frac{\partial h}{\partial \eta_1}
\left(y_1, \frac{y_1}{\e} \right)
 \right) \frac{\partial \varphi}{\partial y_2} \left(y_1, y_2, \frac{y_1}{\e}
\right)  \theta (t) \, dy dt
 \right|  \le {\mathcal O}(\e)
\end{eqnarray*}
By taking the limit as $\e$ tends to zero, we have
\begin{eqnarray*}
\int_0^T \int_{\Om \times (0,1)} p^0 (t, y_1, \eta_1)\left(   \frac{\partial
\varphi}{\partial \eta_1} (y_1, y_2, \eta_1)  h( y_1, \eta_1)  - y_2
\frac{\partial h}{\partial \eta_1}
(y_1, \eta_1) \frac{\partial \varphi}{\partial y_2} (y_1, y_2, \eta_1) \right)
\theta(t) \, d \eta_1 d y dt =0.
\end{eqnarray*}
Reminding that $p^0$ is independent of $y_2$ and $\varphi \in {\mathcal D}\bigl(
\Omega; {\mathcal C}^{\infty}_{\sharp} (0,1) \bigr)$, we get
\begin{eqnarray*}
&& \displaystyle \int_0^T \int_{\Om \times (0,1)} p^0 (t, y_1, \eta_1)\left(
\frac{\partial \varphi}{\partial \eta_1} (y_1, \eta_1)  h( y_1, y_2, \eta_1)  -
y_2 \frac{\partial h}{\partial \eta_1}
(y_1, \eta_1) \frac{\partial \varphi}{\partial y_2} (y_1, y_2, \eta_1) \right)
\theta(t) \, d \eta_1 d y dt \\
&& \displaystyle =  \int_0^T \int_{ \Om \times (0,1)} p^0 (t, y_1, \eta_1)\left(
  \frac{\partial \varphi}{\partial \eta_1} (y_1, y_2, \eta_1)  h( y_1, \eta_1)
+ \frac{\partial h}{\partial \eta_1}
(y_1, \eta_1)  \varphi (y_1, y_2, \eta_1) \right)
\theta(t) \, d \eta_1 d y dt \\
&& \displaystyle =  \int_0^T \int_{\Om \times (0,1)} p^0 (t, y_1, \eta_1) \frac{
\partial ( h \varphi) }{\partial \eta_1} (y_1, y_2, \eta_1) \, d \eta_1 dy dt
=0.
\end{eqnarray*}
Then for any $\phi \in {\mathcal D}\bigl( \Omega; {\mathcal C}^{\infty}_{\sharp}
(0,1) \bigr)$, we may define $\displaystyle \varphi = \frac{\phi}{h} \in
{\mathcal D}\bigl( \Omega; {\mathcal C}^{\infty}_{\sharp} (0,1) \bigr)$ and we
obtain
\begin{eqnarray*}
 \int_0^T \int_{\Om \times (0,1)} p^0 (t, y_1, \eta_1) \frac{ \partial
\phi}{\partial \eta_1} (y_1, y_2, \eta_1) \, d \eta_1 dy dt
=0.
\end{eqnarray*}
Thus we can conclude that $p^0 $ is independent of $\eta_1$.

Now let $\varphi \in {\mathcal C}^{\infty}_{\sharp} (0,L) $ and $\theta \in
{\mathcal D}(0,T)$. We define
$\varphi^{\e}$  by
\begin{eqnarray}\label{eq:rajout1}
\varphi^{\e} (y) = \frac{\varphi(y_1)}{h\left(y_1, \frac{y_1}{\e} \right)}
\left( y_2 e_1 + \e y_2^2 \left(  \frac{\partial h}{\partial y_1} \left(y_1,
\frac{y_1}{\e} \right) + \frac{1}{\e}  \frac{\partial h}{\partial \eta_1}
\left(y_1, \frac{y_1}{\e} \right) \right) e_2 \right)
\end{eqnarray}
for all $(y_1, y_2) \in \Om$. We can check that $\varphi^{\e} \in \tilde V$ and with Lemma \ref{lemma3.1},
Proposition \ref{pro1} and (\ref{eq3.43})-(\ref{e3.43}), we obtain
\begin{eqnarray*}
\displaystyle \left| \int_0^T \int_{\Om} \e^2 p^{\e} \left( (b_{\e} \cdot \nabla
\varphi^{\e}_1) + \frac{1}{\e h^{\e}} \frac{\partial \varphi^{\e}_2}{\partial
y_2} \right) h^{\e} \theta (t) \, dy dt \right|
 \displaystyle \le  {\mathcal O}( \e) + C \| \varphi \theta \|_{L^2 ((0,T)
\times (0,L))}
\end{eqnarray*}
with a constant $C>0$ independent of $\e$. Hence
\begin{eqnarray*}
\left|  \int_0^T \int_{\Om} \e^2 p^{\e} (t,y) y_2 \frac{\partial
\varphi}{\partial y_1} (y_1) \theta (t) \, dy dt \right|
\le  {\mathcal O}( \e) + C \| \varphi \theta \|_{L^2 ((0,T) \times (0,L))} .
\end{eqnarray*}
We pass to the limit as $\e$ tends to zero:
\begin{eqnarray*}
\left|  \int_0^T \int_{\Om}  p^{0} (t,y_1) y_2 \frac{\partial
\varphi}{\partial y_1} (y_1)   \theta (t) \, dy dt \right| = \frac{1}{2}
\left|  \int_0^T \int_0^L p^{0} (t,y_1)  \frac{\partial \varphi}{\partial y_1}
(y_1)   \theta (t) \, dy dt \right|
\le   C \| \varphi \theta \|_{L^2 ((0,T) \times (0,L))}
\end{eqnarray*}
and we infer that $p^0 \in H^{-1} \bigl( 0,T; H^1_{\sharp} (0,L) \bigr)$.

Finally, recalling that $\displaystyle \int_{\Omega^{\e}} p^{\e} (t,z) \, dz =0$
almost everywhere in $(0,T)$, we have
\begin{eqnarray*}
\int_0^T \int_{\Omega} {\e}^2 p^{\e} (t,y) h^{\e} (y) \theta (t) \, dy dt =0
\quad \forall \theta \in {\mathcal D} (0,T)
\end{eqnarray*}
and by passing to the limit as $\e$ tends to zero, we get
\begin{eqnarray*}
\int_0^T \int_{\Omega \times (0,1)}  p^{0} (t,y_1, \eta_1) h (y_1, \eta_1)
\theta (t) \, d \eta_1 dy dt =0 \quad \forall \theta \in {\mathcal D} (0,T)
\end{eqnarray*}
which allows us to conclude the proof of Proposition \ref{prop4.3}.
\end{proof}

\renewcommand{\theequation}{5.\arabic{equation}}
\setcounter{equation}{0}
\section{The limit problem}\label{limitprobl}

Now let us pass to the limit in equation (\ref{eqn:er2.14}). It is convenient to
introduce the following functional spaces:
\begin{eqnarray*}
&& \displaystyle \tilde V = \left\{ \varphi \in \bigl( {\mathcal C}^{\infty} (
\overline{ \Omega} ; {\mathcal C}^{\infty}_{\sharp} (0,1) \bigr) \bigr)^2; \,
 \varphi \,  \hbox{\rm is L-periodic in }  y_{1}, \ \varphi = 0 \  \hbox{\rm on }
\Gamma_0 \times (0,1),  \right. \\
&& \displaystyle \left. \qquad \ \   \ - \varphi_1 \frac{\partial h}{\partial
\eta_1}+ \varphi_2 =0 \ \hbox{\rm on } \Gamma_1 \times (0,1) \right\}
\end{eqnarray*}
\begin{eqnarray*}
\tilde V_{div} = \left\{ \varphi \in \tilde V; \ h \frac{\partial
\varphi_1}{\partial \eta_1} - y_2 \frac{\partial h}{\partial \eta_1}
\frac{\partial \varphi_1}{\partial y_2} + \frac{\partial \varphi_2}{\partial
y_2} =0 \ \hbox{\rm in } \Omega \times (0,1) \right\}
\end{eqnarray*}
\begin{eqnarray*}
\tilde H^1 = \left\{ \psi \in {\mathcal C}^{\infty} (\overline{ \Omega} ;
{\mathcal C}^{\infty}_{\sharp} (0,1) \bigr) ;  \,   \psi \mbox{  is
L-periodic in }  y_{1},  \, \psi=0 \ \mbox{ on } (\Gamma_{0} \cup \Gamma_{1}
\times (0,1) \right\} ,
\end{eqnarray*}
\begin{eqnarray*}
\begin{array}{l}
\displaystyle V_{div} = \mbox{ closure of } \tilde V_{div} \mbox{ in }
L^2_{\sharp}
\bigl(0,L; {\mathcal F} \bigr)^{2},\\
\displaystyle  H^1_{0 , \sharp} = \hbox{\rm closure of } \tilde H^1 \mbox{ in }
L^2_{\sharp} \bigl( 0,L; {\mathcal F} \bigr)
\end{array}
\end{eqnarray*}
with
\begin{eqnarray*}
{\mathcal F} = \left\{ v \in L^2 \bigl( (0,1);  H^1_{\sharp} (0,1) \bigr) ,
\frac{\partial v}{\partial y_2} \in L^2 \bigl((0,1)  \times (0,1) \bigr)
\right\}.
\end{eqnarray*}

\begin{theorem} \label{limit_pb}
Assume that there exist $f \in {\mathcal C}\bigl( [0,T] ; {\mathcal C} \bigl(
\overline{\Om}; {\mathcal C}_{\sharp}(0,1) \bigr) \bigr)^2$ and $g \in {\mathcal
C}\bigl( [0,T] ; {\mathcal C} \bigl( \overline{\Om}; {\mathcal C}_{\sharp}(0,1)
\bigr) \bigr)$ such that $f$ and $g$ are $L$-periodic in $y_1$ and
\begin{eqnarray*}
\e^2 f^{\e} (t,y) = f \left( t, y, \frac{y_1}{\e} \right), \quad \e^2 g^{\e}
(t,y) = g \left( t, y, \frac{y_1}{\e} \right) \quad \forall (t,y) \in [0,T]
\times \overline{\Om}.
\end{eqnarray*}
Then the functions $v^0$, $Z^0$ and $p^0$ satisfy the following limit problem:
\begin{eqnarray*}
&& \displaystyle  ( \nu + \nu_r)  \int_0^T \int_{\Omega \times (0,1) }
\sum_{i=1}^2 \left(
h (\bar b \cdot \nabla v_i^0) (\bar b \cdot \nabla \varphi_i) + \frac{1}{h}
\frac{\partial v_i^0}{\partial y_2} \frac{\partial \varphi_i}{\partial y_2}
\right) \theta \, d\eta_1 d y dt \\
&& \displaystyle + \alpha \int_0^T  \int_{\Omega \times (0,1)} \left(
h (\bar b \cdot \nabla Z^0) (\bar b \cdot \nabla \psi) + \frac{1}{h}
\frac{\partial Z^0}{\partial y_2}
 \frac{\partial \psi}{\partial y_2} \right) \theta \, d\eta_1 d y dt \\
&& \displaystyle  - \int_0^T \int_{\Omega \times (0,1)} \frac{\partial
p^0}{\partial y_1} h \varphi_1 \theta \, d \eta_1 dy dt \\
&& \displaystyle = -  ( \nu + \nu_r)  \int_0^T \int_{\Omega \times (0,1) }
\left(
h (\bar b \cdot \nabla \bar U) (\bar b \cdot \nabla \varphi_1) + \frac{1}{h}
\frac{\partial \bar U}{\partial y_2} \frac{\partial \varphi_1}{\partial y_2}
\right) \theta \, d\eta_1 d y dt \\
&& \displaystyle  - \alpha \int_0^T  \int_{\Omega \times (0,1)} \left(
h (\bar b \cdot \nabla \bar W) (\bar b \cdot \nabla \psi) + \frac{1}{h}
\frac{\partial \bar W}{\partial y_2}
 \frac{\partial \psi}{\partial y_2} \right) \theta \, d\eta_1 d y dt \\
&& \displaystyle + \int_0^T  \int_{\Omega \times (0,1)} f \varphi h \theta \, d
\eta_1 d y dt
+ \int_0^T  \int_{\Omega \times (0,1)} g \psi h \theta \, d \eta_1 d y dt
\end{eqnarray*}
for all $\Theta = (\varphi, \psi) \in  V_{div} \times  H^1_{0 {\sharp}}$ and
$\theta \in {\mathcal D} (0,T)$, where $\bar b \cdot \nabla$ is the differential
operator defined by
 \begin{eqnarray*}
\bar b \cdot \nabla = \left( 1, - \frac{y_2}{h (y_1, \eta_1)} \frac{\partial
h}{\partial \eta_1} (y_1, \eta_1) \right)
\left(\begin{array}[c]{c}
\displaystyle  {\partial \over\partial \eta_{1}} \\ \\
\displaystyle {\partial \over\partial y_{2}}
\end{array}
\right)
\end{eqnarray*}
and
\begin{eqnarray*}
&& \displaystyle   \bar U ( t, y_1, y_2, \eta_1) = U_0 (t) {\mathcal U} \left(
h(y_1, \eta_1) y_2 \right), \\
&& \displaystyle \bar W (t,  y_1, y_2, \eta_1) = W_0 (t) {\mathcal W}  \left(
h(y_1, \eta_1) y_2 \right)
 \end{eqnarray*}
for all $(t, y_1, y_2, \eta_1, t) \in [0,T] \times \overline{\Omega} \times
[0,1] $.
\end{theorem}
\begin{proof} With the above assumptions for $f^{\e}$ and $g^{\e}$ we can check
immediately that
\begin{eqnarray*}
\| \e^2 f^{\e} \|_{L^2 ((0,T) \times \Om)} \le \sqrt{ T | \Om | } \|
f\|_{{\mathcal C}\bigl( [0,T] ; {\mathcal C} \bigl( \overline{\Om}; {\mathcal
C}_{\sharp}(0,1) \bigr) \bigr)}, \quad
\| \e^2 g^{\e} \|_{L^2 ((0,T) \times \Om )} \le \sqrt{ T | \Om | } \| g
\|_{{\mathcal C}\bigl( [0,T] ; {\mathcal C} \bigl( \overline{\Om}; {\mathcal
C}_{\sharp}(0,1) \bigr) \bigr)}
\end{eqnarray*}
and
\begin{eqnarray*}
\e^2 f^{\e} \toH f, \quad \e^2 g^{\e} \toH g.
\end{eqnarray*}
Let us recall that
\begin{eqnarray*}
b_{\e} \cdot \nabla = \left( 1, - \frac{y_2}{h^{\e} (y_1)} \frac{\partial
h^{\e}}{\partial y_1} (y_1) \right)
\left(\begin{array}[c]{c}
\displaystyle  {\partial \over\partial y_{1}} \\ \\
\displaystyle {\partial \over\partial y_{2}}
\end{array}
\right)
\end{eqnarray*}
 Taking into account the convergence results of Proposition \ref{prop4.1} and
Proposition \ref{prop4.2}, we get
 \begin{eqnarray*}
&& \displaystyle \e  b_{\e} \cdot  \nabla v_i^{\e} = \e \frac{\partial
v_i^{\e}}{\partial y_1} (y)
 - \frac{y_2}{h \left( y_1, \frac{y_1}{\e} \right)} \left( \e \frac{\partial
h}{\partial y_1}  \left( y_1, \frac{y_1}{\e} \right) +  \frac{\partial
h}{\partial \eta_1}  \left( y_1, \frac{y_1}{\e} \right) \right)
 \frac{\partial v_i^{\e} }{\partial y_2} (y) \\
&& \displaystyle  \toH \frac{\partial v^0_i}{\partial \eta_1} (y, \eta_1)
  - \frac{y_2}{h (y_1, \eta_1) }  \frac{\partial h}{\partial \eta_1}  (y_1,
\eta_1)
\left(  \frac{\partial v^0_i }{\partial y_2}  ( y, \eta_1  ) + \frac{\partial
v^1_i}{\partial \eta_2} (y, \eta) \right) = \bar b \cdot \nabla v^0_i -
\frac{y_2}{h }  \frac{\partial h}{\partial \eta_1}
 \frac{\partial v^1_i}{\partial \eta_2}
 \end{eqnarray*}
 for $i=1,2$ and
  \begin{eqnarray*}
   \displaystyle b_{\e} \cdot  \nabla Z^{\e} \toH  \bar b \cdot \nabla Z^0 -
\frac{y_2}{h  }  \frac{\partial h}{\partial \eta_1}
 \frac{\partial Z^1}{\partial \eta_2}  .
 \end{eqnarray*}

Similarly, let $\phi \in {\mathcal C}^{\infty} \bigl( \overline{ \Omega};
{\mathcal C}^{\infty}_{\sharp} (0,1) \bigr)$ and $\displaystyle \phi^{\e}
(y_1,y_2) = \phi \left( y_1, y_2, \frac{y_1}{\e} \right)$ for all $(y_1, y_2)
\in \Omega$.
We have
\begin{eqnarray*}
&& \displaystyle b_{\e} \cdot  \nabla \phi^{\e} = \frac{\partial
\phi^{\e}}{\partial y_1} (y)
 - \frac{y_2}{h \left( y_1, \frac{y_1}{\e} \right)} \left( \frac{\partial
h}{\partial y_1}  \left( y_1, \frac{y_1}{\e} \right) + \frac{1}{\e}
\frac{\partial h}{\partial \eta_1}  \left( y_1, \frac{y_1}{\e} \right) \right)
 \frac{\partial \phi^{\e} }{\partial y_2} (y) \\
 && \displaystyle = \frac{\partial \phi}{\partial y_1}  \left( y, \frac{y_1}{\e}
 \right) + \frac{1}{\e} \frac{\partial \phi}{\partial \eta_1}  \left( y,
\frac{y_1}{\e}  \right)  - \frac{y_2}{h \left( y_1, \frac{y_1}{\e} \right)}
\left( \frac{\partial h}{\partial y_1}  \left( y_1, \frac{y_1}{\e} \right) +
\frac{1}{\e} \frac{\partial h}{\partial \eta_1}  \left( y_1, \frac{y_1}{\e}
\right) \right)
 \frac{\partial \phi }{\partial y_2}  \left( y, \frac{y_1}{\e}  \right) .
 \end{eqnarray*}
Now let $\theta \in {\mathcal D}( 0,T)$, $\Theta = (\varphi, \psi) \in \tilde
V_{div} \times \tilde H^1$ and let $\Theta^{\e} = (\varphi^{\e}, \psi^{\e})$
with
\begin{eqnarray*}
 \varphi^{\e} (z) = \varphi \left( z_1, \frac{ z_2}{\e h^{\e} (z_1)} ,
\frac{z_1}{\e} \right) +  \frac{ z_2}{ h^{\e} (z_1)}  \frac{\partial h}{\partial
y_1} \left( z_1, \frac{z_1}{\e} \right) \varphi_1 \left( z_1, \frac{z_2}{\e
h^{\e}(z_1)}, \frac{z_1}{\e} \right) e_2
\end{eqnarray*}
 and $\displaystyle \psi^{\e} (z) = \psi \left( z_1, \frac{z_2}{\e h^{\e} (z_1)}
, \frac{z_1}{\e} \right)$ for all $(z_1, z_2) \in \Omega^{\e}$. We have
$\Theta^{\e}  \in \tilde V^{\e} \times \tilde H^{1, \e}$ and  from (\ref{fla})
\begin{eqnarray*}
&& \displaystyle \e \int_0^T  a \bigl( \bar v^{\e} (t) , \Theta^{\e} \bigr)
\theta (t) \, dt  \to
( \nu + \nu_r)  \int_0^T \int_{\Omega \times Y} \sum_{i=1}^2 \left(
h (\bar b \cdot \nabla v_i^0) (\bar b \cdot \nabla \varphi_i) + \frac{1}{h}
\frac{\partial v_i^0}{\partial y_2} \frac{\partial \varphi_i}{\partial y_2}
\right) \theta \, d\eta d y dt \\
&& \displaystyle + \alpha \int_0^T \int_{\Omega \times Y} \left(
h (\bar b \cdot \nabla Z^0) (\bar b \cdot \nabla \psi) + \frac{1}{h}
\frac{\partial Z^0}{\partial y_2}
 \frac{\partial \psi}{\partial y_2} \right) \theta \, d\eta d y dt \\
&& \displaystyle + (\nu + \nu_r) \int_0^T \int_{\Omega \times Y} \sum_{i=1}^2
\left(
\left( - y_2 \frac{\partial h}{\partial \eta_1} \frac{\partial v_i^1}{\partial
\eta_2} \right)
(\bar b \cdot \nabla \varphi_i) + \frac{1}{h} \frac{ \partial v_i^1}{\partial
\eta_2} \frac{\partial \varphi_i}{\partial y_2} \right) \theta \, d \eta  dy dt
\\
&& \displaystyle + \alpha \int_0^T \int_{\Omega \times Y}  \left( \left( - y_2
\frac{\partial h}{\partial \eta_1} \frac{\partial Z^1}{\partial \eta_2} \right)
(\bar b \cdot \nabla \psi) + \frac{1}{h} \frac{ \partial Z^1}{\partial \eta_2}
\frac{\partial \psi}{\partial y_2} \right) \theta \, d \eta  dy dt .
\end{eqnarray*}
But these last two integral terms vanish since $\varphi$, $\psi$ and $h$ do not
depend on $\eta_2$ and $v^1$ and $Z^1$ are $\eta_2$-periodic. Hence we obtain
\begin{eqnarray*}
\e \int_0^T   a \bigl( \bar v^{\e} (t) , \Theta^{\e} \bigr) \theta (t) \, dt
\to \int_0^T \bar a \bigl( \bar v^0 (t), \Theta \bigr) \theta (t) \, d t
\end{eqnarray*}
with $\bar v^0 = (v^0, Z^0)$ and
\begin{eqnarray*}
&& \displaystyle \bar a ( \bar v, \Theta) = ( \nu + \nu_r)  \int_{\Omega \times
(0,1) } \sum_{i=1}^2 \left(
h (\bar b \cdot \nabla v_i) (\bar b \cdot \nabla \varphi_i) + \frac{1}{h}
\frac{\partial v_i}{\partial y_2} \frac{\partial \varphi_i}{\partial y_2}
\right) \, d\eta_1 d y  \\
&& \displaystyle + \alpha  \int_{\Omega \times (0,1)} \left(
h (\bar b \cdot \nabla Z) (\bar b \cdot \nabla \psi) + \frac{1}{h}
\frac{\partial Z}{\partial y_2}
 \frac{\partial \psi}{\partial y_2} \right) \, d\eta_1 d y
 \end{eqnarray*}
 for all $\bar v = (v,Z) \in V_{div} \times H^1_{0 {\sharp}}$, for all $\Theta =
(\varphi, \psi) \in  V_{div} \times H^1_{0 {\sharp}}$.

 \smallskip

 From (\ref{trifb}) and the estimates (\ref{E3.13})-(\ref{E3.14})-(\ref{E3.16})
obtained at Proposition \ref{pro1} we get
 \begin{eqnarray*}
\e  \int_0^T  B \bigl(\bar v^{\e} (t), \bar v^{\e} (t) , \Theta^{\e} \bigr) \theta
(t) \, dt  = {\mathcal O}(\e) \to 0
  \end{eqnarray*}
  and similarly, from (\ref{Rot}) and (\ref{E3.13})-(\ref{E3.14})-(\ref{E3.16})
 \begin{eqnarray*}
\e  \displaystyle \int_0^T  {\mathcal R} \bigl(\bar v^{\e} (t), \bar v^{\e} (t) ,
\Theta^{\e} \bigr) \theta (t) \, dt
  = {\mathcal O}(\e) \to 0.
  \end{eqnarray*}


 Let us consider now the right hand side of equation (\ref{eqn:er2.14a}). We
recall that $\bar \xi^{\e} = (U^{\e} e_1, W^{\e} )$ with
 \begin{eqnarray*}
&& \displaystyle  U^{\e} (t,z) = U_0 (t) {\mathcal U} \left( h^{\e}(y_1) y_2
\right) =  \bar U \left( t, y_1, y_2, \frac{y_1}{\e} \right) \\
&& \displaystyle W^{\e}(t,z) = W_0 (t) {\mathcal W} \left( h^{\e}(y_1) y_2
\right) =  \bar W \left( t, y_1, y_2, \frac{y_1}{\e} \right)
\end{eqnarray*}
and $U_0$, $W_0$ belong to $H^1 (0,T)$, ${\mathcal U}$, ${\mathcal W}$ belong to
${\mathcal D} \bigl( (-\infty, h_{max}) \bigr)$.
Hence $\bar U$ and $\bar W$ belong to ${\mathcal C} \bigl( [0,T];  {\mathcal
C}^{1} (\overline{ \Omega}; {\mathcal C}^{1}_{\sharp} (0,1) \bigr) \bigr)$ and
with (\ref{For3.10})-(\ref{Rot1})-(\ref{Rott})-(\ref{Rotxi})
\begin{eqnarray*}
&& \displaystyle \e \int_0^T  a \bigl( \bar \xi^{\e} (t), \Theta^{\e} \bigr) \theta
(t) \, dt  \to \int_0^T \bar a \bigl( \bar \xi (t), \Theta) \theta (t) \, dt \\
&& \displaystyle \e \int_0^T  B \bigl( \bar \xi^{\e} (t), \bar v^{\e} (t) ,
\Theta^{\e} \bigr) \theta (t) \, dt = {\mathcal O} (\e) \to 0  \\
&& \displaystyle \e \int_0^T   B \bigl( \bar v^{\e} (t) , \bar \xi^{\e} (t),
\Theta^{\e} \bigr) \theta (t) \, dt = {\mathcal O} (\e) \to 0  \\
&& \displaystyle \e \int_0^T  {\mathcal R} \bigl( \bar \xi^{\e} (t), \Theta^{\e}
\bigr) \theta (t) \, dt = {\mathcal O} (\e) \to 0
\end{eqnarray*}
with $\bar \xi = ( \bar U e_1, \bar W)$.

Next, using (\ref{pression}) and reminding that $\varphi^{\e} \in \tilde
V^{\e}$:
\begin{eqnarray*}
&& \displaystyle \e \int_0^T \int_{\Omega^{\e}}  p^{\e} (t,z)  {\rm div}_z
\varphi^{\e} (z) \theta (t) \, dz dt
= \int_0^T \int_{\Omega} \e p^{\e} \left( (\e b_{\e} \cdot \nabla
\varphi_1^{\e}) + \frac{1}{h^{\e}} \frac{\partial \varphi_2^{\e}}{\partial y_2}
\right) h^{\e} \theta \, dy dt \\
&& \displaystyle =  \int_0^T \int_{\Omega} \e p^{\e}
 \left(
\e h \left( y_1, \frac{y_1}{\e} \right)  \frac{\partial \varphi_1}{\partial y_1}
\left( y, \frac{y}{\e} \right)
+ h \left( y_1, \frac{y_1}{\e} \right)  \frac{\partial \varphi_1}{\partial
\eta_1} \left( y, \frac{y}{\e} \right)  \right.\\
&& \displaystyle   - y_2 \left(
\e \frac{\partial h}{\partial y_1} \left( y_1, \frac{y_1}{\e} \right) +
\frac{\partial h}{\partial \eta_1} \left( y_1, \frac{y_1}{\e} \right)
\right)
\frac{\partial \varphi_1}{\partial y_2} \left( y, \frac{y}{\e} \right)
+  \frac{\partial \varphi_2}{\partial y_2} \left( y, \frac{y}{\e} \right) \\
&& \displaystyle  \left. + \e y_2  \frac{\partial h}{\partial y_1} \left( y_1,
\frac{y_1}{\e} \right)  \frac{\partial \varphi_1}{\partial y_2} \left( y,
\frac{y}{\e} \right) + \e  \frac{\partial h}{\partial y_1} \left( y_1,
\frac{y_1}{\e} \right)  \varphi_1 \left( y, \frac{y}{\e} \right)
\right) \theta (t) \, dy dt \\
&& \displaystyle = \int_0^T \int_{\Omega} \e^2 p^{\e}
 \left(
h \left( y_1, \frac{y_1}{\e} \right)  \frac{\partial \varphi_1}{\partial y_1}
\left( y, \frac{y}{\e} \right)
+  \frac{\partial h}{\partial y_1} \left( y_1, \frac{y_1}{\e} \right)  \varphi_1
\left( y, \frac{y}{\e} \right)
\right) \theta (t) \, dy dt \\
&& \displaystyle \to  \int_0^T \int_{\Omega \times (0,1)} p^0 \frac{ \partial ( h
\varphi_1)}{\partial y_1} \theta \, d \eta_1 dy dt
= - \int_0^T \int_{\Omega \times (0,1)} \frac{\partial p^0}{\partial y_1} h
\varphi_1 \theta \, d \eta_1 dy dt.
 \end{eqnarray*}
 Finally
 \begin{eqnarray*}
&& \displaystyle \e^2  \int_0^T  \frac{d}{dt} \bigl[ \bar v^{\e} , \Theta^{\e}
\bigr] (t) \theta (t) \, dt =
- \e^2 \int_0^T   \bigl[ \bar v^{\e} , \Theta^{\e} \bigr] (t) \theta' (t) \, dt =
{\mathcal O} (\e^2) \to 0 \\
&&\displaystyle - \e^2 \int_0^T \left[ \frac{\partial \bar \xi^{\e}}{\partial t},
\Theta^{\e} \right] (t) \theta (t) \, dt
= {\mathcal O} (\e^2) \to 0.
\end{eqnarray*}

By multiplying equation (\ref{eqn:er2.14a}) by $\e \theta (t)$, integrating over
$[0,T]$ and passing to the limit as $\e $ tends to zero we obtain
\begin{eqnarray*}
&& \displaystyle
\int_0^T \bar a \bigl( \bar v^0 (t), \Theta) \theta (t) \, dt - \int_0^T
\int_{\Omega \times (0,1)} \frac{\partial p^0}{\partial y_1} h \varphi_1 \theta
\, d \eta_1 dy dt \\
 && \displaystyle
= - \int_0^T \bar a \bigl( \bar \xi (t), \Theta) \theta (t) \, dt + \int_0^T
\int_{\Om \times (0,1)} (f \varphi + g \psi) h \theta \,  d \eta_1 dy dt
\end{eqnarray*}
for all $\Theta = (\varphi, \psi) \in \tilde V_{div} \times \tilde H^1$ and
$\theta \in {\mathcal D} (0,T)$. By density of $\tilde V_{div} \times \tilde
H^1$ into $V_{div} \times H^1_{0 {\sharp}}$ we get the announced result.
\end{proof}

 We may observe that the limit problem is totally decoupled with respect to the
velocity and micro-rotation fields. Furthermore the time variable appears as a
parameter in the limit problem. More precisely,
 for all $y_1 \in [0,L] $,  let $a_{y_1}$ be the bilinear symmetric form defined
on
 ${\mathcal F}$ by
 \begin{eqnarray*}
 a_{y_1} (w, \psi) = \int_Y \left(  h(y_1, \eta_1) ( \bar b \cdot \nabla w)(
y_2, \eta_1) (\bar b \cdot \nabla \psi)(y_2,  \eta_1) + \frac{1}{h(y_1, \eta_1)}
\frac{\partial w}{\partial y_2} (y_2, \eta_1)  \frac{\partial \psi}{\partial
y_2} (y_2, \eta_1) \right) \, d \eta_1 dy_2
 \end{eqnarray*}
 for all $(w, \psi) \in {\mathcal F}$.
  The limit velocity, pressure and micro-rotation fields are solution of the
problems $(P_{v^0,p^0})$ and $(P_{Z^0})$ given respectively by
 \begin{eqnarray*}
&& \displaystyle \hbox{\rm Find } v^0 \in  L^2( 0,T; V_{div}) \mbox{ and } p^0 \in
H^{-1} (0,T;  H^1_{\sharp} (0,L) ) \mbox{ such that } \\
&& \displaystyle \displaystyle \int_0^L p^0 (t, y_1) \left( \int_0^1
h ( y_1, \eta_1) \, d \eta_1 \right) \, dy_1 =0 \mbox{ a.e. } t \in [0,T] \mbox{ and } \\
&& \displaystyle  ( \nu + \nu_r) \int_0^L \sum_{i=1}^2 a_{y_1} (v^0_i,
\varphi_i) \, dy_1
-  \int_0^L \frac{\partial p^0}{\partial y_1} \left( \int_0^1 h(y_1, \cdot )
\varphi_1 \, d \eta_1 \right) \, d y_1  \\
&& \displaystyle =- (\nu + \nu_r)  \int_0^L  a_{y_1} \bigl(\bar U (t) ,
\varphi_1\bigr) \, dy_1
+ \int_0^L \left( \int_Y f(t, y_1, \cdot, \cdot) h( y_1, \cdot) \varphi \, d
\eta_1 d y_2 \right) d y_1 \\
 && \displaystyle  \forall \varphi \in V_{div} , \ {\rm a.e.} \  t \in [0,T]
\end{eqnarray*}
and
  \begin{eqnarray*}
&& \displaystyle \hbox{\rm Find } Z^0 \in L^2 (0,T; H^1_{0, {\sharp}} ) \mbox{ such
that } \\
&& \displaystyle \alpha  \int_0^L  a_{y_1} (Z^0, \psi) \, dy_1
 =-  \alpha \int_0^L  a_{y_1} \bigl(\bar W (t) , \psi \bigr) \, dy_1
+ \int_0^L \left( \int_Y g(t, y_1, \cdot, \cdot) h( y_1, \cdot) \psi \, d \eta_1
d y_2 \right) d y_1 \\
&& \displaystyle \forall \psi \in  H^1_{0, {\sharp}}, \ {\rm a.e.} \  t \in
[0,T].
\end{eqnarray*}


\begin{proposition} \label{prop5.1}
Under the assumptions  of theorem \ref{limit_pb}, the limit micro-rotation
field $Z^0$ is uniquely  given by
\begin{eqnarray*}
Z^0 (t, y_1, y_2, \eta_1) =  W_0 (t) z^1_{ y_1} (y_2, \eta_1) + z^2_{ t, y_1}
(y_2, \eta_1) \quad \hbox{\rm a.e. in } (0,T) \times \Omega \times (0,1)
\end{eqnarray*}
where $z^1_{y_1} \in H^1_{0, \sharp} $ and $z^2_{t, y_1} \in H^1_{0, \sharp} $
are the unique solutions of the following auxiliary  problems:
  \begin{eqnarray*}
   a_{y_1} (z^1_{ y_1}, \psi) = - a_{y_1} \bigl( {\mathcal W} (y_1, \cdot), \psi
\bigr)
\quad \forall \psi \in H^1_{0, \sharp}
\end{eqnarray*}
and
  \begin{eqnarray*}
  \alpha  a_{y_1} (z^2_{t, y_1}, \psi) =
 \int_Y g_{t, y_1} h( y_1, \cdot) \psi \, d \eta_1 d y_2
\quad \forall \psi \in H^1_{0, \sharp} .
\end{eqnarray*}
\end{proposition}

\begin{proof}
 It is clear that, for all $y_1 \in [0,L] $, the mapping  $a_{y_1}$  is
continuous on ${\mathcal F}$.
  Moreover
 \begin{eqnarray*}
  a_{y_1} (w,  w) \ge h_{min} \| \bar b \cdot \nabla w \|^2_{L^2(Y)} +
\frac{1}{h_{max}} \left\| \frac{\partial w}{\partial z_2} \right\|_{L^2(Y)}
  \end{eqnarray*}
  and
  \begin{eqnarray*}
&& \displaystyle   \| \bar b \cdot \nabla w \|^2_{L^2(Y)} = \left\|
\frac{\partial w}{\partial \eta_1} \right\|^2_{L^2(Y)} + \int_Y
\frac{y_2^2}{h^2(y_1, \eta_1)} \left( \frac{\partial h}{\partial \eta_1} (y_1,
\eta_1, t) \right)^2 \left( \frac{\partial w}{\partial y_2} \right)^2 \, d
\eta_1 dy_2 \\
&& \displaystyle  - 2 \int_Y \frac{y_2}{h(y_1, \eta_1)}  \frac{\partial
h}{\partial \eta_1} (y_1, \eta_1)   \frac{\partial w}{\partial y_2}
\frac{\partial w}{\partial \eta_1} \, d \eta_1 dy_2 \\
&& \displaystyle   \ge (1 - \lambda) \left\| \frac{\partial w}{\partial \eta_1}
\right\|^2_{L^2(Y)} + \left( 1- \frac{1}{\lambda} \right)
\int_Y \frac{y_2^2}{h^2(y_1, \eta_1)} \left( \frac{\partial h}{\partial \eta_1}
(y_1, \eta_1) \right)^2 \left( \frac{\partial w}{\partial y_2} \right)^2 \, d
\eta_1 dy_2  \quad \forall \lambda >0.
\end{eqnarray*}
But, recalling that $h \in {\mathcal C} \bigl( [0,L] \times [0,1] \bigr)$, there
exists $C>0$, independent of $y_1$, such that
\begin{eqnarray*}
\left|  \frac{y_2}{h(y_1, \eta_1)}  \frac{\partial h}{\partial \eta_1} (y_1,
\eta_1)  \right| \le C \quad \forall (y_1, y_2, \eta_1) \in [0,L] \times Y
\end{eqnarray*}
and, for all $\lambda \in (0,1)$
\begin{eqnarray} \label{coer_1}
a_{y_1} (w,w) \ge C_{1}(\lambda)\left\| \frac{\partial
w}{\partial
\eta_1} \right\|^2_{L^2(Y)} + C_{2}(\lambda)
\left\| \frac{\partial w}{\partial y_2} \right\|^2_{L^2(Y)},
  \end{eqnarray}
where $C_{1}(\lambda) =  h_{min} (1 - \lambda)$ and $C_{2}(\lambda)
=\left(\left(1-\frac{1}{\lambda}\right) C^2 h_{min} + \frac{1}{h_{max}}
\right)$. Then we may choose $\lambda$ such that
 \begin{eqnarray} \label{coer_2}
 \lambda \in \left( \frac{C^2 h_{max} h_{min}}{1+ C^2 h_{max} h_{min}}, 1
\right)
  \end{eqnarray}
  which shows that $a_{y_1}$ is coercive on $H^1_{0, \sharp}$,
  uniformly with respect to $y_1$.  Since $g \in {\mathcal C} \bigl( [0,T];
{\mathcal C} \bigl( \overline{\Om}; {\mathcal C}_{\sharp} (0,1) \bigr) \bigr)$
 the mapping $g_{t,y_1} = g(t, y_1, \cdot, \cdot) $ belongs to $L^2(Y)$ for all
$(t, y_1) \in [0,T] \times [0,L]$.
Then Lax-Milgram's theorem implies that, for all $(t, y_1)  \in  [0,T] \times
[0,L] $ the problems
  \begin{eqnarray*}
 && \displaystyle  \hbox{\rm Find } z^1_{y_1} \in H^1_{0, \sharp}  \mbox{ such that}
\\
&& \displaystyle a_{y_1} (z^1_{ y_1}, \psi) = - a_{y_1} \bigl( {\mathcal W}
(y_1, \cdot) , \psi \bigr)
\quad \forall \psi \in H^1_{0, \sharp}
\end{eqnarray*}
and
  \begin{eqnarray*}
 && \displaystyle  \hbox{\rm Find } z^2_{t, y_1} \in H^1_{0, \sharp} \mbox{ such that}
\\
&& \displaystyle \alpha  a_{y_1} (z^2_{t, y_1}, \psi) =
 \int_Y g_{t, y_1} h( y_1, \cdot) \psi \, d \eta_1 d y_2
\quad \forall \psi \in H^1_{0, \sharp}
\end{eqnarray*}
admit a unique solution. Furthermore, recalling that $W_0 \in H^1(0,T) \subset
{\mathcal C} \bigl( [0,T] \bigr)$ and $h \in {\mathcal C}^1 \bigl( [0,L] \times
[0,1] ; \br \bigr) $ with values in $[h_{min}, h_{max} ] \subset \br^+_*$, we
infer that the mapping $(t, y_1 ) \mapsto Z^0_{t, y_1} = W_0 z^1_{ y_1} + z^2_{
t, y_1}$ is continuous on $[0,T] \times [0,L]$
 with values in $H^1_{0, \sharp}$ and is $L$-periodic in $y_1$.

 Thus the mapping $Z^0: (t, y_1, y_2, \eta_1) \mapsto  Z^0_{t, y_1} (y_2,
\eta_1)$ belongs to $L^2 (0,T ; H^1_{ 0, {\sharp}} \bigr)$ and
solves the problem $(P_{Z_0})$. Indeed, let $\psi \in \tilde H^1$. Then $\psi (
y_1, \cdot, \cdot) \in  H^1_{\sharp} $ and we get
\begin{eqnarray*}
&& \displaystyle \alpha a_{y_1} \bigl(Z^0_{t, y_1}, \psi ( y_1, \cdot, \cdot)
\bigr) = -  \alpha a_{y_1} \bigl( \bar W(t, y_1, \cdot, \cdot)  , \psi ( y_1,
\cdot, \cdot) \bigr)  \\
&& \displaystyle + \int_Y g (t, y_1, \cdot, \cdot)  h( y_1, \cdot) \psi \, d
\eta_1 d y_2
\quad \forall y_1 \in [0,L] .
\end{eqnarray*}
Both sides of this equality are continuous on $[0,L]$, hence we may integrate
with respect to $y_1$ and
\begin{eqnarray*}
\int_0^L  a_{y_1} (Z^0_{t, y_1}, \psi) \, dy_1 = - \int_0^L a_{y_1} \bigl( \bar
W , \psi \bigr) \, dy_1
+ \int_0^L \left( \int_Y g_{t, y_1} h( y_1, \cdot) \psi \, d \eta_1 d y_2
\right)  d y_1
\quad \forall \psi \in \tilde H^1.
\end{eqnarray*}
It follows that
\begin{eqnarray*}
&& \displaystyle \int_0^L  a_{y_1} (Z^0 (t), \psi) \, dy_1 = - \int_0^L a_{y_1}
\bigl( \bar W (t) , \psi \bigr) \, dy_1
+ \int_0^L \left( \int_Y g (t)  h( y_1, \cdot) \psi \, d \eta_1 d y_2 \right)  d
y_1 \\
&& \displaystyle \forall \psi \in \tilde H^1, \ \hbox{\rm a.e.} \ t \in [0,T]
\end{eqnarray*}
 and the density of $\tilde H^1$ into $H^1_{ 0, {\sharp}} $ allows us to
conclude the existence part of the proof.
Then we observe that  the uniqueness is a immediate consequence of the uniform
coercivity of $a_{y_1}$ with respect to $y_1$.
\end{proof}

Now, for all $y_1 \in [0,L] $,  let
\begin{eqnarray*}
\tilde V_{y_1} =  \left\{ \varphi \in \bigl( {\mathcal C}^{\infty} \bigl([0,1];
{\mathcal C}^{\infty}_{\sharp} (0,1) \bigr) \bigr)^2; \ \varphi (0, \cdot)= 0 \
\hbox{\rm on } (0,1), \ - \varphi_1 (1, \cdot) \frac{\partial h}{\partial y_1}
(y_1, \cdot) + \varphi_2 (1, \cdot) =0 \ \hbox{\rm on } (0,1) \right\},
\end{eqnarray*}
\begin{eqnarray*}
\tilde V_{y_1, div} = \left\{ \varphi \in \tilde V_{y_1}; \ h(y_1, \cdot)
\frac{\partial \varphi_1}{\partial \eta_1} - y_2 \frac{\partial h  }{\partial
\eta_1} (y_1, \cdot) \frac{\partial \varphi_1}{\partial y_2} + \frac{\partial
\varphi_2}{\partial y_2} =0 \ \hbox{\rm in } Y \right\}
\end{eqnarray*}
and
 \begin{eqnarray*}
V_{y_1, div} = \hbox{\rm closure of } \tilde V_{y_1, div} \mbox{ in } {\mathcal F}^{2}.
\end{eqnarray*}
Let $\displaystyle \bar a_{y_1} (w, \varphi) = (\nu + \nu_r) \sum_{i=1}^2
a_{y_1} (w_i, \varphi_i)$ for all $(w, \varphi) \in V_{y_1, div}^2$.
With Poincar\'e's inequality we know that $w \mapsto \| \nabla w\|_{L^2 (Y)}$
defines a norm on $V_{y_1, div}$ which is equivalent to the $H^1$-norm.
Furthermore, with (\ref{coer_1})-(\ref{coer_2}), we may infer that $\bar
a_{y_1}$ is coercive on $V_{y_1, div}$ for all $y_1 \in [0,L] $, uniformly with
respect to $y_1$. It follows that we can define $w_{y_1}^1 \in V_{y_1, div}$,
$w_{y_1}^2 \in V_{y_1, div}$ and $w_{t,y_1}^3 \in V_{y_1, div}$ as the unique
solutions of
\begin{eqnarray*}
 \bar a_{y_1} (w^1_{y_1}, \varphi) = - \int_Y h (y_1, \cdot ) \varphi_1 \, d
\eta_1 d y_2 \quad \forall \varphi \in V_{y_1, div},
\end{eqnarray*}
\begin{eqnarray*}
\bar a_{y_1} (w^2_{y_1}, \varphi)  = - (\nu + \nu_r)  a_{y_1} \bigl({\mathcal U}
( y_1, \cdot ) , \varphi_1 \bigr) \quad \forall \varphi \in V_{y_1, div}
\end{eqnarray*}
and
\begin{eqnarray*}
\bar a_{y_1} (w^3_{t,y_1}, \varphi)  = \int_Y  f_{t,y_1} h (y_1, \cdot )
\varphi_1 \, d \eta_1 d y_2 \quad \forall \varphi \in V_{y_1, div}
\end{eqnarray*}
with $f_{t,y_1} = f(t, y_1, \cdot, \cdot)$ for all $(t,y_1) \in [0,T] \times
[0,L]$.

Then we have

\begin{proposition}\label{prop5.2}
Under the assumptions  of theorem \ref{limit_pb},  the limit velocity $v^0$ is
uniquely given by
\begin{eqnarray*}
v^0 (t, y_1, y_2, \eta_1) = \frac{\partial p^0}{\partial y_1} (t, y_1) w^1_{y_1}
 (y_2, \eta_1) +  U_0 (t) w^2_{ y_1} (y_2, \eta_1) + w^3_{t,y_1} (y_2, \eta_1)
\quad \hbox{\rm a.e. in  } (0,T) \times \Omega\times (0,1).
\end{eqnarray*}
Furthermore, for almost every  $t \in [0,T]$, the limit pressure $p^0 (t, \cdot)
$ is the unique solution in $H^1_{\sharp} (0,L)_{| \br}$   of the following
homogenized Reynolds equation
\begin{eqnarray*}
\int_0^L \frac{\partial p^0}{\partial y_1} \frac{\partial \psi}{\partial y_1}
\bar a (w^1_{y_1} , w^1_{y_1} ) \, dy_1 = - \int_0^L U_0 (t) \frac{\partial
\psi}{\partial y_1} \bar a \bigl(w^1_{y_1} ,  w^2_{ y_1} \bigr) \, dy_1
- \int_0^L  \frac{\partial \psi}{\partial y_1} \bar a \bigl(w^1_{y_1} ,  w^3_{
t, y_1} \bigr) \, dy_1
 \quad \forall \psi \in H^1_{\sharp} (0,L)
\end{eqnarray*}
satisfying $\displaystyle \int_0^L p^0  \left( \int_0^1 h( \cdot , \eta_1) \, d
\eta_1 \right) \, dy_1 =0$.
\end{proposition}

\begin{proof} The first part of the result is obtained by using the same kind of
arguments as in Proposition \ref{prop5.1}.

Let $\theta \in {\mathcal D} (0,T)$, $\psi \in {\mathcal C}^{\infty}_{\sharp}
\bigl( [0,L] \bigr)$ and $\psi^{\e} (z) = \psi(z_1)$ for all $z = (z_1, z_2) \in
\Omega^{\e}$. Recalling that ${\rm div}_z v^{\e} =0$ in $\Omega^{\e}$ and using
the boundary conditions
(\ref{eqn:er2.8a})-(\ref{eqn:er2.11})-(\ref{eqn:er2.11b}) we get
\begin{eqnarray*}
&& \displaystyle 0 = \frac{1}{\e} \int_0^T \int_{\Omega^{\e}} \left(
\frac{\partial v^{\e}_1 }{\partial z_1}  (t,z)+ \frac{\partial v^{\e}_2
}{\partial z_2} (t,z) \right) \psi^{\e} (z) \theta (t) \, dz dt \\
&& \displaystyle 0 = -  \frac{1}{\e} \int_0^T \int_{\Omega^{\e}} v_1^{\e} (t,z)
\frac{\partial \psi^{\e} }{\partial z_1}  (z) \theta (t) \, dz dt  = -
\int_0^T \int_{\Omega} v_1^{\e} (t,y) (b_{\e} \cdot \nabla \psi^{\e} ) (y)
h^{\e} (y) \theta (t) \, dy dt \\
&& \displaystyle = - \int_0^T \int_{\Omega} v_1^{\e} (t,y)  \frac{\partial \psi
}{\partial y_1} (y_1) h \left( y_1, \frac{y_1}{\e} \right) \theta (t) \, dy dt.
\end{eqnarray*}
By passing to the limit as $\e$ tends to zero we get
\begin{eqnarray*}
0 = \int_0^T \int_{\Omega \times (0,1)} v_1^0 (t, y, \eta_1)  \frac{\partial
\psi }{\partial y_1} (y_1) h (y_1, \eta_1)  \theta (t) \, d \eta_1 dy dt.
\end{eqnarray*}
It follows that
 \begin{eqnarray*}
&& \displaystyle  \int_0^L \frac{\partial p^0}{\partial y_1} \frac{\partial
\psi}{\partial y_1}\left(  \int_Y w_{y_1, 1}^1 h (y_1, \cdot) \, d \eta_1 d y_2
\right)  \, dy_1 +  \int_0^L U_0 (t) \frac{\partial \psi}{\partial y_1} \left(
\int_Y w_{y_1, 1}^2 h (y_1, \cdot) \, d \eta_1 d y_2 \right)  \, dy_1 \\
&& \displaystyle + \int_0^L  \frac{\partial \psi}{\partial y_1} \left(  \int_Y
w_{t, y_1, 1}^3 h (y_1, \cdot) \, d \eta_1 d y_2 \right)  \, dy_1
=0 \quad {\rm a.e.} \ t \in [0,T].
\end{eqnarray*}
But
\begin{eqnarray*}
&& \displaystyle  \int_Y w_{y_1, 1}^1 h (y_1, \cdot) \, d \eta_1 d y_2  = - \bar
a_{y_1} (w_{y_1}^1, w_{y_1}^1) , \\
&& \displaystyle \int_Y w_{y_1, 1}^2 h (y_1, \cdot) \, d \eta_1 d y_2  = - \bar
a_{y_1} (w_{y_1}^1, w_{y_1}^2), \\
&& \displaystyle \int_Y w_{t,y_1, 1}^3 h (y_1, \cdot) \, d \eta_1 d y_2  = -
\bar a_{y_1} (w_{y_1}^1, w_{y_1}^3)
\end{eqnarray*}
and by density of $ {\mathcal C}^{\infty}_{\sharp} \bigl( [0,L] \bigr)$ in
$H^1_{\sharp} (0,L)$ we get
\begin{eqnarray*}
&& \displaystyle \int_0^L \frac{\partial p^0}{\partial y_1} \frac{\partial
\psi}{\partial y_1} \bar a (w^1_{y_1} , w^1_{y_1} ) \, dy_1 = - \int_0^L U_0 (t)
\frac{\partial \psi}{\partial y_1} \bar a \bigl(w^1_{y_1} ,  w^2_{ y_1} \bigr)
\, dy_1 \\
&& \displaystyle -  \int_0^L  \frac{\partial \psi}{\partial y_1}   \bar a
\bigl(w^1_{y_1} ,  w^3_{ t, y_1} \bigr) \, dy_1 \quad \forall \psi \in
H^1_{\sharp} (0,L) ,  \ {\rm a.e.} \ t \in [0,T].
\end{eqnarray*}
We can check that this Reynolds problem admits a unique solution in
$H^1_{\sharp} (0,L)_{| \br}$. Indeed, let
$\displaystyle   \varphi_{y_1} (y_2, \eta_1) = \left( \frac{ -y_2 +
y_2^2}{h(y_1, \eta_1)}, \frac{\partial h}{\partial \eta_1} (y_1, \eta_1) \frac{
y_2^2 (y_2 -1)}{h(y_1, \eta_1)} \right)$ for all $(y_2, \eta_1) \in Y$, for all
$y_1 \in [0,L]$. Then we obtain $ \varphi_{y_1} \in  V_{y_1, div}$ and
\begin{eqnarray*}
\bar a_{y_1} (w^1_{y_1},  \varphi_{y_1}) = - \int_Y h (y_1, \eta_1)
\varphi_{y_1, 1} (y_2 , \eta_1) \, d \eta_1 d y_2 = \frac{1}{6}.
\end{eqnarray*}
Since $\bar a_{y_1}$ defines an inner product on $V_{y_1, div}$,
we have
\begin{eqnarray*}
\frac{1}{6 } = \bar a_{y_1} (w^1_{y_1},  \varphi_{y_1}) \le \bar a_{y_1}
(w^1_{y_1}, w^1_{y_1} )^{1/2} \bar a_{y_1} ( \varphi_{y_1},
\varphi_{y_1})^{1/2} .
\end{eqnarray*}
But the mapping $y_1 \mapsto  \bar a_{y_1} ( \varphi_{y_1},  \varphi_{y_1})$ is
continuous on $[0,L]$ and does not vanish since $ \varphi_{y_1} \not \equiv 0$.
It follows that   there exists $\alpha >0$ such that $ \bar a_{y_1} (
\varphi_{y_1},  \varphi_{y_1}) \ge \alpha$  for all $y_1 \in [0,L]$ and $ \bar a
(w^1_{y_1} , w^1_{y_1} ) \ge \frac{1}{ 36 \alpha}$ for all $y_1 \in [0,L]$.

We may observe also that the mapping \newline
$\displaystyle \psi \mapsto - \int_0^L U_0 (t) \frac{\partial \psi}{\partial
y_1} \bar a \bigl(w^1_{y_1} ,  w^2_{ y_1} \bigr) \, dy_1 -   \int_0^L
\frac{\partial \psi}{\partial y_1}   \bar a \bigl(w^1_{y_1} ,  w^3_{ t, y_1}
\bigr) \, dy_1 $ is linear and continuous on $H^1_{\sharp} (0,L)$ for  every $t
\in [0,T]$ and the mapping
$\displaystyle (p, \psi) \mapsto \int_0^L \frac{\partial p}{\partial y_1}
\frac{\partial \psi}{\partial y_1} \bar a (w^1_{y_1} , w^1_{y_1} ) \, dy_1$ is
bilinear, symmetric, continuous  and coercive on $H^1_{\sharp} (0,L)_{|\br}$. We
can apply Lax-Milgram's theorem to conclude the proof of Proposition \ref{prop5.2}.
\end{proof}

As a consequence of the uniqueness of $p^0$, we can state the next result:

\begin{theorem}
The whole sequences $( \e^2 p^{\e})_{\e >0}$, $( v^{\e})_{\e >0}$ and
$(Z^{\e})_{\e >0}$ satisfy the following convergence:
\begin{eqnarray*}
&& \e p^{\e} \toH p^0 \\
&& v^{\e} \toH v^0 \\
&& Z^{\e} \toH Z^0.
\end{eqnarray*}
\end{theorem}

\section{Concluding remarks}\label{open-Pbs}

A possible generalization of this study consists in considering a domain $\Om^{\e}$ where both the upper and lower boundary are oscillating. More precisely, let us assume  that
\begin{eqnarray*}
 \Om^{\e} =\bigl\{ (z_{1} , z_{2}) \in \br^{2}; \quad 0 < z_{1}< L,  \quad
-\e \beta(z_{1}) h^{\e}(z_{1})< z_{2} < \e h^{\e}(z_{1}) \bigr\}
\end{eqnarray*}
where $\beta$ belongs to $ {\cal C}^{\infty}([0 , L];  \br^+)$ and is $L$-periodic in $z_1$ (with $\beta \equiv 0$ we recognize the case presented in the previous sections). Now we should denote by $\Gamma_0^{\e}$ the lower boundary of $\Om^{\e}$ and we can choose the functions ${\cal U}$ and ${\cal W}$ (see Lemma \ref{lUW}) such that ${\cal U}$ and ${\cal W}$ belong to ${\cal C}^{\infty}(\br, \br)$ with ${\cal U}(\sigma) = {\cal W} (\sigma) = 1$ for all $\sigma \le 0$ and ${\rm Supp} ({\cal U}) \subset (- \infty, h_m)$, ${\rm Supp} ({\cal W}) \subset (- \infty, h_m)$. Then we define again
$$
 U^{\e}(t ,z_{2})= {\cal U}^{\e}(z_{2})U_{0}(t)
 = {\cal U}({z_{2}\over \e})U_{0}(t) ,
\quad
 W^{\e}(t ,z_{2})= {\cal W}^{\e}(z_{2})W_{0}(t)
= {\cal W}({z_{2}\over \e})W_{0}(t)$$
and we get the same variational problem $(P^{\e})$. It follows that the existence and uniqueness result given at Theorem
\ref{th2.1} is still valid.
Furthermore, we can use the same scalings (see (\ref{2.2a}) and (\ref{2.2b})) which transforms the domain $\Om^{\e}$ into
\begin{eqnarray*}
\Om= \bigl\{ (y_1, y_2) \in \br^2;  \quad 0 < y_1 <L, \quad - \beta (y_1) < y_2 <1 \bigr\}
\end{eqnarray*}
and by reproducing the same computations, we obtain the same a priori estimates as in Proposition \ref{pro1} and Proposition \ref{prop2}.

Finally we may apply once again the two-scale convergence technique to pass to the limit as $\e$ tends to zero. We obtain the same convergence properties for the velocity and the micro-rotaion field as in Proposition \ref{prop4.1} and Proposition \ref{prop4.2} with $\Gamma_0 = \bigl\{ \bigl(y_1, - \beta(y_1) \bigr); \ 0< y_1 < L \bigr\}$. For the convergence of the pressure, we follow the same arguments as in Proposition \ref{prop4.3} with a natural modification of the test-function $\varphi^{\e}$ introduced at formula (\ref{eq:rajout1}) which may be chosen now  as
\begin{eqnarray*}
\varphi^{\e} (y) = \frac{\varphi(y_1)}{h\left(y_1, \frac{y_1}{\e} \right)} \left( \bigl(y_2 + \beta(y_1) \bigr) e_1 + \e y_2  \bigl(y_2 + \beta(y_1) \bigr) \left(  \frac{\partial h}{\partial y_1} \left(y_1, \frac{y_1}{\e} \right) + \frac{1}{\e}  \frac{\partial h}{\partial \eta_1} \left(y_1, \frac{y_1}{\e} \right) \right) e_2 \right)
\end{eqnarray*}
for all $(y_1, y_2) \in \Om$, which leads to
\begin{eqnarray*}
 \left|  \int_0^T \int_0^L p^{0} (t,y_1)  \frac{\partial }{\partial y_1}\bigl( \frac{1}{2} (1 + \beta)^2 \varphi \bigr) (y_1)   \theta (t) \, dy dt \right|
\le   C \| \varphi \theta \|_{L^2 ((0,T) \times (0,L))}.
\end{eqnarray*}
Then we may conclude by considering any $\phi \in {\mathcal C}^{\infty}_{\sharp} (0,L) $ and letting $\varphi = \frac{2 \phi}{(1+ \beta)^2}$.

Hence the limit problem remains the same as in Theorem \ref{limit_pb}: $Z^0$ and $v^0$ can be decomposed by using the same auxiliary problems and $p^0$ is the unique solution of the same Reynolds equation, with obvious adaptations in the definition of $a_{y_1}$ and $\bar a_{y_1}$, i.e. for all $y_1 \in [0,L]$:
 \begin{eqnarray*}
 a_{y_1} (w, \psi) = \int_{ - \beta(y_1)}^1 \int_0^1 \Bigl( && h(y_1, \eta_1) ( \bar b \cdot \nabla w)( y_2, \eta_1) (\bar b \cdot \nabla \psi)(y_2,  \eta_1)  \\
&& + \frac{1}{h(y_1, \eta_1)} \frac{\partial w}{\partial y_2} (y_2, \eta_1)  \frac{\partial \psi}{\partial y_2} (y_2, \eta_1) \Bigr) \, d \eta_1 dy_2
 \end{eqnarray*}
 for all $\displaystyle (w, \psi) \in {\mathcal F}_{y_1} = \Bigl\{ v \in L^2 \bigl( (- \beta(y_1), 1 ); H^1_{\sharp} (0,1) \bigr) ; \frac{\partial v}{\partial y_2} \in L^2 \bigl( (- \beta(y_1), 1 ) \times (0,1) \bigr)  \Bigr\}$ and
$\displaystyle \bar a_{y_1} (w, \varphi) = (\nu + \nu_r) \sum_{i=1}^2 a_{y_1} (w_i, \varphi_i)$ for all $(w, \varphi) \in V_{y_1, div}^2$ where $V_{y_1, div}$ is the closure of $\tilde V_{y_1, div}$ in ${\mathcal F}_{y_1}^2$ and
\begin{eqnarray*}
\tilde V_{y_1, div} =  \Bigl\{&& \varphi \in \bigl( {\mathcal C}^{\infty} \bigl([- \beta(y_1),1]; {\mathcal C}^{\infty}_{\sharp} (0,1) \bigr) \bigr)^2; \ \varphi (- \beta(y_1), \cdot)= 0 \  \hbox{\rm on $(0,1)$,} \\
&&  - \varphi_1 (1, \cdot) \frac{\partial h}{\partial y_1} (y_1, \cdot) + \varphi_2 (1, \cdot) =0 \ \hbox{\rm on $ (0,1)$,} \\
&&
h(y_1, \cdot)  \frac{\partial \varphi_1}{\partial \eta_1} - y_2 \frac{\partial h  }{\partial \eta_1} (y_1, \cdot) \frac{\partial \varphi_1}{\partial y_2} + \frac{\partial \varphi_2}{\partial y_2} =0 \ \hbox{\rm in $(- \beta(y_1), 1) \times (0,1)$}
\Bigr\}.
\end{eqnarray*}

\bigskip
\noindent{\bf Acknowledgements:}
We would like to thank the anonymous referees for their constructive comments.

\end{document}